\documentclass[11pt]{article}
\usepackage[T1]{fontenc}
\usepackage[utf8]{inputenc}
\usepackage[letterpaper,margin=1in]{geometry}
\usepackage{amsmath}

\usepackage{amssymb,amsthm}
\usepackage{bm}
\usepackage{graphicx}
\usepackage{xcolor}
\usepackage{booktabs}
\usepackage{threeparttable}
\usepackage{array}
\usepackage{tabularx}
\usepackage{enumitem}
\usepackage{placeins}
\usepackage{tikz}
\usetikzlibrary{positioning, arrows.meta, decorations.pathreplacing, calc, fit, backgrounds, shapes.geometric}
\graphicspath{{icons/}{figures/}}
\usepackage[authoryear,round]{natbib}
\usepackage{url}
\usepackage{hyperref}
\hypersetup{
  colorlinks=true,
  linkcolor=blue!70!black,
  citecolor=blue!70!black,
  urlcolor=blue!70!black
}
\usepackage[nameinlink,noabbrev]{cleveref}
\newtheorem{theorem}{Theorem}
\newtheorem{lemma}[theorem]{Lemma}
\newtheorem{proposition}[theorem]{Proposition}

\theoremstyle{definition}
\newtheorem{definition}{Definition}

\newtheorem{example}{Example}
\theoremstyle{remark}
\newtheorem{remark}{Remark}
\newtheorem{appalemma}{Lemma}

\newtheorem{appaproposition}[appalemma]{Proposition}
\newtheorem{appacorollary}[appalemma]{Corollary}
\newtheorem{appatheorem}{Theorem}

\newtheorem{appblemma}{Lemma}

\newtheorem{appbtheorem}{Theorem}

\newtheorem{appbproposition}{Proposition}

\newtheorem{appbcorollary}{Corollary}

\theoremstyle{definition}
\newtheorem{appbexample}{Example}

\theoremstyle{remark}
\newcommand{\StartAppendixBTheoremNumbering}{%
  \setcounter{appblemma}{0}%
  \setcounter{appbtheorem}{0}%
  \setcounter{appbproposition}{0}%
  \setcounter{appbcorollary}{0}%
  \setcounter{appbexample}{0}%
}
\renewenvironment{proof}[1][Proof]{\par\noindent{\itshape #1.}\ }{\hfill\(\square\)\par}
\newif\ifshowcomments
\showcommentsfalse
\ifshowcomments
\newcommand{\jd}[1]{{\color{orange}[JD: #1]}}
\newcommand{\arc}[1]{{\color{blue}[ARC: #1]}}
\else
\newcommand{\jd}[1]{}
\newcommand{\arc}[1]{}
\fi
\newcommand{\inputterm}{input}
\newcommand{\Inputterm}{Input}
\newcommand{\inputctx}{input context}

\newcommand{\inputlen}{input length}
\newcommand{\inputlens}{input lengths}
\newcommand{\Inputlen}{Input length}

\newcommand{\inputlenratios}{input-length ratios}
\newcommand{\inputdominated}{input-dominated}
\newcommand{\commoninput}{common-input}
\newcommand{\finiteinput}{finite-input}
\newcommand{\heteroinput}{heterogeneous-input}
\title{\Large Service-Induced Congestion in Memory-Constrained LLM Serving}
\author{%
Ruicheng Ao\thanks{Institute for Data, Systems, and Society, Massachusetts Institute of Technology, Cambridge, MA 02139. Email: \texttt{aorc@mit.edu}.}
\and Jing Dong\thanks{Columbia Business School, Columbia University, New York, NY 10027. Email: \texttt{jing.dong@gsb.columbia.edu}.}
\and Gan Luo\thanks{School of Mathematical Sciences, Peking University, Beijing, China. Email: \texttt{luogan@stu.pku.edu.cn}.}
\and David Simchi-Levi\thanks{Institute for Data, Systems, and Society, Massachusetts Institute of Technology, Cambridge, MA 02139. Email: \texttt{dslevi@mit.edu}.}
}
\date{\today}
\begin{document}
\maketitle
\begin{abstract}
\noindent
In large language model (LLM) serving, each request accumulates persistent graphics processing unit (GPU) memory during service as its key--value cache grows with every generated token. Under high concurrency, aggregate memory usage therefore increases endogenously over time: the service process itself creates future capacity pressure. When memory capacity is exceeded, systems must evict active requests, i.e., discarding their cached state and restarting them later, which wastes computation and reduces throughput. This progressive consumption of resources gives rise to a new form of service-induced congestion. In this work, we develop a discrete-time dynamical model of memory-constrained LLM inference that captures the interaction among admission, memory growth, and eviction under continuous batching. In the saturated-input regime, the system admits both eviction-free fixed points and limit cycles with evictions. For homogeneous workloads, we show that the eviction-free equilibrium is unstable under standard batching dynamics and that, except for a Lebesgue-measure-zero exact-capture set, the system converges to a unique worst-case limit cycle that is asymptotically stable outside this exceptional set, where throughput losses can be as large as 50\%. For heterogeneous workloads, we prove a sharp stability criterion in the two-class \commoninput{} setting and explain how the survival-polynomial mechanism generalizes to multiple classes and \heteroinput{} lengths. Under an \inputdominated{} scaling regime, coprime decoding lengths stabilize the eviction-free equilibrium, while non-coprime lengths create synchronized modes that drive instability. This result characterizes when workload heterogeneity desynchronizes completion events and helps stabilize the system around the eviction-free equilibrium. More broadly, we identify service-induced congestion as a structural instability mechanism in memory-constrained systems and offer scheduling design principles for sustaining high throughput.
\end{abstract}
\section{Introduction}
\label{sec:introduction}

Large language models (LLMs) now serve as shared online services that must handle high request volumes under strict latency and reliability requirements~\citep{orca2022,vllm2023}. These models generate text using autoregressive decoding. Given an \inputctx{}, the model produces text one token at a time; each newly generated token is appended to the context and used to condition subsequent token generation. Internally, this process requires storing intermediate attention activations from previous tokens, known as the key–value (KV) cache. In particular, at each transformer layer, every token produces a key and a value vector, and the collection of these vectors for all prior tokens is cached and reused during decoding. This avoids recomputing attention over the full prefix at each step, reducing the per-token computational complexity from quadratic to linear in the sequence length. However, the KV cache grows linearly with the number of generated tokens, resulting in a progressively increasing graphics processing unit (GPU) memory footprint during inference.

As such, inference for a single request is not a one-shot computation but a sequence of decoding steps that create, update, and retain request-specific state. This behavior contrasts with traditional stateless inference workloads such as image classification or fixed-length embedding generation, where each request is processed in a single forward pass. For example, in an image classification service, an input image is mapped to a label through one feedforward evaluation of the neural network, after which all intermediate activations can be discarded. In LLM serving, by contrast, a request accumulates internal state across many decoding steps and continuously occupies GPU memory throughout its lifetime. Requests therefore consume not only compute capacity but also persistent memory capacity, making memory a first-class scheduling constraint.

Another important characteristic of LLM inference is that modern serving systems do not process requests one at a time. A single decoding step uses only a small fraction of a GPU’s available compute. To achieve high throughput, systems advance many requests concurrently, sharing compute across them. This high concurrency keeps the GPU well utilized, but it also means that the memory states of many requests coexist in the GPU memory.

In many practical serving regimes, GPU memory becomes the binding constraint, especially under high concurrency and long-context workloads. In modern agentic and retrieval-augmented deployments, the \inputctx{} may include conversation history, retrieved documents, uploaded files, tool outputs, and other application state, so the initial context itself can be large and heterogeneous across requests~\citep{lewis2020rag,yao2023react,schick2023toolformer,sglang2024}. Production systems therefore treat memory pressure and memory-induced preemption as explicit scheduling concerns: vLLM implements swapping and recomputation for evicted KV caches~\citep{vllm2023}, SGLang employs least-recently-used (LRU) eviction for its radix-tree cache~\citep{sglang2024}, and FastServe~\citep{wu2023fastserve} and TensorRT-LLM~\citep{tensorrt2023} schedule preemptively to manage memory overflow. Unlike compute resources, which can be time-shared across requests, GPU memory cannot be safely overcommitted. When aggregate memory usage exceeds capacity, the system must evict in-progress requests, discarding previously accumulated computation, which reduces effective throughput.

This memory-driven behavior differs qualitatively from congestion in classical queueing models. In most classical formulations, a job occupies a fixed amount of capacity throughout its lifetime, whether that capacity represents a server, bandwidth, or memory. 
LLM requests, in contrast, consume increasing capacity: even when a request's total decoding length is known at admission, its memory footprint grows monotonically during service. A request that fits in memory at admission may become infeasible solely because it continues to run. In this context, eviction is not merely an implementation artifact or rare failure mode. Under sustained load, it becomes unavoidable because memory state grows during service. Service itself therefore creates future capacity pressure, giving rise to a new form of congestion.

This structural difference fundamentally changes the role of admission decisions, creating an inherent tension between protecting the system from future memory overflow and maintaining high utilization.
A naive myopic admission policy checks only current KV-cache usage: it admits a new request whenever the instantaneous footprint fits within available memory, without reserving capacity for future growth. Such locally feasible decisions can create synchronized memory peaks later, triggering eviction cascades that discard accumulated computation and collapse throughput.
At the other extreme, overly conservative policies that reserve memory for worst-case future growth can also perform poorly. Excessive reservation suppresses concurrency and leaves substantial GPU capacity idle, a phenomenon widely observed in practice 
\citep{vllm2023,distserve2024}. 
The central challenge is thus how to balance memory protection against utilization when capacity consumption evolves during service.

To analyze these dynamics formally, we introduce a discrete-time queueing system in which jobs have known total sizes but resource footprints grow during service. 
The resulting model departs from classical queues with fixed per-job capacity consumption and reveals phenomena absent from traditional settings, including multiple equilibria and eviction-driven throughput loss.
We first analyze the homogeneous workload case, and then study the heterogeneous setting. Near homogeneous workloads arise naturally in schema-constrained LLM applications, such as binary decisions, tool-routing outputs, or fixed-format responses, where requests from the same application have nearly identical decoding budgets~\citep{sglang2024,schick2023toolformer,openai2024structuredoutputs}. For example, an evaluation or moderation pipeline may ask the model to return a fixed schema with a small enumerated field, such as \texttt{\{answer: yes\}} or \texttt{\{answer: no\}}, creating many requests with nearly identical generation lengths. We show that homogeneous workloads are prone to synchronized memory growth, which can lead to memory peaks and throughput collapse, whereas some form of heterogeneity in decoding lengths can help sustain throughput by desynchronizing memory growth across requests. Our analysis also provides an analytical basis for studying congestion and efficiency in stateful artificial intelligence (AI) services where resource consumption evolves during service.

\subsection{Contributions}

We formalize LLM serving with KV-cache memory as a discrete-time dynamical system that captures the interaction among admission, memory growth, and eviction (Section~\ref{sec:model}). The model isolates a congestion mechanism absent from classical queueing dynamics: resource consumption increases progressively during service. 
We characterize the long-run behavior of the system in a saturated-input regime, where admission is constrained only by memory feasibility. This regime isolates the intrinsic congestion effects induced by progressive memory growth, independent of stochastic arrival fluctuations, and reveals the fundamental stability limit imposed by the serving architecture itself.

For homogeneous job sizes, we provide a general stability characterization of the eviction cascade (Theorem~\ref{thm:fcfs}, Section~\ref{sec:single_class}). For jobs with decode length $l_1$, the dynamics admit an eviction-free fixed point and eviction-level limit cycles, including a maximal-eviction cycle. We show that under the baseline continuous batching policy, the eviction-free fixed point is unstable and intermediate eviction-level cycles are also unstable. More generally, after a support pattern is fixed, the only obstruction to further collapse is the balanced case in which the normalized masses of all live positions are equal; any imbalance forces another support loss. The maximal-eviction limit cycle is the unique cycle that is asymptotically stable outside the Lebesgue-measure-zero exact-capture set on the memory boundary, and its basin contains all memory-boundary initial states except for that set: within each fixed set of occupied stages, only initial states that land exactly on a nonmaximal balanced orbit avoid convergence to this worst-case limit cycle. The resulting throughput losses approach 50\% when decoding lengths are large relative to \inputlens{}, compared with the eviction-free equilibrium. The analysis also develops techniques that may be of independent interest for studying discrete-time service systems with growing resource requirements. These techniques combine imbalance-amplification arguments, support-loss analysis, and cycle-stability analysis to characterize the stability or instability of different equilibrium regimes.

For heterogeneous job sizes, we establish a sharp stability criterion governed by number-theoretic structure under a large-input scaling (Theorem~\ref{thm:gcd}, Section~\ref{sec:multi_class}). The formal theorem is stated and proved for the two-class \commoninput{} setting: the eviction-free equilibrium is asymptotically stable if and only if the two decoding lengths are coprime. Appendix~\ref{app:theorem2} then explains how the same linear-recurrence and root-structure argument generalizes to multiple classes and \heteroinput{} lengths. The underlying mechanism is synchronization: when decoding lengths share a common divisor, completion events align periodically, producing persistent oscillatory modes that drive the system into eviction. When they are coprime, completion phases drift relative to one another, desynchronizing memory release and damping perturbations. Thus, heterogeneity stabilizes the system only when it breaks arithmetic synchronization; not all diversity is beneficial. Technically, the proof combines linear recurrence theory and spectral analysis. By linearizing the eviction-free dynamics, we obtain a finite-order recurrence for admission perturbations whose characteristic polynomial encodes the decoding lengths. Under the large-input scaling, the roots of this polynomial can be characterized and reveal the stability structure.

Service-induced congestion can lead to eviction, which in turn causes throughput degradation and increased latency. 
This observation motivates treating eviction-free operation as a performance target. Guided by our theoretical analysis, we propose two policies to achieve it: rate-limited admission when the eviction-free equilibrium is unstable, and request mixing for heterogeneous systems (Section \ref{sec:policies}). The former regulates concurrency through admission control, while the latter uses routing diversification to mitigate synchronization. Together, these policies
provide practical guidance for memory-aware LLM serving systems.

\jd{Resolved: Discuss numerical experiments here as well.}
Simulation experiments are consistent with the theoretical mechanisms and illustrate the practical impact of the proposed policies (Section~\ref{sec:numerics}). A model-based simulator suggests that, under stochastic arrivals, the open system exhibits an empirical transition near the worst-cycle throughput identified in the saturated-input analysis: in our runs, baseline admission accumulates backlog once the arrival rate exceeds the worst-cycle throughput, even if it remains below the nominal eviction-free equilibrium throughput. We then use Vidur, a high-fidelity LLM inference serving simulator, together with real-GPU experiments, to show that practical batching, scheduling, and memory-accounting dynamics are consistent with the same synchronization mechanism. Request mixing reduces synchronized completions and improves performance even when mixing only partially reduces the common periodicity. Rate-limited admission (using the analytically derived eviction-free admission rate as a per-iteration cap) mitigates eviction cascades while keeping the system close to full utilization. This approach can also yield substantial latency improvements.

\subsection{Related Work}
\label{sec:related_work}
\textbf{LLM serving systems.}
Efficient LLM inference has motivated a growing literature on batching, memory management, and scheduling. \citet{orca2022} introduce continuous batching, or iteration-level scheduling, which allows new requests to enter an active batch between decoding iterations and substantially improves GPU utilization. This batching paradigm is now central to modern serving systems, but it also creates the high-concurrency regime in which many requests simultaneously maintain growing KV caches. Memory management has therefore become a first-order concern. \citet{vllm2023} propose PagedAttention to reduce KV-cache fragmentation and enable flexible memory sharing across requests. NVIDIA TensorRT-LLM~\citep{tensorrt2023} similarly supports in-flight batching and paged KV-cache management. \citet{sglang2024} introduce RadixAttention, which reuses KV-cache states through a radix-tree structure that exploits shared prefixes across requests. FastServe~\citep{wu2023fastserve} focuses on scheduling, using token-level preemption and a skip-join multi-level feedback queue to improve responsiveness under load. Other systems improve efficiency by restructuring the prefill-decode workflow: \citet{sarathi2024} use chunked prefill, and \citet{distserve2024} and \citet{splitwise2024} separate or specialize resources across prefill and decode phases. These works provide mechanisms for efficient LLM serving under memory constraints. Our work is complementary: rather than proposing another memory-management primitive, we analyze the system-level dynamics induced by growing KV caches under continuous batching. We show that locally feasible admission can create future memory pressure, leading to eviction cascades and throughput loss when aggregate demand approaches capacity.

A recent stream of work applies stochastic modeling and online optimization to LLM inference scheduling. \citet{li2025throughput} establish throughput optimality for a broad class of work-conserving scheduling policies in both standard LLM inference and AI-agent workloads, highlighting work conservation as a central design principle for high-throughput serving. \citet{mitzenmacher2025queueing} survey emerging queueing-theoretic questions in LLM inference, including dynamic KV-cache memory footprints, output-length uncertainty, and preemption. \citet{nie2026queueing} develop a queueing-theoretic stability framework for KV-cache-constrained inference services. Most closely related to our work, \citet{ao2025optimizing} formulate LLM inference as a multi-stage online scheduling problem with memory constraints and derive fluid-guided threshold policies. Complementary online-optimization perspectives are developed by \citet{jaillet2025online}, who study online scheduling under KV-cache constraints, and \citet{chen2025robust}, who address output-length prediction uncertainty and provide logarithmic competitive guarantees. Recent work also studies routing and load balancing for LLM serving systems; see, for example, \citet{bari2025optimal}, \citet{chen2026universal}, \citet{lin2026large}, and \citet{zhang2024distributed}. Our work takes a different perspective from this literature. We characterize the qualitative behavior of the baseline continuous-batching dynamics near the memory boundary. We show that the feedback between admission, memory growth, and eviction creates a structural failure mode that is not visible from standard throughput or feasibility arguments.

\noindent\textbf{Queueing theory.}
Conceptually, our work is related to queueing models in which congestion affects service capacity, resource availability, or admission decisions. Queueing models with state-dependent service rates capture settings where congestion changes the effective processing capacity of the system, such as service slowdowns under high load~\citep{dong2015slowdown,delasay2019load,wu2022service}. Bandwidth-sharing and processor-sharing models study how a fixed service capacity is divided among concurrent jobs flows~\citep{massoulie2000bandwidth,zwart2000sojourn,gupta2022approximations}. A related stream on admission control regulates congestion by restricting entry into capacity-constrained systems through tolls, finite waiting room, or dynamic policies~\citep{naor1969,stidham1985,yoon2004optimal,ayhan2022optimal,cohen2024learning,peng2024admission}. These works capture important forms of state-dependent congestion and resource sharing, but they typically treat an admitted job's resource requirement as fixed during service. In LLM serving, by contrast, each admitted request continues to accumulate KV-cache memory during service, so an admission that is feasible at the current time can create future memory overflow.

Our modeling approach also connects to fluid and stability analyses of queueing networks, which use deterministic approximations to characterize congestion and performance when the system operates near or above capacity~\citep{dai1995positive,whitt2006fluid, bassamboo2010accuracy, chan2014use, bramson2021stability}. Our model shares this goal, but the mechanism is different: the memory state evolves through discrete decoding iterations, and the interaction between stage advancement, admission, and eviction creates synchronization effects. As a result, the system may admit multiple equilibria, including periodic limit-cycle behavior.

\section{Model}
\label{sec:model}

\subsection{Background: How LLM Inference Uses GPU Memory}

Before the formal model, we summarize the inference mechanics that create memory pressure.
An LLM generates text one \textit{token} at a time, where a token is the atomic unit of text (roughly a word or word fragment). Each request arrives with an \textit{\inputctx{}} whose length varies across requests. The model first reads the entire \inputctx{} in a single pass (the \textit{prefill} phase) and then enters the \textit{decoding} phase, producing one output token per iteration. Because each new token depends on all preceding tokens (autoregressive generation), the model must store a pair of key and value vectors per layer for every token processed so far, collectively called the \textit{KV cache}. This cache lets the model attend efficiently: without it, the model would reprocess the entire generated sequence at every step.

The KV cache must reside in GPU high-bandwidth memory (HBM) for fast attention computation. HBM is orders of magnitude smaller than system random-access memory (RAM) and, unlike disk-backed virtual memory, cannot be overcommitted: if the aggregate KV cache of all active requests exceeds capacity, the system must evict some requests. In practice, eviction removes a request from the active batch and discards its KV cache, forcing the system to recompute the cache later~\citep{vllm2023,tensorrt2023}. Crucially, each request's KV cache grows with every generated token. A request that fits in memory at admission may no longer fit several tokens later.

Modern LLM servers use \textit{continuous batching} to exploit GPU parallelism~\citep{orca2022,vllm2023,sarathi2024}. Rather than serving one request to completion before starting the next, the server processes many requests simultaneously. In each iteration, all active requests advance by one decoding step, while new requests may be admitted and completed ones retired. Batching is essential for high throughput because a single decoding step typically occupies only a small fraction of the GPU's parallel compute capacity. Advancing many requests in parallel amortizes per-step overheads and keeps GPU cores busy. However, batching also means that every active request maintains its own KV cache in GPU memory. Total memory usage scales with the number of concurrent requests, making memory exhaustion inherent to high-throughput operation.

\subsection{System Model}
\label{sec:system_model}
\label{sec:protocol}

Consider a single GPU with a memory capacity of $M$ tokens serving a stream of heterogeneous requests. Each request arrives with an \inputctx{} of length $l_0$ and requires $l_1$ decoding tokens; both $l_0$ and $l_1$ can vary across requests. %
We abstract away the prefill phase and focus on the decoding stage, where the KV cache grows over time and drives memory pressure. This is justified by disaggregated or prefill-specialized serving architectures studied in recent systems~\citep{sarathi2024,distserve2024,splitwise2024}: prefill runs on separate hardware, and the resulting KV cache is transferred to the decoding GPU via high-bandwidth interconnects (e.g., NVLink). The transfer is fast relative to decoding iteration time, so the decoding GPU sees each new request as arriving with its \inputlen{} KV cache already in place.
At decoding stage $j \in \{0, \ldots, l_1{-}1\}$, the request occupies $l_0 + 1 + j$ tokens of GPU memory, consisting of the \inputctx{}, the j tokens generated so far, and one additional token reserved for the next decoding step.

For analytical tractability, we group requests into $K \geq 1$ classes by their \inputterm{} and decoding lengths: class-$k$ requests have \inputlen{} $l_{0,k}$ and decoding length $l_{1,k}$. Throughout, we assume deterministic lengths within each class to isolate the effect of memory growth on stability. We also assume the GPU memory capacity satisfies $M>\max_{k}\{l_{0,k}+l_{1,k}\}$, so that even the largest request can be processed in isolation. In practice, GPU memory is typically much larger than the footprint of a single request, allowing many requests to be served concurrently. In simulations, we assume class-$k$ requests arrive as independent Poisson processes with rate $\lambda_k$.

Let $Q^n=(Q^n_{1},\dots, Q^n_{K})$ denote the queue of waiting requests at time~$n$, and let
\[
X^n = \bigl(x^{n}_{k,j} : k\in[K],\ j\in\{0,\ldots,l_{1,k}-1\}\bigr)
\]
denote the vector of active request mass by class and decoding stage.
A class-$k$ request at stage $j$ occupies
\(
w_{k,j} = l_{0,k} + 1 + j
\)
tokens of GPU memory. We adopt the convention $w_{k,l_{1,k}} = 0$: a request that completes its final decoding step releases all memory. The total memory usage at time $n$, corresponding to the post-decoding (one-step-ahead) state, is
\[
M^n = \sum_{k=1}^{K}\sum_{j=0}^{l_{1,k}-1} w_{k,j}\,x^{n}_{k,j} .
\]

At each iteration $n$, the system evolves as follows.

\paragraph{1. Execute.}
Every active request advances one decoding stage and generates one token.
Class-$k$ requests at their final stage ($j = l_{1,k}-1$) complete and leave the system, releasing their memory.
Define the post-execution state and the corresponding memory usage as
\[
\tilde x^{n}_{k,j} =
\begin{cases}
x^{n}_{k,j-1}, & j=1,\dots, l_{1,k}-1,\\[4pt]
0, & j=0,
\end{cases}
\quad \mbox{ and } \quad
\tilde M^n
= \sum_{k=1}^{K}\sum_{j=0}^{l_{1,k}-1} w_{k,j}\,\tilde x^{n}_{k,j} \quad \mbox{ respectively.}
\]

\paragraph{2. Arrive.}
New requests arrive and join the waiting queue:
\(
Q^n \leftarrow Q^n + A^n
\),
where $A^n = (A^n_{1},\ldots,A^n_{K})$.
(We do not specify the arrival process at this stage; this will be addressed later.)
\paragraph{3. Evict.}
Each surviving request grew by one token during execution.
If the post-execution memory exceeds capacity,
\(
\tilde M^n > M,
\)
the system evicts active requests until feasibility is restored ($\tilde M^n \le M$).
Evicted requests forfeit their KV cache and re-enter the queue.
Evictions follow a \textit{least-progressed-first} (LPF) rule: requests at smaller (earlier) stages are evicted before those at larger (later) stages, regardless of class. If LPF partially evicts a stage containing multiple classes, it removes each class in proportion to its occupancy within that stage. Let $\hat X^n$ denote the resulting active state after eviction.

\paragraph{4. Admit.}
Requests from $Q^n$ are admitted into stage $0$ subject to the memory constraint:
\[
\sum_{k,j} w_{k,j}\,\hat x^{n}_{k,j}
+ \sum_{k} w_{k,0}\,a^n_{k} \le M,
\]
where $a^n_{k}$ is the number of class-$k$ requests admitted at time $n$, and
each admission consumes $w_{k,0} = l_{0,k} + 1$ tokens of memory. (We do not specify how requests are selected from the queue at this stage; this will be addressed later.)
Admitted requests enter at stage $0$, and the remaining requests stay waiting in the queue. The resulting state after admission defines $X^{n+1}$ and $Q^{n+1}$.

The sequence above defines one \emph{decoding iteration}.
Completions occur when class-$k$ requests at stage $j=l_{1,k}-1$ advance and
leave the system during the execution step. Let $D^n$ denote the number of
completions at time $n$. The throughput is
\[
\bar{T}=\lim_{N\rightarrow\infty}\frac{1}{N}\sum_{n=0}^{N}D^n
\]
whenever the limit exists. This per-iteration completion rate reflects the
degree of concurrency sustained by the active batch; in the saturated-input
regime formalized below, it isolates the memory-limited throughput. When
iteration duration is approximately constant under sustained continuous
batching, $\bar{T}$ also approximates wall-clock throughput, but our analysis
is stated entirely in completions per iteration.

This LPF rule captures a common design principle in LLM serving systems under
memory pressure: when preemption or recomputation is needed, schedulers often
preserve requests that have already made more decoding progress and remove
younger or less-progressed work first~\citep{vllm2023,tensorrt2023}. In the
homogeneous continuous-batching model, arrival order and decode progress are
aligned, so evicting the latest admitted requests is equivalent to evicting the
least-progressed requests. We therefore use LPF as a decode-progress-aware
abstraction of such preemption behavior. The rule prioritizes retaining
later-stage requests, while abstracting away system-specific recovery costs such
as prefill recomputation, swapping, or prefix-cache reuse.

The baseline system employs no admission control beyond a one-step memory feasibility check:
requests are admitted whenever their anticipated memory footprint (including reserved space for the next decoding step) fits within the available GPU memory.
In particular, the admission decision does not reserve capacity for the future growth of KV caches. This greedy policy reflects common practice in existing serving systems ~\citep{vllm2023,tensorrt2023} and isolates the mechanisms that drive throughput loss. Section~\ref{sec:policies} builds on our analysis to develop improved admission policies.

\begin{example}[Protocol trace]
\label{ex:protocol_trace}

Consider a single class of requests ($K=1$) with $l_0 = 2$, $l_1 = 3$, and memory capacity $M = 24$.
Stage weights are $w_j = l_0 + 1 + j$, $j=0,1,2$,
so stage-$0$, stage-$1$, and stage-$2$ requests occupy $3$, $4$, and $5$ tokens of memory, respectively.

We trace two consecutive iterations starting from
$X^n=(x^n_0,x^n_1,x^n_2)=(1,1,2)$ and $Q^n=8$, and assume $A^n=5$ and $A^{n+1}=0$.

\noindent{\bf Iteration $n$.}
\begin{enumerate}[nosep]
\item \textit{Execute.}
The two stage-$2$ requests complete, while the remaining requests advance one stage.
The post-execution state becomes
$\tilde X^n=(0,1,1)$,
with memory usage
$\tilde M^n = 4 + 5 = 9$.

\item \textit{Arrive.}
Five new requests arrive, increasing the queue length to
$Q^n \leftarrow 8 + 5 = 13$.

\item \textit{Evict.}
Since $\tilde M^n = 9 \le 24$, no eviction occurs.

\item \textit{Admit.}
Each stage-$0$ request requires $w_0=3$ tokens of memory.
The feasibility condition
$9 + 3a \le 24$
implies $a \le 5$, so the system admits $a=5$ requests.
Then,
$X^{n+1}=(5,1,1)$,
with memory usage
$M^{n+1}=5\cdot3 + 4 + 5 = 24$,
and queue length
$Q^{n+1}=13-5=8$.
\end{enumerate}

\noindent{\bf Iteration $n+1$.}
\begin{enumerate}[nosep]
\item \textit{Execute.}
The single stage-$2$ request completes, while the remaining requests advance one stage, yielding
$\tilde X^{n+1}=(0,5,1)$.
The resulting memory usage is
$\tilde M^{n+1}=0\cdot3 + 5\cdot4 + 1\cdot5 = 25 > 24$.

\item \textit{Arrive.}
No new requests arrive, so the queue length remains
$Q^{n+1} \leftarrow Q^{n+1} + A^{n+1} = 8 + 0 = 8$.

\item \textit{Evict.}
Since $\tilde M^{n+1} > 24$, the system evicts requests until feasibility is restored.
Under LPF, requests at smaller stages are evicted first.
Evicting one stage-$1$ request frees $w_1=4$ tokens, reducing memory to
$25 - 4 = 21 \le 24$.
The post-eviction state is therefore
$\hat X^{n+1}=(0,4,1)$,
and the evicted request rejoins the queue, giving
$Q^{n+1} \leftarrow 8 + 1 = 9$.

\item \textit{Admit.}
After eviction, $24-21=3$ tokens of memory remain available.
Since each stage-$0$ admission requires $3$ tokens, the system admits $a=1$ request.
The resulting state is
$X^{n+2}=(1,4,1)$,
with memory usage
$M^{n+2}=1\cdot3 + 4\cdot4 + 1\cdot5 = 24$,
and $Q^{n+2}=9-1=8$.
\end{enumerate}

This trace illustrates how admission decisions that are feasible at the time they are made can nevertheless lead to memory overflow in the subsequent iteration.
\end{example}

\subsection{Equilibria and Stability}

We consider a \emph{saturated-input regime} in which enough waiting work is always available, so admission is never demand-constrained. Formally, we replace the exogenous arrival process by an effectively infinite backlog,
so that the admission decision at time $n$ is constrained only by GPU memory feasibility.
This decouples memory-driven throughput limits from the stochastic arrival process, isolating the maximum sustainable throughput imposed by memory growth and eviction. This is a theoretical abstraction, not a literal description of day-to-day operation. Real systems may move in and out of this regime. Nevertheless, operators strive to keep GPU memory utilized, since idle capacity represents wasted capital~\citep{vllm2023,tensorrt2023}.
The memory-constrained regime therefore reveals the fundamental limits that memory growth imposes on performance. Numerical experiments in Section~\ref{sec:numerics} confirm that these insights extend beyond this idealized regime.

In the saturated-input regime, the queue never empties, so admissions depend only on the current active state~$X^n$, the memory constraint~$M$, and the specified class-selection rule. In the single-class model this rule is automatic; in the multi-class analysis it is fixed by the proportional-admission convention in Section~\ref{sec:multi_class}. We therefore take $X^n$ as the system state. Under these conventions, the dynamics reduce to a deterministic map
\begin{equation}\label{eq:dynamic}
X^{n+1} = F(X^n).
\end{equation}
Our objective is to characterize the macroscopic structure of this map, e.g., its equilibrium behaviors and their dependence on job sizes. Throughout the theoretical analysis in Sections~\ref{sec:single_class} and \ref{sec:multi_class}, we use a continuous deterministic formulation in which the state variables \(x^n_{k,j}\) represent nonnegative request masses. Admission and LPF eviction are interpreted in the same continuous state space: the admission rule can fill available memory exactly, and when LPF reaches a partially occupied stage, it may remove a proportional amount of request mass from that stage. We call
\[
\mathcal B
:=
\left\{X\ge0:\ \sum_{k=1}^K\sum_{j=0}^{l_{1,k}-1} w_{k,j}x_{k,j}=M\right\}
\]
the \emph{memory boundary}. We use the calligraphic symbol \(\mathcal B\) for this boundary; ordinary symbols such as \(B\) are reserved for matrices and other local objects. In the saturated-input continuous model, post-admission states lie on \(\mathcal B\), and the stability statements below are interpreted relative to the induced memory-boundary dynamics. This is the dynamical system to which the equilibrium and stability results apply. The numerical experiments in Section~\ref{sec:numerics} examine finite-request stochastic behavior separately.

For the discrete-time dynamical system \eqref{eq:dynamic}, two classes of long-run behaviors can arise.

\begin{definition}[Fixed point]
\label{def:fix_point}
A \textbf{fixed point} is a state $X^*$ satisfying $F(X^*) = X^*$.
\end{definition}

In our model, a fixed point is a steady regime: the number of requests at each stage, total memory usage, and throughput all remain constant, with no eviction.

\begin{definition}[Limit cycle]
\label{def:limit_cycle}
A \textbf{limit cycle} of period $p \ge 2$ is a sequence $X_1^*, \dots, X_p^*$ where
\[
F(X_i^*) = X_{i+1}^* \quad \text{for } i=1,\ldots,p-1,
\qquad
F(X_p^*) = X_1^*,
\]
and $p$ is the smallest positive integer for which this property holds.
\end{definition}

We refer to both fixed points and limit cycles as equilibria. For a limit cycle of period $p$, the throughput reduces to
\[
\bar{T}=\frac{1}{p}\sum_{i=1}^{p}\sum_{k=1}^{K}(X_i^*)_{k,l_{1,k}-1},
\]
where $(X_i^*)_{k,l_{1,k}-1}$ denotes the mass of class-$k$ requests at its final stage in the $i$th state of the cycle. For a fixed point, $\bar{T}=\sum_{k=1}^{K}X^*_{k,l_{1,k}-1}$.

We further distinguish equilibria by their stability. 
Let $\|\cdot\|$ denote a norm on the state space. For a limit cycle $X_1^*,\ldots,X_p^*$, define the cycle set
$\mathcal{C} = \{X_1^*,\ldots,X_p^*\}$ and
$\mathrm{dist}(X,\mathcal{C}) := \min_{Y\in\mathcal{C}} \|X-Y\|$.

\begin{definition}[Stability]
\label{def:stability}
A fixed point $X^*$ is said to be:
\begin{itemize}[nosep]
    \item \textbf{Stable} if for every $\epsilon>0$ there exists $\delta>0$ such that for any $X^0$ satisfying $\|X^0 - X^*\| < \delta$, $\|X^n - X^*\| < \epsilon$ for all~$n$.
    \item \textbf{Asymptotically stable} if it is stable and there exists $\delta>0$ such that for any $X^0$ satisfying $\|X^0 - X^*\| < \delta$,
    \(
        \lim_{n\to\infty} \|X^n - X^*\| = 0
    \).
\end{itemize}
A limit cycle $\mathcal{C} = \{X_1^*,\ldots,X_p^*\}$ is said to be:
\begin{itemize}
    \item \textbf{Stable} if for every $\epsilon>0$ there exists $\delta>0$ such that for any $X^0$ satisfying $\mathrm{dist}(X^0,\mathcal{C}) < \delta$,
$\mathrm{dist}(X^n,\mathcal{C}) < \epsilon$
for all $n\ge 0$.
\item \textbf{Asymptotically stable} if it is stable and there exists $\delta>0$ such that for any $X^0$ satisfying $\mathrm{dist}(X^0,\mathcal{C}) < \delta$,
\(
\lim_{n\to\infty} \mathrm{dist}(X^n,\mathcal{C}) = 0
\).
\end{itemize}
\end{definition}

For a limit cycle, stability requires convergence to the orbit $\mathcal{C}$, not to a single point. In particular, asymptotic stability implies that there exists an integer shift $\tau\in\{0,\ldots,p-1\}$ such that
\[
\|X^{n} - X_{(n+\tau)\bmod p}^*\| \to 0 .
\]

In Sections~\ref{sec:single_class} and~\ref{sec:multi_class}, the stability notions are applied to the induced memory-boundary dynamics on \(\mathcal B\). Some of our results involve nongeneric exact-capture exceptions. For these statements, we use the same stability definition relative to an exceptional set. Specifically, given a set \(\mathcal E\) of initial states, a limit cycle \(\mathcal C\) is \textbf{asymptotically stable outside \(\mathcal E\)} if the conditions above hold for every initial state satisfying \(\mathrm{dist}(X^0,\mathcal C)<\delta\) and \(X^0\notin\mathcal E\). In addition, \(\mathcal E\) is a relative Lebesgue-measure-zero subset of the relevant memory-boundary faces.

\section{Structural Instability of the Homogeneous System}
\label{sec:single_class}

We begin with a homogeneous system in which every request has \inputlen{} $l_0$ and decoding length $l_1$. Dropping the class superscript, write $w_j = l_0 + 1 + j$ for the memory footprint at stage $j$.

\subsection{Eviction-Free Dynamics and Fixed Point}
\label{sec:balance}

Consider one iteration starting from state $X^n=(x^{n}_{0},\ldots,x^{n}_{l_1-1})$ with total memory usage $M^n=M$.
During the \textit{Execute} step, every stage-$j$ request advances to stage $j{+}1$ and increases its memory footprint by one token.
Requests at the final stage $j=l_1-1$ complete and depart, releasing their entire footprint, i.e., $w_{l_1-1}=l_0+l_1$.
The net memory freed by the execute step equals $(l_0+l_1)\,x^{n}_{l_1-1} -\sum_{j=0}^{l_1-2} x^{n}_{j}$.
In the saturated-input regime, the system converts all freed memory into new stage 0 requests, each costing $w_0=l_0+1$ tokens. Therefore, the number of newly admitted (stage 0) requests at the start of iteration $n+1$ satisfies the balance relation
\begin{equation}\label{eq:single_balance}
(l_0+1)\,x^{n+1}_{0}
=
(l_0+l_1)\,x^{n}_{l_1-1}
-
\sum_{j=0}^{l_1-2} x^{n}_{j}.
\end{equation}
For $j=0,1,\ldots,l_1-2$, stage advancement gives
$x^{n+1}_{j+1} = x^{n}_{j}$.
Combining these relations yields a linear update
$X^{n+1} = B\,X^n$ where
\begin{equation}\label{eq:B1}
B = \begin{bmatrix}
\alpha & \alpha & \cdots & \alpha & \beta \\
1 & 0 & \cdots & 0 & 0 \\
\vdots & & \ddots &  & \vdots \\
0 & 0 & \cdots & 1 & 0
\end{bmatrix}, \quad
\alpha = \frac{-1}{l_0+1},\;\; \beta = \frac{l_0+l_1}{l_0+1}.
\end{equation}
The first row encodes the balance equation; subsequent rows encode the stage shift. The coefficient $\beta > 1$ captures a fundamental asymmetry: a completing request frees $\beta$ admission slots. Note that when $l_1 \gg l_0$, $\beta$ is large and each completion can trigger a disproportionate admission burst.

Solving $F(X^*)=X^*$ on the memory boundary yields the eviction-free fixed point. Let $\mathbf{1}\in\mathbb{R}^{l_1}$ denote the all-ones vector.
Since $\alpha(l_1-1)+\beta=1$, we have $B\mathbf{1}=\mathbf{1}$, and hence any fixed point must be of the form $X^*=x^*\mathbf{1}$.
Imposing $M=\sum_{j=0}^{l_1-1} (l_0+1+j)x^*$ yields
\begin{equation}\label{eq:xstar}
x^* = \frac{2M}{l_1(2l_0 + l_1 + 1)}.
\end{equation}
At $X^*$, each stage holds $x^*$ requests, memory remains at capacity $M$ with no eviction, and throughput equals $x^*$ completions per iteration.

\subsection{Equilibria with Eviction}
\label{sec:equilibria_landscape}

The fixed point $X^*$ is not the only equilibrium. When memory overflows, the dynamics admit periodic regimes in which some stages remain empty at every iteration.

For an integer $i\in\{0,1,\ldots,l_1-1\}$, we say the system is at \emph{eviction level $i$} at time $n$ if exactly $i$ coordinates of $X^n$ equal zero and the remaining $l_1-i$ coordinates are strictly positive.
In particular, level~$0$ corresponds to eviction-free operation (no empty stages).
An \emph{eviction-level-$i$ limit cycle} is a periodic orbit $\{X_1,\ldots,X_p\}$ such that every state on the orbit has exactly $i$ empty coordinates.

\begin{example}[Example equilibria]
\label{ex:canonical_cycles}

Consider $l_0 = 2$, $l_1 = 3$, memory capacity $M = 24$.
The system admits three equilibria corresponding to eviction levels $i=0,1,2$.

\noindent{\bf Eviction level $0$.}
The fixed point is
$X^*=(2,2,2)$
and achieves throughput
$x^* = 2$.

\noindent{\bf Eviction level $1$.}
At eviction level $1$, the system admits the period-$3$ limit cycle
\[
\Bigl(\tfrac{72}{13},\, \tfrac{24}{13},\, 0\Bigr)
\;\to\;
\Bigl(0,\, \tfrac{48}{13},\, \tfrac{24}{13}\Bigr)
\;\to\;
\Bigl(\tfrac{24}{13},\, 0,\, \tfrac{48}{13}\Bigr)
\;\to\;
\Bigl(\tfrac{72}{13},\, \tfrac{24}{13},\, 0\Bigr),
\]
with throughput $(24/13 + 48/13)/3 \approx 1.85$. Note that eviction occurs on the first transition.

\noindent{\bf Eviction level $2$.}
At eviction level 2, the system admits the period-3 limit cycle
\[
(8,0,0) \;\to\; (0,6,0) \;\to\; \Bigl(0,0,\tfrac{24}{5}\Bigr),
\]
with throughput $8/5$, corresponding to a $20\%$ reduction from the fixed-point throughput $x^*$.
\end{example}

\begin{remark}
For $l_1\ge 4$, limit cycles at a given eviction level $i$, $2\leq i\leq l_1-2$, need not be unique. Stage
$j$ is said to be \emph{occupied} if it contains at least one active request.
Different configurations of occupied stages can produce distinct periodic orbits with different throughputs.
Figure~\ref{fig:appendix_l4_cycles} illustrates this landscape. For a fixed eviction level, some cycles concentrate all occupied stages into a single \emph{contiguous block} of stage indices, while others distribute occupied stages as evenly as possible around the pipeline. Appendix~\ref{app:sec3_proofs} characterizes these families through the fixed-support gap-vector envelope (Proposition~\ref{prop:gap_vector_envelope}); its endpoint formulas give the contiguous-block lower-throughput family and the exactly evenly spaced upper-envelope cases. Thus, contiguous-block configurations achieve the lowest throughput, whereas as-evenly-spaced configurations achieve the highest throughput among fixed-relative-support periodic cycles with the same number of occupied stages.
Equivalently, for the cycled equilibria described by a fixed set of relative
live positions, the predecessor-gap vector gives a general parametrization, and
the contiguous and as-evenly-spaced gap vectors give the lower and upper
throughput bounds within that equilibrium class.
\end{remark}

\begin{figure}[ht]
\centering
\includegraphics[width=\textwidth]{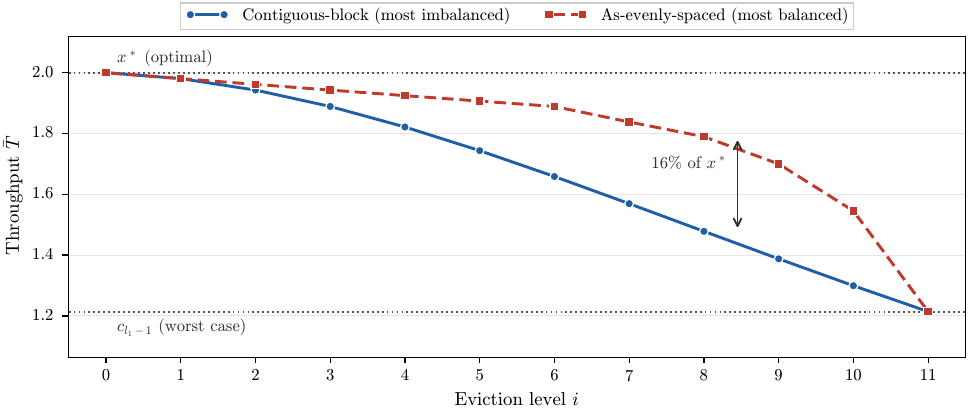}
\caption{Throughput landscape across eviction levels for $l_0 = 2$, $l_1 = 12$. The fixed-support gap-vector envelope (Proposition~\ref{prop:gap_vector_envelope}) shows that the contiguous-block construction gives the lower-throughput envelope, while the as-evenly-spaced construction gives the upper envelope, with the exactly evenly spaced divisor cases lying on that curve. All fixed-relative-support periodic cycle throughputs at the same eviction level lie between these two curves. Both envelopes coincide at level~0 ($x^*$, the eviction-free optimum) and level~$l_1{-}1$ ($c_{l_1-1}$, the worst case). The gap widens at intermediate levels, reaching 16\% of $x^*$ at level~8.}
\label{fig:appendix_l4_cycles}
\end{figure}

At the maximal eviction level $i=l_1-1$, exactly one stage is active at each iteration, whose size saturates the memory constraint.
The resulting cycle has the form
\[
\Bigl(\frac{M}{w_0},0,\ldots,0\Bigr)
\to
\Bigl(0,\frac{M}{w_1},0,\ldots,0\Bigr)
\to
\cdots
\to
\Bigl(0,\ldots,0,\frac{M}{w_{l_1-1}}\Bigr).
\]
The corresponding throughput equals
\[
\bar{T}_{l_1-1}:=\frac{M}{l_1(l_0+l_1)}.
\]
The maximal-eviction cycle leads to the lowest throughput among all equilibria:

\begin{lemma}[Worst-case throughput]
\label{lem:monotonicity}
Let $\mathcal{C}$ be any limit cycle and let $\bar{T}(\mathcal{C})$ denote its throughput.
\[
\bar{T}(\mathcal{C}) \ge \bar{T}_{l_1-1} = \frac{M}{l_1(l_0+l_1)}.
\]
Moreover, equality holds if and only if $\mathcal{C}$ is the maximal-eviction cycle (i.e., eviction level $l_1-1$).
\end{lemma}

We define the worst-case throughput ratio as
\[
\frac{\bar{T}_{l_1-1}}{x^*}
= \frac{l_1(2l_0+l_1+1)/2}{l_1(l_0+l_1)}
= \frac{2l_0+l_1+1}{2(l_0+l_1)} .
\]
This ratio lies in $(1/2,1)$. It decreases with the decoding length $l_1$ and increases with the \inputlen{} $l_0$. The ratio approaches $1/2$ as $l_1/l_0\rightarrow\infty$.

\subsection{Eviction Cascade} \label{sec:single_cascade}
Although the eviction-free fixed point $X^*$ achieves the maximum throughput, it is unstable. Baseline admission amplifies any small imbalance across stages, eventually triggering eviction. Once eviction begins, the system does not return to eviction-free operation; instead, it cascades through regimes with progressively more empty stages and settles into the worst-case limit cycle at eviction level~$l_1{-}1$. This eviction cascade is formalized in the following theorem.

\begin{theorem}[Baseline Instability]
\label{thm:fcfs}
For a fixed decoding length $l_1\ge 2$, in the saturated-input continuous model, the system admits an eviction-free fixed point, a maximal-eviction limit cycle at eviction level \(l_1-1\), and, when \(l_1\ge3\), intermediate eviction-level limit cycles. The stability and basin structure under baseline admission are as follows.
\begin{enumerate}[nosep]
    \item The eviction-free fixed point is not stable.
    \item Eviction-level-$i$ limit cycles, for $i=1,\dots, l_1-2$, are not stable.
\end{enumerate}
Let \(\mathcal E\) denote the set of memory-boundary initial states whose trajectories are captured exactly by a nonmaximal balanced orbit; equivalently, there exists a finite time after which the trajectory lies on either the eviction-free fixed point or an intermediate balanced limit cycle. Within the relative interior of each fixed support face of \(\mathcal B\), i.e., after fixing which coordinates are zero and which are positive, the set \(\mathcal E\) has relative Lebesgue measure zero. 
\begin{enumerate}
    \item[3.] The maximal-eviction limit cycle, at eviction level \(l_1-1\), is asymptotically stable outside \(\mathcal E\). Its basin of attraction contains every memory-boundary initial state except those whose trajectories are captured exactly by a nonmaximal balanced orbit. The exceptional set \(\mathcal E\) has relative Lebesgue measure zero within each such fixed support face of the memory boundary.
\end{enumerate}
\end{theorem}

The instability is driven by an amplification mechanism, captured by the \emph{spread}
\begin{equation}\label{def:spread}
G^n:=\max_j x^{n}_{j}-\min_j x^{n}_{j}.
\end{equation}
Before any eviction happens (i.e., before any stage occupancy has reached
zero), the stage shift preserves all non-final coordinates in $X^n$, up to the deterministic relabeling of stages.
For the final and first stage, rewrite the memory balance equation in
\eqref{eq:single_balance} as
\begin{equation} \label{eq:balance_re}
x^{n+1}_{0}
=
x^{n}_{l_1-1}
+
\frac{\sum_{j=0}^{l_1-1}(x^{n}_{l_1-1}-x^{n}_{j})}{l_0+1}
\end{equation}
Let $x_{\mathrm{avg}}^n=l_1^{-1}\sum_{j=0}^{l_1-1}x^n_j$. Then if $x^{n}_{l_1-1}\leq x_{\mathrm{avg}}^n$, $x^{n+1}_{0}\leq x^{n}_{l_1-1}$. If $x^{n}_{l_1-1}>x_{\mathrm{avg}}^n$, $x^{n+1}_{0}>x^{n}_{l_1-1}$. Thus, $G^n$ is non-decreasing.
Moreover, whenever an extremal cohort reaches the final stage, the spread grows
strictly. If the current minimum occupies the final stage, then every term
\(x^n_{l_1-1}-x^n_j\) in~\eqref{eq:balance_re} is nonpositive and at least one
term is at most \(-G^n\). Hence
\[
x^{n+1}_{0}
\leq
x^{n}_{l_1-1}-\frac{G^n}{l_0+1}.
\]
The old maximum is preserved by the shift, so the spread increases by at least
\(G^n/(l_0+1)\). Similarly, if the current maximum occupies the final stage,
then
\[
x^{n+1}_{0}
\geq
x^{n}_{l_1-1}+\frac{G^n}{l_0+1},
\]
while the old minimum is preserved, and the spread again increases by at least
\(G^n/(l_0+1)\).
Thus, if an eviction-free trajectory with \(G^0>0\) persisted, following a
current maximum or minimum around the pipeline would force repeated positive
spread increments. If that extremum is replaced before reaching the final
stage, the replacement is already a strict renewal of the same extremum; the
formal proof in Lemma~\ref{lem:A2} handles this bookkeeping through the running
maximum. Since the memory constraint bounds every stage occupancy by
\(M/(l_0+1)\), these repeated increments cannot continue forever. Therefore the
eviction-free regime must break down in finite time: the first memory overflow
triggers LPF eviction, creating the first eviction-induced support loss and
moving the trajectory to the next lower live-support level.

After support loss, the physical-stage spread is no longer the right object, because
eviction changes which physical stages carry mass. We therefore switch to a relative-position representation on a circular pipeline. In this representation, live positions move deterministically around the circle, and once a relative position becomes empty, it never revives. We then sample the dynamics only at \emph{block-end} times. An eviction block is a consecutive sequence of iterations in which evictions occur but no live position completes; the corresponding block-end is the first subsequent iteration in which a live position completes and new requests are admitted. Sampling at block-ends reduces the post-support-loss dynamics to a lower-dimensional iterative map. Lastly, we measure imbalance using a normalized block-end spread, where each live mass is divided by the predecessor gap, which is the number of stages from the previous live position. The normalization makes the mass more comparable because positions with larger gaps accumulate more mass during eviction. Under the normalized spread, the same amplification mechanism holds: unless the normalized live masses are exactly balanced, the block-end dynamics amplify imbalance and eliminate at least one additional live position after finitely many updates. The trajectory then moves to a lower-dimensional support face, where the same argument applies recursively (see Appendix~\ref{app:theorem1} for more details).

Since empty relative positions never revive, support losses form a finite cascade. Every trajectory not lying exactly on, or landing exactly on, a balanced intermediate cycle eventually reaches the single-live-position
regime, namely the maximal-eviction cycle. Such exact capture is nongeneric, requiring the normalized live masses to satisfy a collection of exact linear equalities. At the maximal-eviction level, the memory boundary and LPF eviction
damp perturbations within one pipeline cycle, so the maximal-eviction cycle is
asymptotically stable outside the exact-capture set.

\begin{example}[Instability spiral]
\label{ex:spiral}
Consider the system with $l_0 = 2$, $l_1 = 3$, and memory capacity $M = 24$ (same as Example~\ref{ex:protocol_trace}). The eviction-free equilibrium has $x^* = 2$.
We perturb this equilibrium by increasing stage~0 slightly, by $\tfrac12$.
Specifically, we start from
$X^0 = \left(\tfrac{5}{2},\, 2,\, \tfrac{17}{10}\right)$.
The balance equation~\eqref{eq:single_balance} yields the following eviction-free trajectory:

\begin{center}
\begin{minipage}[t]{0.44\textwidth}
\begin{center}
\renewcommand{\arraystretch}{1.45}
\begin{tabular}{c|c|c}
$n$ & $X^n=(x^{n}_{0},x^{n}_{1},x^{n}_{2})$ & $G^n$ \\ \hline
0 & $\left(\tfrac{5}{2},\,2,\,\tfrac{17}{10}\right)$ & $\tfrac{4}{5}$ \\
1 & $\left(\tfrac{4}{3},\,\tfrac{5}{2},\,2\right)$ & $\tfrac{7}{6}$ \\
2 & $\left(\tfrac{37}{18},\,\tfrac{4}{3},\,\tfrac{5}{2}\right)$ & $\tfrac{7}{6}$ \\
3 & $\left(\tfrac{82}{27},\,\tfrac{37}{18},\,\tfrac{4}{3}\right)$ & $\tfrac{46}{27}$ \\
4 & $\left(\tfrac{85}{162},\,\tfrac{82}{27},\,\tfrac{37}{18}\right)$ & $\tfrac{407}{162}$ \\
5 & $\left(\tfrac{544}{243},\,\tfrac{85}{162},\,\tfrac{82}{27}\right)$ & $\tfrac{407}{162}$ \\
6 & $\left(\tfrac{6037}{1458},\,\tfrac{544}{243},\,\tfrac{85}{162}\right)$ & $\tfrac{2636}{729}$
\end{tabular}
\end{center}
\end{minipage}\hfill
\begin{minipage}[t]{0.54\textwidth}
\begin{center}
\renewcommand{\arraystretch}{1.25}\setlength{\tabcolsep}{4pt}
\begin{tabular}{c|c|c|l}
$n$ & $X^n$ (approx.) & level & event \\ \hline
$7$ & $(0,\;3.20,\;2.24)$ & $1$ & first eviction \\
$10$ & $(0,\;2.67,\;2.66)$ & $1$ & second eviction \\
$13$ & $(0,\;1.56,\;3.55)$ & $1$ & third eviction \\
$16$ & $(0,\;0,\;4.80)$ & $2$ & second position dies \\
$17$ & $(8,\;0,\;0)$ & $2$ & worst-case cycle begins \\
$18$ & $(0,\;6,\;0)$ & $2$ & $\cdots$ \\
$19$ & $(0,\;0,\;4.8)$ & $2$ & $\cdots$
\end{tabular}
\end{center}
\end{minipage}
\end{center}

At iteration $n=6$, the execute step shifts the state to
$\tilde X^6=(0,\,x^{6}_{0},\,x^{6}_{1})=\Bigl(0,\,\tfrac{6037}{1458},\,\tfrac{544}{243}\Bigr)$,
whose post-execution memory usage is
\[
\tilde M^6=4x^{6}_{0}+5x^{6}_{1}=\frac{20234}{729}%
>24.
\]
LPF evicts from the least-progressed occupied stage (stage~1), reducing it by
$e=(\tilde M^6-24)/4=\tfrac{1369}{1458}$
and leaving no slack for admission:
$X^7=\Bigl(0,\,\tfrac{778}{243},\,\tfrac{544}{243}\Bigr)$.
This is the first eviction (see also Figure~\ref{fig:instability}).

Continuing from $X^7=(0,\,\tfrac{778}{243},\,\tfrac{544}{243})$, the system enters eviction level~$1$ (one empty stage per iteration).
Over the next nine iterations, LPF evictions progressively deplete the second live position. At $n=16$, the system enters the maximal-eviction cycle $(8,0,0)\to(0,6,0)\to(0,0,4.8)$, which repeats stably with throughput $\bar{T}_2=8/5$, a 20\% reduction from $x^*=2$.

\end{example}

Figure~\ref{fig:instability} in Appendix~\ref{app:theorem1} visualizes the
growing oscillation and first eviction in this example.

Above all, the instability is structural: baseline admission converts each completion burst into a disproportionate admission pulse ($\beta>1$), and progressive memory growth amplifies these pulses as they propagate through the pipeline. The pulses eventually synchronize all requests at a single stage, triggering maximal eviction. %

\section{Multi-Class Systems and Structural Stability}
\label{sec:multi_class}

Real LLM workloads are heterogeneous: short chat completions, long code-generation sessions, and multi-turn reasoning chains coexist on the same GPU pool, each with different \inputterm{} and decoding lengths. Variation in decoding lengths can desynchronize completions and mitigate the eviction cascades identified in Section ~\ref{sec:single_class}. As we show in this section, whether this desynchronization is strong enough to stabilize the system depends on the arithmetic structure of the decoding lengths.

\subsection{Multi-Class Model and Large-\Inputterm{} Regime}
Consider $K \geq 2$ request classes with decoding lengths $l_{1,1} < l_{1,2} < \cdots < l_{1,K}$ and arrival proportions $p_1, \ldots, p_K$ ($\sum_{k=1}^K p_k = 1$, and $p_k > 0$ for $k=1,\dots, K$). Each class~$k$ has \inputlen{} $l_{0,k}$.
As before, we focus on the saturated-input regime: the waiting queue is never empty, and admission is constrained only by memory. Because the backlog is effectively infinite, the system must specify how request classes are selected for admission. We impose the natural proportion constraint that \(x^n_{k,0}=p_k \sum_{h=1}^K x^n_{h,0}\).
This constraint separates memory dynamics from class-selection effects. 

To isolate the structural effects of heterogeneous decoding lengths, we analyze an \inputdominated{} asymptotic regime. Specifically, we consider a sequence of systems in which the decoding lengths remain fixed, while \inputlens{} and total memory scale proportionally:
\begin{equation}\label{eq:scaling}
l_{0,k} = r_k L, \quad M = c\,L, \quad L \to \infty,
\end{equation}
where $r_1, \ldots, r_K > 0$ and $c > 0$ are fixed constants. This scaling preserves non-degenerate concurrency while making per-stage memory costs nearly constant across stages. Indeed, for a class~$k$ request at stage~$j$, $w_{k,j} = l_{0,k}+1+j = r_k L\bigl(1 + O(1/L)\bigr)$, so the relative variation in memory usage across decoding stages vanishes as $L$ grows. In this limit, the dominant source of dynamics is not the precise stage-dependent memory weight, but the timing of request completions, which is
determined by the interaction among the decoding lengths.

This regime also reflects an \inputdominated{} setting common in modern LLM serving. In retrieval-augmented and agentic workflows, the \inputctx{} can include retrieved documents, uploaded files, interaction histories, tool outputs, and other application state~\citep{lewis2020rag,yao2023react,schick2023toolformer,sglang2024}, while individual decoding outputs may span only tens to a few hundred tokens. %

It is worth noting that large inputs also attenuate the amplification mechanism observed in the homogeneous model.
There, eviction cascades are driven by the memory gradient across stages: a completion releases $w_{l_1-1}=l_0+l_1$ tokens of memory while admissions consume only $w_0=l_0+1$, yielding amplification factor $\beta=w_{l_1-1}/w_0=1+O(l_1/l_0)$. Thus, when inputs dominate memory usage, the stage-to-stage memory gradient becomes small, and eviction cascades cause smaller throughput losses. The eviction-free fixed point nevertheless remains unstable for every finite $l_0$, and eviction cascades can still occur.

In the heterogeneous setting, we use the large-input regime to remove the
confounding effect of stage-dependent memory weights and expose the role of
completion timing.  The limiting dynamics are then governed by the arithmetic
structure of the decoding lengths, leading to the $\gcd(l_{1,1},\ldots,l_{1,K})$-based dichotomy derived below.
The numerical experiments in Section~\ref{sec:numerics} confirm that these structural insights persist at finite \inputlens{}, including regimes where the homogeneous throughput loss is substantial.

We first characterize the eviction-free equilibrium. In the saturated-input regime with proportion constraint, an eviction-free fixed point is one in which:
1) No request is evicted before completing decoding; 2) Memory is exactly saturated;
3) Admissions occur at a constant rate $x^*$ per iteration; and
4) The composition of active jobs reflects the proportions $p_k$.
At such a state, each class-$k$ job occupies stages $0,\ldots,l_{1,k}-1$, and there are $p_k x^*$ jobs in each active stage of class $k$. Memory balance therefore requires
\[
\sum_{k=1}^K
p_k x^*
\sum_{j=0}^{l_{1,k}-1}
(l_{0,k} + 1 + j)
= M .
\]
Define the total lifetime memory footprint of a class-$k$ job as
\[
C_k
=
\sum_{j=0}^{l_{1,k}-1}
(l_{0,k} + 1 + j)
=
l_{1,k}\!\left(l_{0,k} + \frac{l_{1,k}+1}{2}\right).
\]
Then, the eviction-free admission rate is
\begin{equation}
\label{eq:xstar_multi}
x^*
=
\frac{M}{\sum_{k=1}^K p_k C_k}.
\end{equation}
At this fixed point, every iteration admits $x^*$ new jobs, completes $x^*$ jobs, and maintains a constant occupancy profile across decoding stages.

As before, a central question is whether this eviction-free fixed point is dynamically stable, which we study next.

\subsection{Heterogeneity, Synchronization, and Stability}

Unlike the homogeneous case, heterogeneity introduces additional structure that can fundamentally alter system dynamics. Requests of different classes are completed at different speeds due to the heterogeneous decoding lengths, so their completion times may or may not align over repeated iterations. Such alignment (or lack thereof) determines whether memory-release bursts synchronize and trigger eviction cascades.

To understand when heterogeneity mitigates eviction-induced instability, we
study perturbations around the eviction-free fixed point. In the saturated-input regime, the admission and execution rules define a deterministic discrete-time dynamical system. In the eviction-free regime under proportional admissions, aggregate admission perturbations satisfy a recurrence that describes how small imbalances propagate.

Let $\Delta^n := x^{n+1}_{0} - x^n_0$ denote the change in aggregate stage-0 admissions at iteration~$n$. A class-$k$ request admitted $m$~iterations ago occupies $l_{0,k}+1+m$ memory tokens if $l_{1,k}>m$. Thus, conservation of total memory in the eviction-free
regime gives the linearized recurrence
\begin{equation}
\label{eq:multi_recurrence}
\sum_{m=0}^{l_{1,K}-1} \Bigl(\sum_{k: l_{1,k} > m} p_k (l_{0,k}+1+m)\Bigr)\Delta^{n-m} = 0.
\end{equation}
The outer sum ranges over possible lags, while the inner sum includes exactly
the request classes that remain active at that lag.
The equality to zero enforces that total memory remains at capacity when no eviction occurs.

By standard linear recurrence theory~\citep{elaydi2005}, perturbations are
linear combinations of modes $z^n$, where $z$ is a root of the
characteristic polynomial
\begin{equation}
F(z) = \sum_{m=0}^{l_{1,K}-1}\!\Bigl(\sum_{k:\, l_{1,k} > m} p_k\,(l_{0,k}+1+m)\Bigr)\,z^{l_{1,K}-1-m}.
\label{eq:characteristic}
\end{equation}
Each root corresponds to one perturbation mode: modes with $|z|<1$ decay, while modes with $|z|>1$ grow.  Thus the linear stability of the
eviction-free equilibrium is determined by whether all roots lie inside the
unit circle.

Under the scaling~\eqref{eq:scaling}, dividing $F$ by $L$ and
sending $L\to\infty$ gives
the limiting polynomial
\[
A(z) = \sum_{m=0}^{l_{1,K}-1}\!\Bigl(\sum_{k:\, l_{1,k} > m} p_k\,r_k\Bigr)\,z^{l_{1,K}-1-m},
\]
whose coefficients depend only on the decoding lengths, arrival proportions, and \inputlenratios{}. Let $R = \sum_{k=1}^K p_k\,r_k$. Rearranging gives
\begin{equation}\label{eq:closed_form_main}
(1-z)\,A(z) \;=\; -R\,z^{l_{1,K}} \;+\; \sum_{k=1}^{K} p_k\,r_k\,z^{l_{1,K}-l_{1,k}}.
\end{equation}

Interestingly, the arithmetic structure of the decoding lengths determines
whether the limiting polynomial has non-trivial root-of-unity modes on the unit
circle.  The key quantity is the greatest common divisor of the decoding
lengths. The following theorem formalizes this dichotomy.

We state the formal theorem for the two-class \commoninput{} case.
Both classes are perpetually backlogged, have fixed admission proportions
\(p\in(0,1)\) and \(q=1-p\), and share a common \inputlen{} \(l_0\). 
LPF eviction is applied by decoding stage, and when a partially
evicted stage contains both classes, each class is removed in proportion to its
occupancy within that stage. The dynamics are the continuous-mass
saturated-input memory-boundary dynamics of Section~\ref{sec:model}. 

\begin{theorem}[Greatest Common Divisor (GCD) Stability Condition]
\label{thm:gcd}
Consider two request classes with common \inputlen{} \(l_0\), decoding lengths
\(l_{1,1}<l_{1,2}\), and admission proportions \(p\) and \(1-p\). Let
\(g=\gcd(l_{1,1},l_{1,2})\). In the saturated-input continuous model, for all
sufficiently large \(l_0\):
\begin{enumerate}[nosep]
\item If $g>1$, then the eviction-free equilibrium is linearly unstable: the nonzero eigenvalues, equivalently the roots of \(F\), consist of \(g-1\) unstable roots with modulus
\(|z| = 1 + \Theta(1/l_0)\) and \(l_{1,2}-g\) stable roots; the full stage-state map also has one zero eigenvalue. Consequently, any sufficiently small perturbation with nonzero projection onto the unstable eigenspace leaves the local region where LPF eviction is inactive in finite time.
\item If $g=1$, then the eviction-free equilibrium is globally asymptotically stable: from any feasible initial active state in the saturated-input continuous model, the system converges to the eviction-free fixed point.
\end{enumerate}
\end{theorem}

Appendix~\ref{app:theorem2} proves the theorem and explains how the same
linear-recurrence and root-structure argument generalizes to \(K\)-class and
\heteroinput{} systems. In particular,
Lemma~\ref{lem:noncoprime_complementary_roots} proves the exact unstable/stable
root count in the non-coprime case, and
Lemma~\ref{lem:generic_finite_exit} turns the unstable eigenmodes into the
finite-exit conclusion. We highlight the main
spectral mechanism here. Let $i$ denote the
imaginary unit, i.e., $i^2=-1$. When $g>1$, the limiting polynomial
$A(z)$ has roots at all non-trivial $g$-th roots of unity.
To see this, consider
$z_j=e^{2\pi i j/g}=\cos(2\pi j/g)+i\sin(2\pi j/g)$ for $j=1,\dots, g-1$, which are non-trivial $g$-th roots of unity, i.e., $z_j^g=1$ and $z_j\neq 1$. Since $g$ divides every decoding
length $l_{1,k}$, we have $z_j^{l_{1,K}-l_{1,k}}=1$ for all $k$, which further implies
\[
(1-z_j)A(z_j)=-R+R=0.
\]
Because $z_j\neq 1$, it follows that $A(z_j)=0$. Thus, the
non-trivial $g$-th roots of unity generate unit-circle modes of the limiting
linearized dynamics, independently of the \inputlenratios{} $r_k$.
For the formal two-class \commoninput{} theorem, the finite-\(l_0\)
characteristic polynomial is an \(O(1/l_0)\) perturbation of the limiting
polynomial. Appendix~\ref{app:theorem2} verifies the required nondegeneracy and
shows that each unit-circle root of \(A\) perturbs to a nearby finite-\(l_0\)
root with outward displacement of order \(O(1/l_0)\), so these modes satisfy
\(|z| = 1 + \Theta(1/l_0) > 1\). The \(K\)-class and \heteroinput{} generalization
uses the same root-of-unity mechanism with weighted survival coefficients; the
drift constants depend on the class weights, but the synchronization mechanism
is unchanged.

When $g=1$, there are no non-trivial common roots of unity.  In this case, the
same polynomial representation, together with a triangle-inequality argument,
shows that all roots lie strictly inside the unit circle.  For the formal
two-class \commoninput{} theorem, the Lyapunov argument in
Appendix~\ref{app:theorem2} proves global convergence to the eviction-free
equilibrium.

The dichotomy reflects the synchronization structure of completion events. When $g>1$, completion times share a common period $g$, so memory releases occur in synchronized bursts. These bursts reinforce periodic oscillations and, for generic local perturbations, push the trajectory out of the region where the eviction-free linearization applies. When $g=1$, completion times occupy different phases of the cycle. Memory is released more evenly over time, which damps oscillations and stabilizes the eviction-free equilibrium.

In the heterogeneous non-coprime case, the instability can be ``milder'' than in the homogeneous system. By Theorem~\ref{thm:gcd}, only the \(g-1\) modes associated with \(g\)-periodic oscillations are unstable, while the remaining modes are damped. These synchronized modes cause generic local perturbations to exit the no-eviction neighborhood. Appendix Section~\ref{sec:pulse_cycle_return} complements this local result by constructing an explicit two-class non-coprime instance in which the post-exit dynamics settle into a period-\(g\) limit cycle with recurrent evictions.

\begin{example}[Coprime decode lengths: convergence]
\label{ex:coprime}
Let $K=2$ with $l_{1,1}=2$, $l_{1,2}=3$, common \inputlen{}
$l_{0,1} = l_{0,2} = l_0 = 50$, arrival proportions
$p_1=p_2=1/2$, and $M=518$. This gives the eviction-free equilibrium admission $x^*=4$.

With $r_1 = r_2 = 1$, the closed form~\eqref{eq:closed_form_main} yields
\[
A(z)=z^2+z+\tfrac12,
\]
whose roots are $z=(-1\pm \mathrm{i})/2$, so
$|z|=1/\sqrt2\simeq0.707$.
Since $\gcd(2,3)=1$, all roots lie strictly inside the unit circle, and perturbations decay geometrically. Figure~\ref{fig:gcd_trajectories}(a)
illustrates this behavior: starting from a $25\%$ over-admission ($x^0_0=5$), the system converges to $x^*$ within $2\%$ by iteration~$8$.
Intuitively, no oscillation period can synchronize completions
of both classes. A period-2 fluctuation in admissions propagates
through class-1 completions (delay $2$) but is phase-shifted by
class-2 completions (delay $3$), causing oscillatory modes to
interfere destructively.
\end{example}

\begin{example}[Non-coprime decode lengths: divergence]
\label{ex:noncoprime}
Let $K=2$ with $l_{1,1}=2$, $l_{1,2}=4$, common \inputlen{}
$l_{0,1} = l_{0,2} = l_0 = 50$, and arrival proportions
$p_1=p_2=1/2$. Setting $M=626$ again yields
$x^*=4$.

Because $\gcd(2,4)=2$, the limiting polynomial $A(z)$
places a root at $z=-1$ on the unit circle.
For finite $l_0=50$, perturbation analysis shows that this root moves outside the unit circle to $|z|\simeq 1.019$.
Figure~\ref{fig:gcd_trajectories}(b) shows the resulting dynamics: the same initial perturbation now produces oscillations whose amplitude grows over time.
The instability arises from period-2 synchronization.
Both class-1 (delay $2$) and class-2 (delay $4$) completions echo a period-2 fluctuation in admissions, reinforcing the oscillation rather than damping it.
\end{example}

\begin{figure}[ht]
\centering
\includegraphics[width=\textwidth]{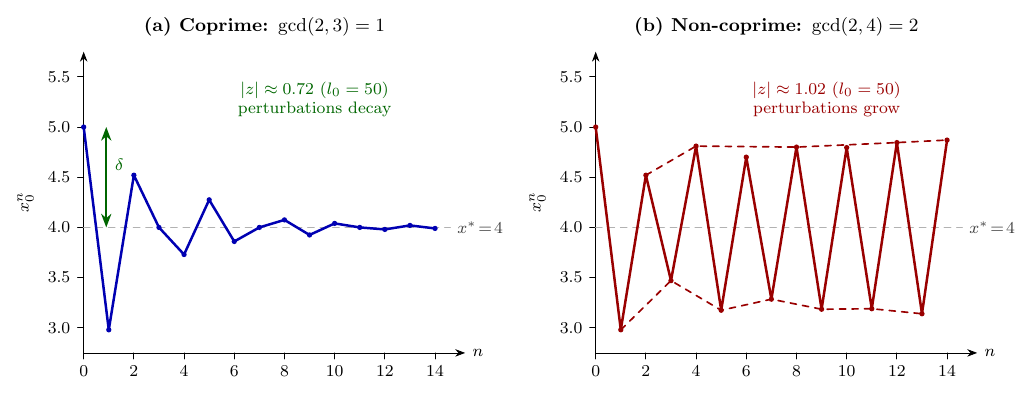}
\caption{Admission trajectories for two multi-class systems ($l_0 = 50$, $p_1 = p_2 = 1/2$, $x^* = 4$, starting at $x^0_0 = 5$). (a)~Coprime decode lengths ($l_{1,1} = 2$, $l_{1,2} = 3$): the limiting polynomial $A(z) = z^2 + z + \tfrac{1}{2}$ has roots of modulus $1/\sqrt{2}$, and for $l_0 = 50$ the dominant root has $|z| \approx 0.72$; the perturbation decays to within 2\% of~$x^*$ by iteration~8. (b)~Non-coprime decode lengths ($l_{1,1} = 2$, $l_{1,2} = 4$, $\gcd = 2$): $A(z)$ has a root at $z = -1$ (the non-trivial second root of unity), which drifts to $|z| \approx 1.02$ for finite $l_0$ (Theorem~\ref{thm:gcd}); oscillations grow geometrically (dashed envelope), driving the system toward eviction. Both trajectories are computed from the finite-$l_0$ recurrence~\eqref{eq:multi_recurrence}.}
\label{fig:gcd_trajectories}
\end{figure}

\begin{remark}[Invariance to \inputlenratios{}]
\label{rem:hetero_prompt}
The GCD synchronization condition depends only on the decoding lengths
$l_{1,1},\ldots,l_{1,K}$ and is invariant to the \inputlenratios{}
$r_1,\ldots,r_K$ and admission proportions $p_1,\ldots,p_K$. Intuitively, in the
\inputdominated{} regime the \inputlens{} affect the amount of memory carried by
each cohort but not the periodic timing of completions. Consequently, the
existence of synchronized unit-circle modes is determined by the arithmetic
relationship among the decoding lengths. Formally, whether the limiting
polynomial has non-trivial root-of-unity modes on the unit circle only depends on
the greatest common divisor of the decoding lengths. The remaining roots, and
therefore the contraction rates and the leading constants that determine the
size of finite-$L$ corrections and the rate at which the asymptotic regime is
approached, still depend on \(r_k\) and \(p_k\).
\end{remark}

\subsection{Finite \Inputlen{} and Stability Threshold}
\label{sec:sensitivity_analysis}

Theorem~\ref{thm:gcd} characterizes stability in the \inputdominated{} limit.
We now ask how large the \inputlen{} must be for a finite system to inherit
the coprime stability predicted by that limit. This quantifies the
``sufficiently large'' input requirement in Theorem~\ref{thm:gcd}. For
clarity, we focus on two request classes with a common \inputlen{}~$l_0$,
coprime decoding lengths:
$l_{1,1} < l_{1,2}$, $\gcd(l_{1,1},l_{1,2})=1$,
and arrival proportions $p$ and $1-p$.

Let $\rho_\infty := \max\{|\alpha|:A(\alpha)=0\}$
be the spectral radius of the limiting polynomial~\eqref{eq:closed_form_main}. Under the coprime condition,
Theorem~\ref{thm:gcd} implies \(\rho_\infty<1\). For finite \(l_0\), however, the roots of the characteristic polynomial drift outward. Write $\varepsilon:=(l_0+l_{1,2})^{-1}$ and let \(\rho(\varepsilon)\) denote the spectral radius of the corresponding characteristic polynomial. At the threshold scale \(\varepsilon=O(l_{1,2}^{-3})\), the perturbation expansion gives
\begin{equation}
\label{eq:finite_prompt_drift_main1}
\rho(\varepsilon)
=
\rho_\infty(1+\varepsilon)
+
O\!\left(l_{1,2}\varepsilon^2\right).
\end{equation}
This implies that finite \inputlen{} erodes the spectral gap created by
coprime decoding lengths. To first order, stability, i.e., $\rho(\varepsilon)<1$, requires
\[
\rho_\infty(1+\varepsilon)<1,
\qquad\text{or equivalently}\qquad
(l_0+l_{1,2})(1-\rho_\infty)\gtrsim 1 .
\]

Define the minimum stable \inputlen{} by
\[
l_{0,\min}
:=
\min\left\{l_0\in\mathbb{Z}_{\ge0}:
\rho\bigl((l_0+l_{1,2})^{-1}\bigr)<1
\right\}.
\]
Here stability refers to Schur stability of the finite-\(l_0\) linearized recurrence, i.e., all roots lie strictly inside the unit disk.
We next characterizes how this threshold scales with the decoding lengths.

\begin{proposition}[Finite-input stability threshold]
\label{prop:finite_prompt_threshold}
Fix \(p\in(0,1)\). In the two-class \commoninput{} setting above, consider a
coprime sequence with \(l_{1,2}\to\infty\) and
\(l_{1,1}/l_{1,2}\to\theta\in[0,1)\). Then the minimum stable \inputlen{} satisfies
\begin{equation}
\label{eq:l0_min_explicit_main1}
l_{0,\min}
\sim
\frac{((1-p)+p\theta)^3}{2\pi^2p(1-p)}
\,l_{1,2}^3 .
\end{equation}
\end{proposition}

Proposition~\ref{prop:finite_prompt_threshold} shows that coprimality is an
asymptotic stabilizing force, but it may require large inputs to be visible
at finite scale. The reason is that the limiting spectral gap shrinks cubically with the longer decoding length. Thus, when decoding lengths are long, the \inputlen{} needed for the coprime system to become stable scales as \(l_{1,2}^3\). Interestingly, the prefactor also has operational meaning: stability is harder to achieve when the two decoding lengths are close, \(\theta\simeq1\), or when the workload mix is highly imbalanced, \(p\simeq0\) or \(p\simeq1\). In both
cases, the desynchronizing effect of heterogeneity is weak.

The proof in Appendix~\ref{app:theorem2} combines four ingredients: the
\finiteinput{} spectral-drift expansion for the two dominant root branches,
whose simplicity is verified there for every fixed \(p\in(0,1)\), the cubic
limiting spectral-gap estimate, the no-stable-islands lemma that rules out
earlier stable windows, and the resulting stability-threshold corollary.
For any fixed coprime pair, the exact finite threshold can be computed by
finding the roots of the finite-\(l_0\) characteristic polynomial. Figure
\ref{fig:spectral_gap_finite_l0} compares these numerical thresholds with the
asymptotic prediction in \eqref{eq:l0_min_explicit_main1}. The figure is
consistent with the outward spectral drift and the cubic scaling of the \finiteinput{}
stability threshold.

\begin{figure}[t]
\centering
\includegraphics[width=\textwidth]{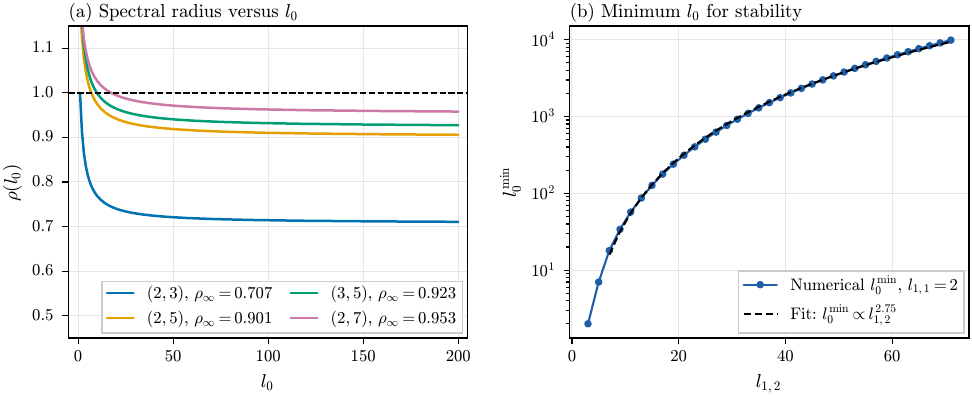}
\caption{Finite-$l_0$ spectral radius for coprime pairs ($p=1/2$). (a)~Spectral radius $\rho(l_0)$ of the characteristic polynomial~$F(z)$ as $l_0$ increases, for four coprime pairs $(l_{1,1},l_{1,2})$. Each curve converges to its limiting value $\rho_\infty$; the dashed line marks $\rho=1$. The pair $(2,7)$ requires $l_0\geq 15$ to enter the stable regime. (b)~Minimum \inputlen{} $l_0^{\min}$ for all eigenvalues to lie inside the unit circle ($l_{1,1}=2$). The log-log fit is consistent with the $\Theta(l_{1,2}^3)$ scaling proved in Corollary~\ref{cor:l0_threshold}.}
\label{fig:spectral_gap_finite_l0}
\end{figure}

\section{Eviction-Aware Control Policies}
\label{sec:policies}

Theorems~\ref{thm:fcfs} and~\ref{thm:gcd} uncover a single structural source of instability of the eviction-free equilibrium: positive feedback between memory release and admission under progressive memory growth. How this feedback manifests depends on workload composition. In homogeneous systems, it emerges as admission pulses that amplify through the decoding pipeline; in heterogeneous systems, under the large input scaling, arithmetic synchronization determines whether those pulses reinforce or dissipate.
Based on these insights, in this section, we propose eviction-aware control policies.

The policies below are intentionally stylized. Each targets a specific manifestation of the underlying feedback mechanism and illustrates how the structural insights developed earlier translate into concrete design principles. They are not intended as complete production schedulers, but rather as minimal interventions that mitigate eviction at its structural source.

\subsection{Rate-Limited Admission}

The instability results above identify a common mechanism across homogeneous
workloads, non-coprime heterogeneous workloads, and coprime workloads below the
\finiteinput{} stability threshold: transient over-admission. When aggregate
admission \(a^n\) exceeds the eviction-free sustainable rate, the excess mass
enters the decoding pipeline, and its memory footprint grows over subsequent
stages. The resulting amplification can push the system into eviction and
even initiate a support-loss cascade.
A natural remedy is therefore to cap aggregate admission at the eviction-free
equilibrium rate. Under proportional class admission, the policy admits
\[
a^n = \min\!\left(Q^n,\,
\frac{S^n}{\sum_{k}p_k(l_{0,k}+1)},\, x^*\right),
\qquad
a_k^n=p_k a^n,
\]
where \(x^*\) is the eviction-free equilibrium admission rate derived in
Sections~\ref{sec:single_class} and~\ref{sec:multi_class},
\(Q^n:=\sum_k Q_k^n\) is the aggregate waiting demand,
$S^n:=M-\sum_{k,j}w_{k,j}\hat x^n_{k,j}$
is the GPU memory available after execution and eviction at iteration~\(n\).
The denominator \(\sum_k p_k(l_{0,k}+1)\) is the stage-0 memory consumed by one aggregate unit of proportional admissions. Thus the middle term is the largest aggregate
admission that can be placed without violating the memory constraint. This cap removes the admission pulses that drive the instability while preserving the eviction-free throughput whenever sufficient demand and memory are available.

The no-over-admission property is immediate from the definition. During the
first few iterations, evictions may still occur because requests admitted before
the cap was imposed can continue to grow in memory. After at most \(l_1\)
iterations in the single-class system, all such pre-policy requests have either
completed or been evicted; from that point on, every active cohort was admitted
under the cap and therefore cannot exceed the eviction-free profile. The
analogous transient bound is \(\max_k l_{1,k}\) in the multi-class setting.

Our analysis is conducted in the saturated-input regime, where a backlog of requests is always available for admission. In practice, systems rarely operate persistently in this regime. Instead, demand is typically below capacity, although fluctuations and temporary surges may occasionally push the system toward saturation. Nevertheless, the analysis identifies the critical admission boundary that mitigates eviction once the system approaches capacity. The proposed policy adapts naturally to non-saturated conditions. When demand is low ($Q^n < x^*$), the threshold is non-binding and requests are admitted immediately, preserving low latency. In particular, the policy does not require waiting to accumulate a large batch of requests, which would unnecessarily increase latency. When demand surges and the system approaches saturation, the rate cap activates automatically and limits admissions to the sustainable rate $x^*$, mitigating admission pulses and the eviction cascades they would otherwise trigger.

\subsection{Request Mixing for Heterogeneous Systems}

When heterogeneous request classes are available, Theorem~\ref{thm:gcd} suggests a complementary control mechanism that operates at the routing level. Rather than assigning requests to GPUs arbitrarily or grouping identical requests together, the system can route requests so that each GPU processes a diversified mixture of decoding lengths whose greatest common divisor equals one. The motivation follows directly from the instability mechanism identified earlier. In non-coprime systems, decoding completions synchronize at multiples of
$g = \gcd(l_{1,1},\ldots,l_{1,K})$,
producing periodic bursts of memory release that amplify admission fluctuations and trigger eviction. By contrast, when the decoding lengths assigned to a GPU are coprime, completion events occur at desynchronized phases, smoothing memory release over time. %

Serving systems such as vLLM and SGLang show that batching structure and cache locality are central to efficient LLM serving~\citep{vllm2023,sglang2024}. In deployments that route similar request classes to the same device for these reasons, specialization can concentrate workloads with identical or commensurate decoding lengths, creating precisely the synchronization conditions that lead to eviction cycles.
Our analysis does not imply that requests with drastically different decoding
lengths must be mixed.  Stability can often be improved through mild
diversification: assign together request classes with comparable decoding lengths
but avoid nontrivial common periods among their lengths. Such routing preserves
most of the batching and memory-efficiency benefits of specialization while reducing the synchronization of completion events that drives eviction-induced instability.

\subsection{Design Implications}

The structural results above suggest several practical guidelines for the design and monitoring of memory-constrained LLM serving systems.

\paragraph{Arithmetic structure as a stability diagnostic.}
The stability of the system depends critically on the arithmetic structure of decoding lengths. When requests assigned to a GPU share a common divisor, completion events can synchronize and trigger eviction cycles. In contrast, mixtures of coprime decoding lengths naturally desynchronize completions and stabilize the eviction-free equilibrium under the model conditions. Monitoring the effective greatest common divisor of active workloads therefore provides a model-informed structural diagnostic for eviction risk.

\paragraph{Engineering stability through mild diversification.}
Stability can often be improved by adding modest heterogeneity to decoding
lengths. For example, if service-level agreements define discrete output-length
tiers such as \(\{128,256,512\}\), all of which share \(\gcd=128\), small
adjustments (e.g., \(\{127,255,509\}\)) restore coprimality and remove the exact
GCD resonance in the \inputdominated{} model. At finite \inputlens{}, such
adjustments mitigate arithmetic synchronization but do not by themselves
guarantee stability; the \finiteinput{} threshold in
Section~\ref{sec:sensitivity_analysis} still matters.
This provides a practical low-overhead design lever.

Operational practices that route similar request classes to dedicated GPUs can improve batching efficiency and cache locality. At the same time, incorporating a modest level of diversification in the requests assigned to each GPU can enhance system stability by mitigating the synchronization of completion events that leads to eviction cycles. This suggests that efficient and stable operation can often be achieved by balancing specialization with mild heterogeneity in decoding lengths.

\paragraph{Eviction as an early warning signal.}
The transition from zero eviction to persistent eviction is a structural transition (Theorem~\ref{thm:fcfs}), rather than a gradual degradation. Once the sustainable admission rate is exceeded, the system can converge to an eviction-driven limit cycle with sustained throughput loss. Persistent eviction therefore can serve as an indicator that the system is operating beyond the eviction-free operating region predicted by the model.

From a control perspective, this observation also highlights the importance of admission regulation. Monitoring eviction rates provides a practical signal for when admission should be moderated to remain within the eviction-free operating region. Policies such as the rate-limited admission rule described above can then act as safeguards, mitigating transient demand surges that would otherwise push the system into an eviction cycle.

\section{Numerical Experiments}
\label{sec:numerics}

The experiments examine whether the stability mechanisms identified in the deterministic saturated-input model persist in more realistic serving environments. We first use a model-based simulator, implemented directly from the memory and admission dynamics in Section~\ref{sec:model}, to test whether stochastic arrivals in the non-saturated-input regime exhibit the corresponding empirical queue-growth transition near the worst-cycle throughput. We then use Vidur, a high-fidelity LLM inference serving simulator, together with real-GPU experiments, to test whether practical serving-system dynamics are consistent with the synchronization mechanism. Across these settings, request mixing reduces synchronized completions, and rate-limited admission reduces eviction cascades.

Vidur is a validated large-scale LLM-serving simulator whose reported latency and throughput errors are below 5\% relative to real systems~\citep{agrawal2024vidur}. Unless otherwise noted, the Vidur experiments use the Sarathi scheduler~\citep{sarathi2024}, Meta-Llama-3-8B, and a single NVIDIA A100 GPU. Sarathi greedily admits pending requests whenever memory is available. Our rate-limited variant adds a per-iteration admission cap based on the eviction-free admission rate. When the resident KV cache later exceeds capacity, Vidur records an eviction event: the affected request loses its current KV cache and must be recomputed from prefill. We apply the same LPF eviction priority as in the model and interpret these events as the serving-system counterparts of model evictions.
For real-GPU hardware, we run SGLang~\citep{sglang2024}, which uses continuous batching and
therefore aligns with the model's per-iteration dynamics, on an NVIDIA
A800 80\,GB GPU with Qwen2.5-1.5B-Instruct. Appendix~\ref{app:numerical_details} provides additional model-based, Vidur, and real-GPU experiment details.

\subsection{Model-based simulation under open arrivals}
\label{subsec:gcd_effect}

In this section, we evaluate the discrete-time model under different load levels. We build a simulator that follows model dynamics in Section~\ref{sec:model} and replaces the saturated-input assumption with a Poisson arrival process. We vary the arrival rate $\lambda$, interpreted as the average number of new requests arriving per period, and examine the resulting system dynamics.

The saturated-input analysis suggests the following open-arrival numerical
validation under baseline admission. Recall that $\bar{T}_{l_1-1}$ denotes the
throughput associated with the worst-case limit cycle, while $x^*$ denotes the
throughput associated with the eviction-free equilibrium. Since
$\bar{T}_{l_1-1}<x^*$, there exists a nontrivial interval of arrival rates
between these two values. Loads in this interval would appear sustainable if one
only considered the eviction-free equilibrium, since they remain below $x^*$.
However, in the saturated-input model that equilibrium is unstable, and the
dynamics are attracted to the asymptotically stable maximal-eviction cycle, which can
sustain only $\bar{T}_{l_1-1}$ requests per period. We therefore test whether the
same throughput bottleneck appears as queue growth under stochastic arrivals.

The open-arrival simulations are consistent with this mechanism. We treat a run as empirically stable when the waiting queue remains small and shows no sustained upward drift over the simulation horizon. When $\lambda<\bar{T}_{l_1-1}$, evictions are transient and the waiting queue stays small. When $\lambda> \bar{T}_{l_1-1}$, even if $\lambda<x^*$, the simulated system settles around the worst-cycle throughput $\bar{T}_{l_1-1}$, while excess arrivals accumulate and the waiting queue grows.
\Cref{fig:lambda_sweep} shows this transition directly, where $Q^n$ denotes the number of waiting requests at iteration $n$. For loads below $\bar{T}_{l_1-1}$, the throughput is equal to $\lambda$ and $Q^n$ remains close to zero. For loads above $\bar{T}_{l_1-1}$, including those in the interval $(\bar{T}_{l_1-1},x^*)$, throughput is equal to $\bar{T}_{l_1-1}$. The system experiences persistent evictions and a growing $Q^n$. Thus, in this simulated open-arrival setting, baseline admission begins to accumulate backlog before the nominal eviction-free capacity is reached.

\begin{figure}[ht]
    \centering
    \includegraphics[width=0.98\linewidth]{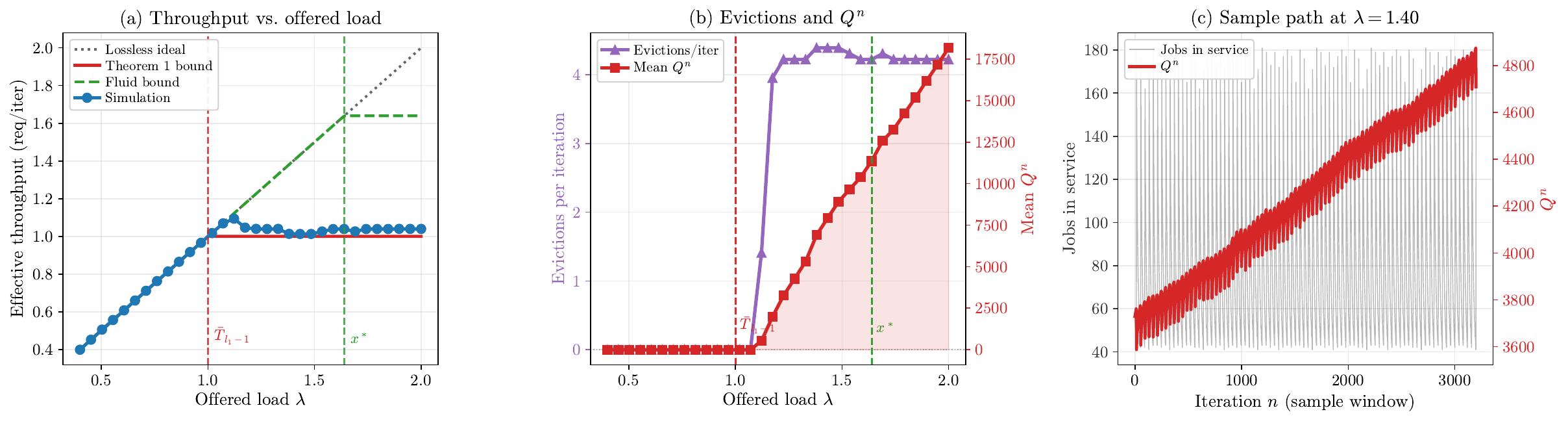}
\caption{Open-system behavior under Poisson arrivals
    in the model-based simulator. With \(M=2000\), \(l_0=10\), and \(l_1=40\),
    the worst-cycle service rate is \(\bar{T}_{l_1-1}=1.00\), while the
    eviction-free fluid optimum is \(x^*=1.64\).
    \textbf{(a)} Throughput departs from the lossless line once
    \(\lambda\) crosses \(\bar{T}_{l_1-1}\), despite \(\lambda<x^*\).
    \textbf{(b)} Evictions per iteration and mean queue length \(Q^n\) are
    negligible below \(\bar{T}_{l_1-1}\), then become positive above it.
    \textbf{(c)} A raw sample path at
    \(\lambda=1.40\in(\bar{T}_{l_1-1},x^*)\) shows oscillating jobs in service and
    steadily growing \(Q^n\). Detailed settings and additional GCD simulations appear
    in Appendix~\ref{app:numerical_details}.}
    \label{fig:lambda_sweep}
\end{figure}

\subsection{Mixing}
\label{subsec:request_mixing}

In this subsection, we analyze the effect of request mixing. Although Theorem~\ref{thm:gcd} is derived in the large-input limit and gives a sharp stabilization criterion for the two-class \commoninput{} setting when the combined decoding lengths are coprime, the underlying mechanism suggests a broader phenomenon. Routing need not eliminate common periodicity completely in order to improve performance. By pooling request classes with different decoding lengths, mixed routing reduces decoding-length homogeneity within each serving pool, which weakens the synchronization of completions. Any reduction in synchronization can weaken the completion pulses that drive large memory-release and memory-growth cycles, thereby reducing the likelihood of eviction cascades. We therefore expect request mixing to remain beneficial in broader serving settings, even when the combined decoding lengths are not coprime or the input length is not large enough.

We first test this mechanism in Vidur. For these fixed-workload Vidur and
real-GPU experiments, requests are generated according to the stated class
proportions and submitted at a sufficiently high rate at the beginning of the
run. The system therefore maintains a persistent waiting queue, and the
scheduler admits from that queue whenever memory is available. This construction
approximates the saturated-input regime analyzed in the theory; stochastic
open-arrival behavior is studied separately in
\Cref{subsec:gcd_effect} and Appendix~\ref{app:numerical_details}.

Table~\ref{tab:sarathi_mixing_both_settings} reports the Vidur results for two main configurations, each consisting of two request classes. In the first configuration, the decoding lengths are 20 and 21, so mixing makes the combined decoding lengths coprime. In the second configuration, the decoding lengths are 100 and 125. Mixing does not eliminate common periodicity in this case: the combined GCD remains 25. However, this common period is substantially smaller than the periodicity induced when each class is served in isolation. For each configuration, we compare segregated serving, where the two request classes are served separately on two devices, with mixed serving, where the same two-device budget serves the pooled workload. The workload size, \inputlen{}, scheduler, and hardware budget are held fixed within each comparison; only the routing of request classes is changed. SegAvg denotes the average across the two segregated devices. We report the number of evictions, request-level mean latency, and throughput measured as completed requests per second. In the coprime configuration, mixing nearly eliminates evictions over the experiment horizon, a 94.2\% reduction, and improves both latency and throughput. In the partial-GCD configuration, the reduction in evictions is smaller, but we still see mixing leads to fewer evictions, lower mean latency, and higher throughput. Appendix~\ref{app:numerical_details} contains the corresponding model-based simulation figures, including memory trajectories for the Vidur decoding-length configurations, and additional mixing configuration results.

\begin{table}[ht]
\centering
\caption{Request mixing with 20k requests and common \inputlen{} $l_0=512$.
}
\label{tab:sarathi_mixing_both_settings}
\small
\begin{tabular}{@{}lcccccc@{}}
\toprule
\textbf{Configuration} & \textbf{GCD} & \multicolumn{2}{c}{\textbf{Evictions}} & \multicolumn{2}{c}{\textbf{Latency (s)}} & \textbf{Throughput} \\
\cmidrule(lr){3-4} \cmidrule(lr){5-6}
& & Value & Improv. & Value & Improv. & (req/s) \\
\midrule
\quad Node 1 & 20 & 908 & --- & 107.8 & --- & 91.7 \\
\quad Node 2 & 21 & 931 & --- & 114.1 & --- & 87.0 \\
\quad SegAvg & --- & 919 & --- & 111.0 & --- & 89.3 \\
\textbf{Mixed} & \textbf{1} & \textbf{53} & \textbf{--94.2\%} & \textbf{95.1} & \textbf{--14.3\%} & \textbf{105.8} \\
\midrule
\quad Node 1 & 100 & 3296 & --- & 367.9 & --- & 27.3 \\
\quad Node 2 & 125 & 4121 & --- & 473.4 & --- & 21.2 \\
\quad SegAvg & --- & 3708 & --- & 420.6 & --- & 24.2 \\
\textbf{Mixed} & \textbf{25} & \textbf{2449} & \textbf{--34.0\%} & \textbf{390.4} & \textbf{--7.2\%} & \textbf{25.6} \\
\bottomrule
\end{tabular}
\vspace{1mm}
\parbox{0.95\linewidth}{\footnotesize Note. SegAvg is the average of the two segregated devices. Mixed uses the same two-device budget with the request classes pooled. Throughput is completed requests per second.}
\end{table}

The same pattern appears in experiments on real GPU hardware.
\Cref{tab:sglang_mixing} reports the number of evictions, mean latency, and throughput in one such study.
Node~1 serves an equal mix of two request classes with decoding lengths $5$ and $10$ (GCD\(=5\));
Node~2 serves an equal mix of two request classes with decoding lengths $6$ and $12$ (GCD\(=6\)); the mixed configuration pools all four request classes, with decoding lengths
\(l_1\in\{5,6,10,12\}\) in equal proportions, giving GCD\(=1\). The comparison again uses the same workload size, \inputlen{}, scheduler, and two-GPU budget: segregated routing assigns one configuration to each GPU, while mixed routing serves the pooled workload on the same two GPUs. Relative to the average of the two segregated configurations (Nodes 1 and 2), mixing reduces evictions by 97.7\%, lowers mean latency by 18.2\%, and increases request throughput by 30.5\%. The time-series diagnostics in Appendix~\ref{app:numerical_details} further demonstrate that the segregated configurations produce periodic eviction bursts, while mixing suppresses them throughout the run. Appendix~\ref{app:benchmark_mixing} reports further GPU experiments in which requests are drawn from standard LLM evaluation benchmarks.

\begin{table}[ht]
\centering
\caption{Request mixing on real GPU hardware with $60$k requests and common \inputlen{} $l_0=300$.}
\label{tab:sglang_mixing}
\small
\begin{tabular}{@{}lcccccc@{}}
\toprule
\textbf{Configuration} & \textbf{GCD} &
\multicolumn{2}{c}{\textbf{Evictions}} &
\multicolumn{2}{c}{\textbf{Latency (s)}} &
\textbf{Throughput} \\
\cmidrule(lr){3-4} \cmidrule(lr){5-6}
& & Value & Improv. & Value & Improv. & (req/s) \\
\midrule
\quad Node 1  & 5   & 422 & ---          & 5.44 & ---          & 2{,}731 \\
\quad Node 2  & 6   & 372 & ---          & 5.52 & ---          & 3{,}264 \\
\quad SegAvg  & --- & 397 & ---          & 5.48 & ---          & 2{,}998 \\
\textbf{Mixed} & \textbf{1} & \textbf{9} & \textbf{$-$97.7\%} & \textbf{4.48} & \textbf{$-$18.2\%} & \textbf{3{,}911} \\
\bottomrule
\end{tabular}
\vspace{1mm}
\parbox{0.95\linewidth}{\footnotesize Note. SegAvg is the average of the two segregated GPUs. Mixed uses the same two-GPU budget with all request classes pooled. Throughput is completed requests per second.}
\end{table}

\subsection{Rate-limited admission}
\label{subsec:rate_limited_admission}

Greedy admission fails because it reacts only to current memory feasibility and ignores the future KV cache growth generated by newly admitted requests. Rate-limited admission addresses this problem by using \(x^*\), the eviction-free equilibrium admission/completion rate, as a per-iteration admission cap. This cap is large enough to keep the system operating near the nominal eviction-free capacity, avoiding an overly conservative policy that leaves GPU capacity idle. At the same time, it prevents the persistent overshoot that would trigger eviction cascades.

The Vidur experiment in \Cref{tab:rate_limit_vidur} uses decoding
lengths \(l_1\in\{100,125,200,250\}\) and an arrival rate of 300 requests per second. Baseline admission accepts every request that fits in currently available memory, which leads to periodic memory overflow and eviction cascades as KV caches grow. Rate-limited admission instead imposes a per-iteration cap of five new requests. Because Vidur implements an integer per-cycle admission cap rather than the continuous admission rate in the fluid model, we select this cap by a grid search over \(\{1,2,\ldots,10\}\), guided by the equilibrium rate \(x^*\). This rate-limited admission eliminates evictions over the experiment horizon, lowers mean latency by 18.9\%,
and increases throughput by 28.3\%.

\begin{table}[htbp]
\centering
\caption{Rate-limited admission. The workload has 20k requests, common input
length \(l_0=512\), decoding lengths \(l_1\in\{100,125,200,250\}\), and arrival
rate 300 req/s. The integer admission cap is selected from \(\{1,\ldots,10\}\).}
\label{tab:rate_limit_vidur}
\small
\begin{tabular}{@{}lccc@{}}
\toprule
\textbf{Scheduler} & \textbf{Evictions} & \textbf{Latency (s)} & \textbf{Throughput (req/s)} \\
\midrule
Sarathi & 2,232 & 234.2 & 42.7 \\
\textbf{Sarathi + rate-limited admission} & \textbf{0} & \textbf{189.9} & \textbf{54.8} \\
\midrule
Improvement & \(-\)100.0\% & +18.9\% & +28.3\% \\
\bottomrule
\end{tabular}
\end{table}

The sample path in \Cref{fig:rate_limit_memory_dynamics_vidur} makes the mechanism visible.
Under Sarathi, memory repeatedly rises above the capacity line, and each overflow is resolved through forced evictions. Under rate-limited admission, memory remains in a narrow band near capacity without crossing it over the experiment horizon.

\begin{figure}[ht]
\centering
\includegraphics[width=0.95\linewidth]{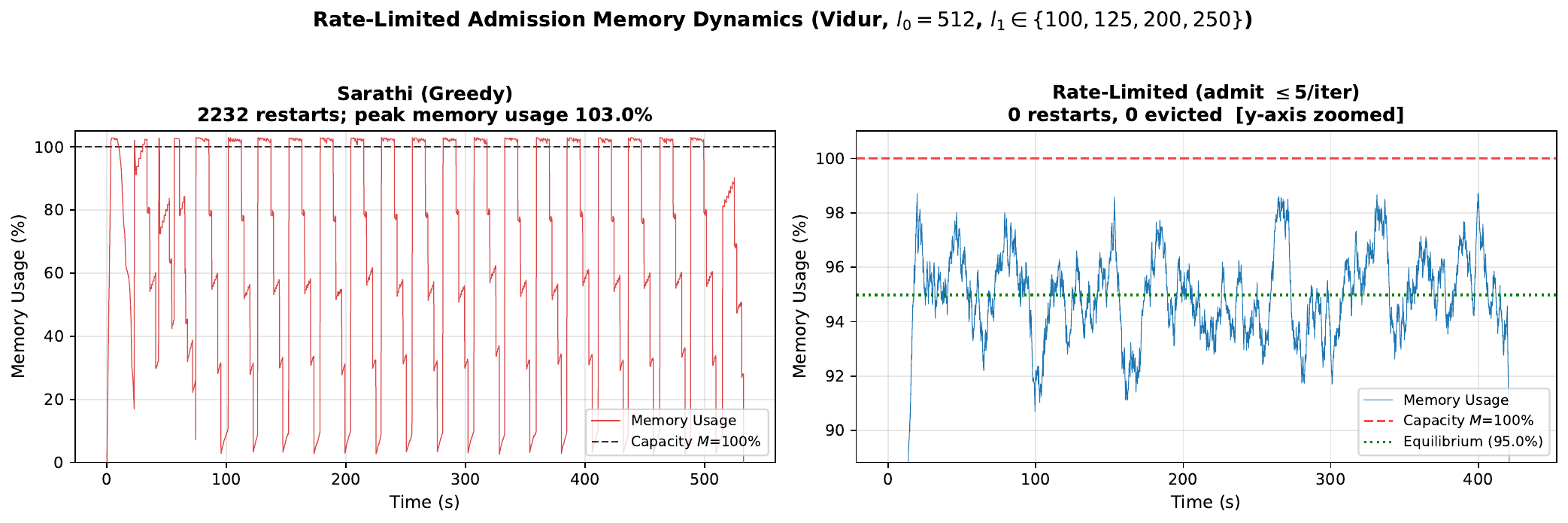}
\caption{Rate-limited admission stabilizes memory dynamics in the Vidur run. Greedy admission
exceeds the \(M=100\%\) capacity line in 25.2\% of scheduling cycles
and incurs 2,232 evictions. Rate-limited admission keeps memory below capacity
and eliminates eviction over this horizon.} 
\label{fig:rate_limit_memory_dynamics_vidur}
\end{figure}

\FloatBarrier

\section{Conclusion}\label{sec:conclusion}

This paper develops a discrete dynamical model for memory-constrained
LLM serving systems in which requests consume increasing memory during service. This progressive KV-cache accumulation, absent from classical fixed-resource
queueing models, creates a feedback loop among admission, memory growth, and
eviction that can destabilize highly utilized systems.

Our analysis identifies two main mechanisms. First, in homogeneous systems, small imbalances in admission propagate through
the decoding pipeline, are amplified by memory growth, and eventually trigger
eviction cascades. Except for exceptional balanced orbits, the system is
attracted to the maximal-eviction limit cycle, which can incur substantial
throughput loss. Second, in heterogeneous systems, stability is governed by synchronization. In the formal two-class \commoninput{} setting, and through the extension argument for multiple classes and \heteroinput{} lengths, coprime decoding lengths desynchronize memory release and stabilize the eviction-free equilibrium under the \inputdominated{} scaling, whereas non-coprime lengths allow synchronized completion peaks that generate persistent oscillatory modes.

These results show that congestion in LLM serving is not only arrival-driven: it can be induced by the service process itself. The proposed policies, rate-limited admission and request mixing, demonstrate that the mechanisms
identified by the theory are actionable. By regulating admission intensity or breaking synchronization across request types, simple interventions can mitigate eviction, improve throughput, and keep the system close to efficient operation.

A natural direction for future research is to incorporate stochastic decoding
lengths. The deterministic-length model studied here isolates the mechanisms
of memory growth and synchronization, but real LLM outputs are variable and
only partially predictable at admission. This uncertainty may affect stability
in two opposing ways: random variation can desynchronize completions and damp
eviction cycles, while long-tail outputs can create unexpected memory pressure
and push the system toward overflow. Extending the analysis to random service
horizons would clarify how output-length uncertainty shifts stability
boundaries and how admission policies should balance utilization against
residual length risk.

\bibliographystyle{plainnat}
\bibliography{references}
\appendix
\section{Proof of Theorem~\ref{thm:fcfs}}
\label{app:theorem1}

We prove Theorem~\ref{thm:fcfs}: under baseline admission at the memory
boundary, the eviction-free fixed point and all intermediate eviction-level
cycles are unstable, while the level-$(l_1-1)$ maximal-eviction cycle is
asymptotically stable on the memory boundary outside the exact nonmaximal
capture set. Its basin contains every memory-boundary initial state except
those captured exactly by a nonmaximal balanced orbit. Within each fixed set of
occupied stages, this exceptional set has Lebesgue measure zero.

\begin{figure}[ht]
\centering
\includegraphics[width=\textwidth]{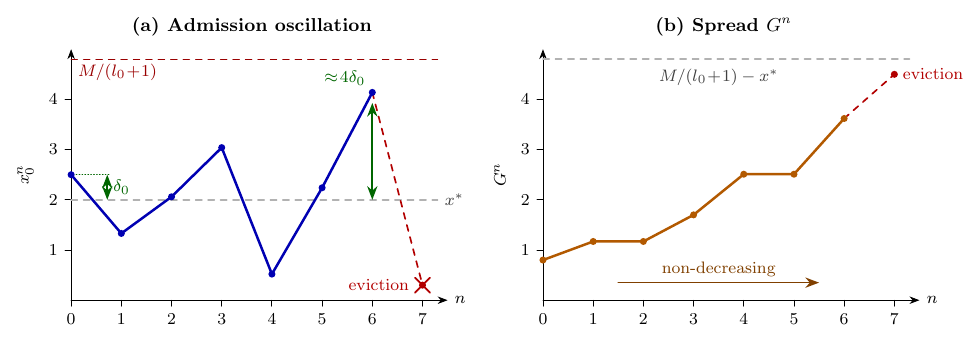}
\caption{Instability trajectory for $l_0 = 2$, $l_1 = 3$, $M = 24$, starting from a small perturbation ($\delta_0 = 0.5$) above the equilibrium $x^* = 2$. (a)~The admission count $x^n_0$ oscillates with geometrically growing amplitude (spectral radius $|z| = \sqrt{5/3} \approx 1.29$). After 7 iterations, the required admission turns negative---memory is exhausted by surviving requests alone---triggering eviction. (b)~The spread $G^n = \max_j x^n_j - \min_j x^n_j$ is monotonically non-decreasing. The spread's inevitable growth drives the system to eviction regardless of the perturbation direction. }
\label{fig:instability}
\end{figure}

\subsection{Stability Analysis}

We split the proof into a pre-eviction stage, where the state is described by
physical stages, and a post-eviction stage, where the correct state description
is a block-end sampled chain of relative positions.

\textbf{Definition (Eviction-free spread).} In the eviction-free regime, define
\[
G^n:=\max_j x^{n}_{j}-\min_j x^{n}_{j}.
\]
Write
\[
\bar x^n:=\max_j x^{n}_{j},
\qquad
\underline x^n:=\min_j x^{n}_{j}.
\]

\begin{appalemma}[Monotonicity before eviction]\label{lm:A1}
Under the saturated-input continuous model of Theorem~\ref{thm:fcfs}, while a
memory-boundary trajectory remains eviction-free, \(G^{n+1}\ge G^n\).
\end{appalemma}

\noindent The proof is deferred to Appendix~\ref{app:aux_lemma_proofs}.

\begin{appalemma}[Finite-time eviction-free breakdown and first support loss]\label{lem:A2}
Assume \(l_1\ge2\). Under the saturated-input continuous model of
Theorem~\ref{thm:fcfs}, for a non-fixed memory-boundary initial state with
\(G^0>0\), the trajectory cannot remain eviction-free forever. The first
memory overflow triggers an LPF eviction; in the relative-position
representation below, this creates the first eviction-induced support loss and
moves the trajectory to the next lower live-support level. In particular, the
eviction-free fixed point is unstable.
\end{appalemma}

\noindent The proof is deferred to Appendix~\ref{app:aux_lemma_proofs}.

After the first eviction-induced support loss, physical-stage indices lose their tracking power: eviction reshuffles which stages carry mass at each iteration. The right coordinate system tracks \emph{relative positions} in the circular pipeline, samples the state only at \emph{block-ends} (the iterations where live positions complete), and normalizes masses by \emph{predecessor gaps} (the number of stages each live position absorbs).

\noindent\textbf{Relative positions.} Fix labels
\(P_0,\ldots,P_{l_1-1}\) on the circular pipeline. At time \(n\), position
\(P_r\) occupies stage \(r+n \pmod{l_1}\). Let \(z^{n}_{r}\) denote the mass
carried by \(P_r\) at $n$, so
\[
x^{n}_{j}=z^{n}_{(j-n)\bmod l_1}.
\]
When a cohort completes and the released memory funds a new admission at
stage~0, the new stage-0 mass inherits the completing cohort's label. Thus, the
labels track relative cohorts rather than physical stage indices.

\begin{appalemma}[Irreversibility of lost positions]\label{lem:A3}
Under the
saturated-input continuous model of Theorem~\ref{thm:fcfs}, if
\(z^{n_0}_{r}=0\) for some position \(P_r\), then \(z^{n}_{r}=0\) for every
\(n\ge n_0\).
\end{appalemma}

\noindent The proof is deferred to Appendix~\ref{app:aux_lemma_proofs}.

Define the \emph{live set} as
\begin{equation}\label{eq:live_set}
L_n:=\{r:z^{n}_{r}>0\}.
\end{equation}
Once a relative position becomes empty, it never revives: when that dead position later cycles back to stage~0, it carries no mass and therefore frees no completion memory, while the surviving positions continue to grow their footprints. No new mass enters the dead position. Hence $L_{n+1}\subseteq L_n$, and the live set can only shrink.

\paragraph{Eviction blocks and block-ends.}
Consider a trajectory with $m=|L_n|\ge 2$ live positions. Within each period of $l_1$ iterations, the dynamics decompose into two phases.
An \emph{eviction block} is a maximal contiguous run of iterations during which eviction removes excess requests but no completions occur (i.e., only dead positions pass stage~$l_1{-}1$). The \emph{block-end} is the first iteration after the eviction block in which a live position completes and new requests are admitted. At block-end, the live/dead pattern returns to the same configuration it had at the previous block-end, making block-ends the natural sampling points for the dynamics (Figure~\ref{fig:block_decomposition}a).

\begin{example}[Block decomposition]
\label{ex:block_decomposition}
Consider the adjacent level-2 orbit for $l_0=2$, $l_1=4$, $M=63$:
\[
(17,3,0,0)\to(0,12,3,0)\to(0,0,9,3)\to(3,0,0,9)\to(17,3,0,0).
\]
The transitions $n=0\to 1\to 2$ form the eviction block: at each step, execute-and-shift produces a memory overshoot corrected by evicting from the most recently admitted stage ($5$~requests at $n{=}0{\to}1$, then $3$ at $n{=}1{\to}2$). No completions occur because only dead positions pass stage~$3$. The block-ends are $n=2$ and $n=3$, where live positions complete and new requests enter.

At $n=2$ (live stages $\{2,3\}$): predecessor gaps $g=(3,1)$, masses $(9,3)$, normalized masses $u=(9/3,\,3/1) = (3,3)$, spread $G^{\mathrm{sc}}=0$. At $n=3$ (live stages $\{0,3\}$): gaps swap to $g=(1,3)$, masses $(3,9)$, normalized $u=(3/1,\,9/3) = (3,3)$, spread $G^{\mathrm{sc}}=0$. The orbit is balanced.
\end{example}

\begin{figure}[t]
\centering
\includegraphics[width=\textwidth]{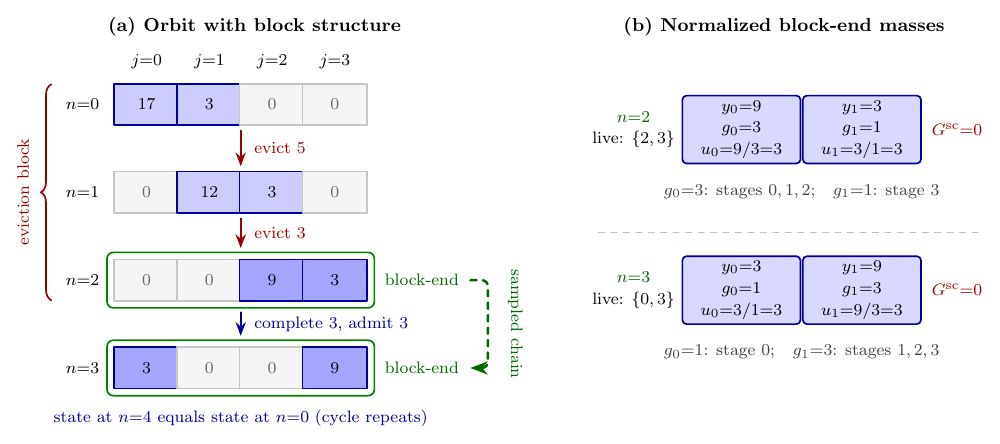}
\caption{Block decomposition for the level-2 adjacent orbit ($l_0 = 2$, $l_1 = 4$, $M = 63$). (a)~The orbit consists of an \emph{eviction block} ($n = 0 \to 2$), during which eviction removes excess requests but no completions occur, followed by a \emph{renewal phase} ($n = 2 \to 0$) in which completions free memory and new requests are admitted. The block-end phases (green borders) are the states immediately after the eviction block settles. (b)~At each block-end, the predecessor gap $g_h$ is the number of stages crossed from the previous live position to the $h$-th one (inclusive, wrapping modulo $l_1$), and the normalized mass is $u_h = y_h/g_h$. The block-end spread $G^{\mathrm{sc}} = \max_h u_h - \min_h u_h$ is non-decreasing at block-ends; here all normalized masses agree ($G^{\mathrm{sc}} = 0$), so the orbit is balanced.}
\label{fig:block_decomposition}
\end{figure}

\paragraph{Predecessor gaps.}
Fix an interval over which the live-position set is unchanged, with
\(m\ge 2\) live positions listed in circular order as
\(r_0,\dots,r_{m-1}\). The \emph{predecessor gap} of the \(h\)-th live
position is
\[
g_h:=(r_h-r_{h-1})\bmod l_1,
\qquad h=0,\dots,m-1,
\]
with indices modulo \(m\). Thus \(g_h\) is the number of service stages crossed
from the previous live position to the current one, with wrap-around
modulo~\(l_1\). In particular,
\(g_h\ge 1\) and \(\sum_h g_h=l_1\).

Write \(S_h:=\sum_{t=0}^h g_t\) for the cumulative gap.
At a block-end, the \(h\)-th live position occupies physical stage \(S_h-1\),
so the live positions sit at stages
\[
g_0-1,\quad g_0+g_1-1,\quad \dots,\quad l_1-1.
\]
The oldest live position is \(y_{m-1}\); it sits at the final stage
\(l_1{-}1\) and is about to complete.
A \emph{block-end sampled state} is described by the live-mass vector
\(Y=(y_0,\dots,y_{m-1})\).
Here \(y_h\) is the number of requests located at the \(h\)-th live
position in the circular ordering, equivalently, the mass at
physical stage \(S_h-1\).
At stage \(S_h-1\), each request occupies \(l_0+S_h\) tokens.
The memory boundary in these coordinates is therefore
\begin{equation}\label{eq:block_boundary}
M=\sum_{h=0}^{m-1}(l_0+S_h)y_h.
\end{equation}

\begin{appaproposition}[Exact block-end map]\label{prop:A4}
Starting from a block-end
sampled state \(Y=(y_0,\dots,y_{m-1})\), the next state of the embedded
block-end chain, conditional on the same support pattern, is
\(Y'=(y'_0,\dots,y'_{m-1})\) with
\[
y'_{h+1}=y_h,\qquad h=0,\dots,m-2,
\]
and
\begin{equation}\label{eq:block_map_y}
y'_0
=
\frac{(l_0+l_1)y_{m-1}-g_{m-1}\sum_{h=0}^{m-2}y_h}{l_0+g_{m-1}}.
\end{equation}
\end{appaproposition}

\begin{proof}
The next sampling epoch occurs after exactly \(g_{m-1}\) physical iterations.
At the current block-end, the oldest live position \(P_{m-1}\) sits at stage
\(l_1{-}1\). After one execute step it completes, releasing \((l_0+l_1)y_{m-1}\)
tokens. Meanwhile, every other live position advances one stage, increasing
total memory by one token per request, while dead positions passing through
stage \(l_1{-}1\) release no memory. The system then admits a new batch at
stage~\(0\) and traverses \(g_{m-1}-1\) pure-evict iterations, until the next
live position reaches the final stage.

We record the fixed-support induction used in this block. After the initial
completion/admission step, let \(a_t\) be the mass of the newly admitted batch
after it has aged to stage \(t\), for \(t=0,\ldots,g_{m-1}-1\). The other live
masses are \(y_0,\ldots,y_{m-2}\) and occupy stages \(S_h+t\). For
\(t<g_{m-1}-1\), no live position reaches stage \(l_1-1\): the next live
position is \(g_{m-1}\) stages behind the completed one, so only dead positions
pass the final stage during these intermediate iterations. Hence the only
memory change before eviction is the one-token footprint growth of the current
live masses. The newly admitted batch is always the least-progressed occupied
stage, since its stage is \(t+1\) while the old live masses are at stages
\(S_h+t+1>t+1\). If LPF ever reached an old live stage, it would first have
exhausted the newly admitted batch, so the new relative position would be absent
at the next block-end. This contradicts the fixed-support branch. Hence,
conditional on the support pattern remaining unchanged, LPF trims a fractional
amount from the new batch, stops before exhausting it, and does not touch any
old live mass. After each such trim the state returns to the memory boundary, so
there is no admission during the pure-evict block. This proves by induction
that the old non-final live masses are copied unchanged and only the new
coordinate changes.

The old non-final live masses shift one slot in circular order:
\[
y'_{h+1}=y_h,\qquad h=0,\dots,m-2.
\]

In the new sampled state, every live stage has advanced by \(g_{m-1}\) positions.
The live stages are therefore
\[
g_{m-1}-1,\quad g_{m-1}+g_0-1,\quad \dots,\quad l_1-1,
\]
with memory weights \(l_0+g_{m-1}\), \(l_0+g_{m-1}+g_0\), \(\dots\),
\(l_0+l_1\). The new boundary equation is
\[
M=(l_0+g_{m-1})y'_0+\sum_{h=0}^{m-2}(l_0+g_{m-1}+S_h)y_h.
\]
Subtracting the old boundary~\eqref{eq:block_boundary} eliminates the constant
terms and yields
\[
(l_0+g_{m-1})y'_0
=
(l_0+l_1)y_{m-1}-g_{m-1}\sum_{h=0}^{m-2}y_h,
\]
which is \eqref{eq:block_map_y}.
\end{proof}

\noindent{\bf Normalized masses.} Define the \emph{normalized mass} of
position~\(i\) by
\begin{equation}\label{eq:normalized_mass_block}
u_i := y_i / g_i .
\end{equation}
This normalization accounts for the length of the predecessor gap: a live
position preceded by more dead stages absorbs more memory growth during the
eviction block, so \(u_i\) measures mass per unit gap
(Figure~\ref{fig:block_decomposition}b).

\begin{appacorollary}[Normalized block-end map]\label{cor:A5}
In normalized coordinates,
the embedded update is
\[
u'_{h+1}=u_h,\qquad h=0,\dots,m-2,
\]
and
\begin{equation}\label{eq:block_map_u}
u'_0
=
u_{m-1}
+
\frac{\sum_{h=0}^{m-2}g_h(u_{m-1}-u_h)}{l_0+g_{m-1}}.
\end{equation}
\end{appacorollary}

\begin{proof}
After rotation, the gap vector becomes \((g_{m-1},g_0,g_1,\dots,g_{m-2})\),
so \(g'_{h+1}=g_h\) for \(h=0,\dots,m-2\) and \(g'_0=g_{m-1}\).
The shift relation \(y'_{h+1}=y_h\) gives
\(u'_{h+1}=y_h/g'_{h+1}=y_h/g_h=u_h\).
For the new admission, dividing~\eqref{eq:block_map_y} by \(g'_0=g_{m-1}\)
and substituting \(y_h=g_h u_h\) yields~\eqref{eq:block_map_u}.
\end{proof}

Next, define the \emph{block-end spread} as
\begin{equation}\label{eq:block_spread}
G^{\mathrm{sc}}:=\max_i u_i-\min_i u_i.
\end{equation}

\begin{appalemma}[Spread monotonicity on a fixed support pattern]\label{lem:A6}
Let
\[
G^{\mathrm{sc}}(Y):=\max_h u_h-\min_h u_h
\]
be the block-end sampled spread. Then
\[
G^{\mathrm{sc}}(Y')\ge G^{\mathrm{sc}}(Y).
\]
Moreover, if \(G^{\mathrm{sc}}(Y)>0\) and the oldest normalized coordinate
\(u_{m-1}\) is an extremum, then the inequality is strict.
\end{appalemma}

\noindent The proof is deferred to Appendix~\ref{app:aux_lemma_proofs}.

\begin{appacorollary}[Strict growth within bounded time]\label{cor:A7}
Fix a support
pattern with \(m\ge 2\). If \(G^{\mathrm{sc}}(Y)>0\) and the support pattern
does not change, then \(G^{\mathrm{sc}}\) strictly increases at least once every
\(m\) sampled steps.
\end{appacorollary}

\begin{proof}
Take a coordinate attaining either the current maximum or the current minimum.
Under the sampled dynamics, old coordinates shift one slot older at each
sampled step. Hence, unless strict growth has already occurred, that extremal
coordinate reaches the oldest slot \(m-1\) within at most \(m-1\) sampled
steps. The following sampled step is then covered by Lemma~\ref{lem:A6} and forces
strict growth.
\end{proof}

\begin{appalemma}[Finite-time support loss on a fixed pattern]\label{lem:A8}
Fix a support
pattern with \(m\ge 2\) and suppose \(G^{\mathrm{sc}}(Y_0)>0\) at the initial
sampled state. Then the fixed-support regime cannot persist: in the physical
LPF dynamics, some live coordinate reaches zero, and the support strictly
shrinks, within finitely many sampled steps.
\end{appalemma}

\noindent The proof is deferred to Appendix~\ref{app:aux_lemma_proofs}.

\begin{appatheorem}[Position-tracking cascade]\label{thm:A1_position}
Under the saturated-input
continuous model of Theorem~\ref{thm:fcfs}, every memory-boundary trajectory
either stays exactly on, or is captured exactly by, a periodic orbit of some
eviction level, or eventually reaches the level-$(l_1-1)$ maximal-eviction
cycle. In particular, every trajectory that is not captured by an intermediate
periodic orbit cascades to maximal eviction.
Lemma~\ref{lem:A9} below shows that these exact nonmaximal capture sets are
of Lebesgue measure zero within each fixed set of occupied stages on the memory
boundary.
\end{appatheorem}

\begin{proof}
By Lemma~\ref{lem:A2}, every non-equilibrium eviction-free trajectory leaves the
eviction-free regime in finite time; equivalently, memory overflow forces an
LPF eviction and moves the trajectory to the next lower live-support level.

Now fix an interval on which the live-position set is constant and has size
\(m\ge 2\). By Lemma~\ref{lem:A3}, dead positions do not revive. Therefore the trajectory
on this interval is governed by the exact block-end sampled chain above.

If \(G^{\mathrm{sc}}=0\) at one sampled phase, then all normalized coordinates
are equal. Equation~\eqref{eq:block_map_u} preserves this equality, so the
entire sampled chain remains balanced. More explicitly, if \(u_h=c\) for all
\(h\), then \(y_h=cg_h\); after each block-end update the gap vector and the
mass vector rotate together, and after one full circuit of the \(m\) live
positions the block-end sampled state returns to itself. Proposition~\ref{prop:A4} and
the fixed-support induction in its proof determine the intervening physical
iterations uniquely from the block-end state. Hence the balanced sampled chain
lifts to the corresponding eviction-level-\((l_1-m)\) periodic orbit of the
physical iteration map.

Suppose instead that \(G^{\mathrm{sc}}>0\) on the sampled chain. Since
\(u_h\in[0,\,M/(l_0+1)]\), the sampled coordinates remain bounded. Lemma~\ref{lem:A8}
shows that the fixed-support regime cannot persist: recurrent evictions drive
some live coordinate to zero in finite time, moving the trajectory to the next
lower live-support level.

Each such support loss strictly decreases the number of live positions, and
Lemma~\ref{lem:A3} rules out revival. After finitely many reductions, either the
trajectory lands exactly on a balanced intermediate periodic orbit, or only one
live position remains. When \(m=1\), the memory boundary forces the deterministic
cycle
\[
\Bigl(\frac{M}{w_0},0,\ldots,0\Bigr)
\to
\Bigl(0,\frac{M}{w_1},0,\ldots,0\Bigr)
\to
\cdots
\to
\Bigl(0,\ldots,0,\frac{M}{w_{l_1-1}}\Bigr),
\]
which is the maximal-eviction cycle.
\end{proof}

\begin{appalemma}[Nonmaximal exact-capture sets have Lebesgue measure zero]\label{lem:A9}
For each nonempty \(I\subseteq\{0,\ldots,l_1-1\}\), let
\[
\mathcal B_I:=\Bigl\{x\in\mathbb R_+^{l_1}:
x_j>0\ \text{for } j\in I,\ x_j=0\ \text{for } j\notin I,
\sum_{j=0}^{l_1-1}w_jx_j=M\Bigr\}
\]
be the relative interior of a memory-boundary face, equipped with
\((|I|-1)\)-dimensional Lebesgue measure on its affine hull. The set of initial
states in \(\mathcal B_I\) whose trajectory is captured in finite time by a
nonmaximal balanced orbit has relative Lebesgue measure zero. For \(|I|=1\), this set is
empty. Consequently, within each fixed set of occupied stages on the memory
boundary, the nonmaximal exact-capture set has Lebesgue measure zero. Viewed as
a subset of \(\mathbb R^{l_1}\), it also has Lebesgue measure zero.
\end{appalemma}

\noindent The proof is deferred to Appendix~\ref{app:aux_lemma_proofs}.

\begin{appalemma}[Local absorption by the maximal-eviction cycle]\label{lem:max_cycle_local_absorption}
Let
\[
C_r:=\frac{M}{w_r}e_r,\qquad r=0,\ldots,l_1-1,
\]
be the states of the level-\((l_1-1)\) maximal-eviction cycle, where \(e_r\)
is the \(r\)th coordinate vector and \(w_r=l_0+1+r\), and write
\(\mathcal C_{\max}:=\{C_0,\ldots,C_{l_1-1}\}\). Let \(\mathcal E\) be the
exact-capture set from Theorem~\ref{thm:fcfs}: the set of memory-boundary
initial states whose trajectories are captured in finite time by the
eviction-free fixed point or by an intermediate balanced limit cycle. For
every \(\epsilon>0\), there exists \(\delta>0\) such that every
memory-boundary trajectory with
\[
\operatorname{dist}(X^0,\mathcal C_{\max})<\delta,
\qquad
X^0\notin\mathcal E,
\]
satisfies
\[
\operatorname{dist}(X^n,\mathcal C_{\max})<\epsilon\quad\text{for all }n\ge0,
\]
and reaches \(\mathcal C_{\max}\) in finite time. Consequently,
\(\mathcal C_{\max}\) is asymptotically stable outside \(\mathcal E\) in the
sense of Definition~\ref{def:stability}.
\end{appalemma}

\noindent The proof is deferred to Appendix~\ref{app:aux_lemma_proofs}.

The exceptional set is not empty. For example, with \(l_0=2\), \(l_1=3\),
\(M=24\), and \(w=(3,4,5)\), the state
\[
X^0=\left(\frac{48}{13},\frac{24}{13},\frac{72}{65}\right)
\]
satisfies the memory boundary and maps in one iteration to
\[
\left(0,\frac{48}{13},\frac{24}{13}\right),
\]
which lies on the level-1 balanced cycle
\[
\left(0,\frac{48}{13},\frac{24}{13}\right)
\to
\left(\frac{24}{13},0,\frac{48}{13}\right)
\to
\left(\frac{72}{13},\frac{24}{13},0\right)
\to
\left(0,\frac{48}{13},\frac{24}{13}\right).
\]

Lemma~\ref{lem:A2} already proves part~(1) of Theorem~\ref{thm:fcfs}. Theorem~\ref{thm:A1_position} gives
part~(2): if \(\mathcal C\) is an eviction-level-\(i\) orbit with
\(1\le i\le l_1-2\), then its block-end sampled representative has
\(G^{\mathrm{sc}}=0\), but any sufficiently small perturbation that makes the
normalized sampled masses unequal produces \(G^{\mathrm{sc}}>0\), and
Theorem~\ref{thm:A1_position} then forces another support loss. So intermediate cycles are not
stable.

Figure~\ref{fig:appendix_cascade} illustrates the full cascade for Example~\ref{ex:spiral}
($l_0 = 2$, $l_1 = 3$, $M = 24$). The three phases are visible:
(i) growing oscillation in the eviction-free regime ($n = 0$--$6$), (ii) intermittent
evictions of increasing severity as the system cascades through levels ($n = 7$--$16$),
and (iii) convergence to the stable level-$2$ cycle ($n \geq 17$) with throughput
$c_2 = 8/5$ per iteration.

\begin{figure}[ht]
\centering
\includegraphics[width=\textwidth]{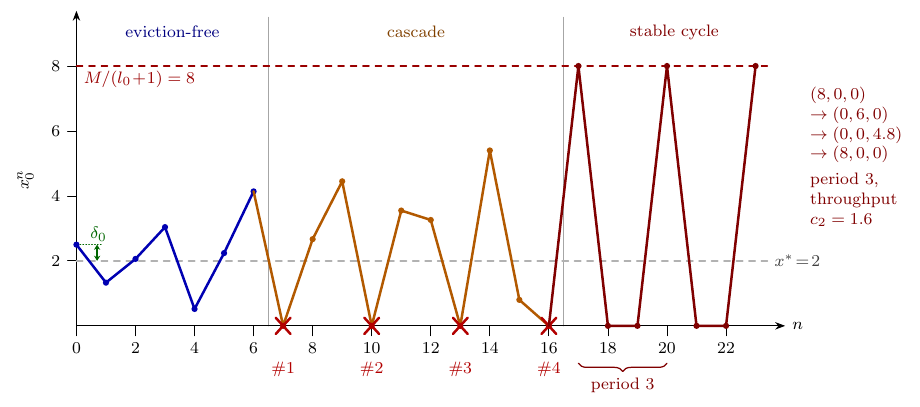}
\caption{Full cascade trajectory for Example~\ref{ex:spiral} ($l_0 = 2$, $l_1 = 3$, $M = 24$). Starting from $\delta_0 = 0.5$ above equilibrium, the system passes through three phases: (i)~growing oscillation in the eviction-free regime ($n = 0$--$6$), (ii)~intermittent evictions of increasing severity (\#1--\#4) as the system cascades through levels ($n = 7$--$16$), and (iii)~convergence to the stable level-$2$ cycle ($n \geq 17$) with throughput $c_2 = 1.6$ per iteration.}
\label{fig:appendix_cascade}
\end{figure}

\paragraph{Higher decoding lengths: $l_1 = 4$.}
With \(l_0=2\), \(l_1=4\), and \(M=48\) (weights \(w=[3,4,5,6]\),
\(x^*=8/3\)), the cascade exhibits \emph{four} phases
(Figure~\ref{fig:appendix_l4_cascade}): (i)~growing oscillation
(\(n=0\)--\(8\)), (ii)~level-1 evictions at \(n=9,13\), (iii)~level-2 and
level-3 evictions at \(n=17,20,21\), and (iv)~convergence to the stable
level-\(3\) cycle (\(n\geq22\)) with throughput \(c_3=2\), a \(25\%\) loss
relative to \(x^*\). The additional decoding stage creates a richer transient:
the system traverses two intermediate eviction levels before reaching the worst
case.

\begin{figure}[ht]
\centering
\includegraphics[width=\textwidth]{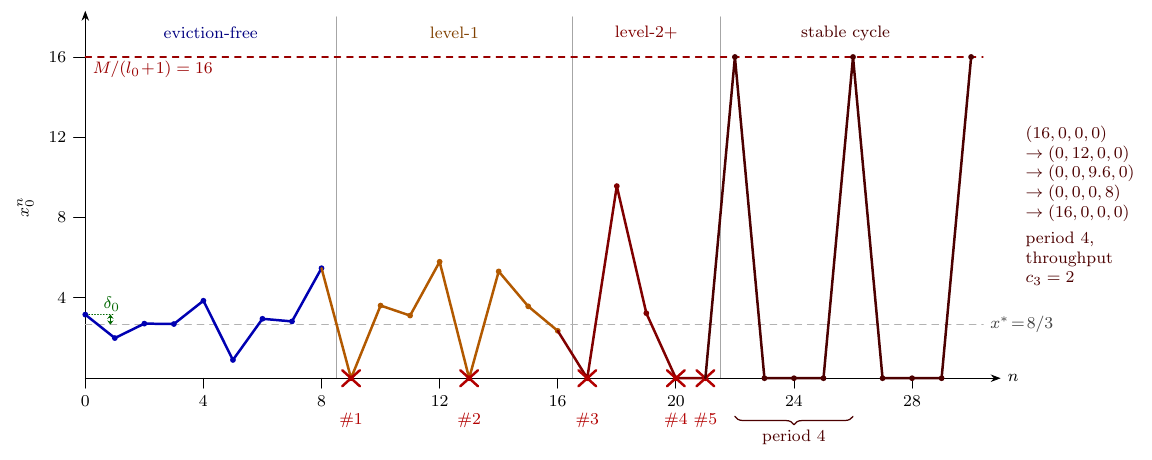}
\caption{Full cascade trajectory for $l_0 = 2$, $l_1 = 4$, $M = 48$ (weights $w = [3,4,5,6]$, $x^* = 8/3$). Starting from $\delta_0 = 0.5$ above equilibrium, the system passes through \emph{four} phases: (i)~growing oscillation in the eviction-free regime ($n = 0$--$8$), (ii)~level-1 evictions at $n = 9, 13$ ($n = 9$--$16$), (iii)~level-2 and level-3 evictions at $n = 17, 20, 21$ ($n = 17$--$21$), and (iv)~convergence to the stable level-$3$ cycle ($n \geq 22$) with throughput $c_3 = 2$ per iteration---a 25\% loss from $x^* = 8/3$.}
\label{fig:appendix_l4_cascade}
\end{figure}

\paragraph{Cycle multiplicity at intermediate levels.}
A qualitatively new phenomenon for $l_1 \geq 4$ is that intermediate eviction levels admit multiple distinct periodic orbits with different throughputs; see Figure~\ref{fig:appendix_l4_cycles} in Section~\ref{sec:single_class} for an illustration. At level-2 (two stages empty per cycle), two orbit families coexist: an \textit{opposite} (maximally separated) orbit (period~2, throughput $= 12/5$, 90\% of~$x^*$) and an \textit{adjacent} orbit (period~4, throughput $= 16/7$, 86\% of~$x^*$). Both are unstable, and the baseline cascade bypasses them to the worst-case level-3 cycle ($c_3 = 2$, 75\% of~$x^*$).

\begin{appacorollary}[Stability of the maximal-eviction cycle]\label{cor:max_cycle_stability}
The level-\((l_1-1)\) maximal-eviction cycle is asymptotically stable outside
\(\mathcal E\). Its basin contains every memory-boundary state outside
\(\mathcal E\), and \(\mathcal E\) has relative Lebesgue measure zero within
each fixed set of occupied stages on the memory boundary. The eviction-free
fixed point and all intermediate balanced limit cycles are unstable in the
sense of Definition~\ref{def:stability}.
\end{appacorollary}

\begin{proof}
Theorem~\ref{thm:A1_position} shows that every memory-boundary
trajectory outside \(\mathcal E\) reaches the single-live-position regime in
finite time. Once \(m=1\), the memory boundary fixes the unique live mass at
each phase, so the trajectory follows the maximal-eviction cycle. This proves
the stated basin. Lemma~\ref{lem:A9} proves the relative-measure-zero claim for
\(\mathcal E\), and Lemma~\ref{lem:max_cycle_local_absorption} supplies the
local stability and convergence quantifiers required by
Definition~\ref{def:stability}. Finally, Lemma~\ref{lem:A2} proves instability
of the eviction-free fixed point, and Theorem~\ref{thm:A1_position} proves
instability of intermediate balanced cycles by showing that arbitrarily small
off-balanced perturbations force another support loss.
\end{proof}

\subsection{Summary}

The proof establishes:
\begin{enumerate}
\item The eviction-free spread $G^n$ is non-decreasing (Lemma~\ref{lm:A1}).
\item For non-equilibrium initial conditions ($G^0 > 0$), the maximum $\bar{x}_n$ grows by at least $G^0/(l_0+1)$ each time it reaches stage $l_1-1$, while $G^n$ remains bounded by $M/(l_0+1)$. This contradiction forces the eviction-free regime to break down in finite time; the resulting LPF eviction creates the first support loss and moves the trajectory to the next lower live-support level (Lemma~\ref{lem:A2}).
\item After the first eviction, relative-position tracking becomes the right state description. Lemma~\ref{lem:A3} shows that once a position reaches zero it never revives, and Proposition~\ref{prop:A4} gives the exact block-end sampled map on any fixed support pattern.
\item On the block-end sampled chain, the normalized spread $G^{\mathrm{sc}}$ is non-decreasing (Lemma~\ref{lem:A6}) and strictly increases within bounded time whenever it is positive (Corollary~\ref{cor:A7}). Thus the only obstruction to further cascade is landing exactly on a balanced intermediate periodic orbit; otherwise recurrent evictions drive another live position to zero in finite time and move the trajectory to the next lower live-support level (Theorem~\ref{thm:A1_position}).
\item Corollary~\ref{cor:max_cycle_stability} collects the stability conclusion:
the level-$(l_1-1)$ cycle is asymptotically stable on the memory boundary
outside the exact nonmaximal capture set.
Its basin contains every memory-boundary initial state except those captured
exactly by a nonmaximal balanced orbit; within each fixed set of occupied
stages, the excluded set has Lebesgue measure zero. All other periodic orbits
are unstable.
\end{enumerate}

This completes the proof of Theorem~\ref{thm:fcfs}. \qed

\subsection{Representative Equilibrium Families}
\label{app:sec3_proofs}

The cascade argument above applies to any fixed support pattern: unless the normalized live masses are balanced, the support shrinks again in finite time. The following proposition characterizes the throughput envelope of the cycled equilibria that remain on a fixed relative support pattern. For such equilibria, the predecessor-gap vector gives a general parametrization. Two named special cases are useful for interpreting the envelope. First, for the contiguous-block geometry at eviction level $i\in\{0,1,\ldots,l_1-1\}$, define
\begin{equation}\label{eq:ci_def}
c_i \;:=\; \frac{M}{(i+1)w_i+\sum_{j=i+1}^{l_1-1} w_j}.
\end{equation}
Here $c_0=x^*$ and $c_{l_1-1}=M/[l_1w_{l_1-1}]$, and the sequence $\{c_i\}_{i=0}^{l_1-1}$ is strictly decreasing. Second, when $k\in\{2,3,\ldots,l_1\}$ divides $l_1$ and $l_1=kd$, the exactly evenly spaced geometry with $k$ live positions has throughput
\begin{equation}\label{eq:balanced_k_cycle_throughput}
\bar{T}_{\mathrm{bal}}(k)
=
\frac{M}{d\sum_{q=1}^{k} w_{qd-1}}.
\end{equation}
The general fixed-support envelope below includes both special cases: contiguous-block gap vectors attain the lower-throughput endpoint, while as-evenly-spaced gap vectors attain the upper-throughput endpoint; the exactly evenly spaced cycles are the divisor case used in the discussion in Section~\ref{sec:single_class}.

\begin{appaproposition}[Fixed-support gap-vector throughput envelope]
\label{prop:gap_vector_envelope}
Fix a fixed-relative-support periodic orbit and a block-end representative with
\(m\in\{1,\ldots,l_1\}\) live positions. Equivalently, consider a cycled
equilibrium whose live relative positions rotate without changing their
predecessor-gap pattern. If \(m=1\), set \(g_0=l_1\). If
\(m\ge2\), let \(g_0,\ldots,g_{m-1}\) be its predecessor gaps. Then
\(g_h\in\mathbb Z_{>0}\), \(\sum_h g_h=l_1\), and the orbit's time-averaged
throughput is
\begin{equation}
\label{eq:gap_vector_throughput}
\bar{T}(g)
=
\frac{M}{l_0l_1+\frac12\left(l_1^2+\sum_{h=0}^{m-1}g_h^2\right)}.
\end{equation}
Consequently, among fixed-relative-support periodic orbits with \(m\) live
positions, throughput is minimized by the contiguous-block gap vector, up to
cyclic rotation,
\((l_1-m+1,1,\ldots,1)\), and is maximized by any as-evenly-spaced gap vector.
Equivalently, if \(l_1=am+r\) with \(a\ge1\) and \(0\le r<m\), the maximizing
gap vector has \(r\) gaps equal to \(a+1\) and \(m-r\) gaps equal to \(a\), up to
permutation of the gaps.
The lower endpoint is the contiguous-block throughput
\[
c_{l_1-m}
=
\frac{M}{l_0l_1+\frac12\left(l_1^2+(l_1-m+1)^2+m-1\right)},
\]
and the upper endpoint is
\[
\bar{T}_{\mathrm{even}}(m)
=
\frac{M}{l_0l_1+\frac12\left(l_1^2+r(a+1)^2+(m-r)a^2\right)}.
\]
When \(m\) divides \(l_1\), \(\bar{T}_{\mathrm{even}}(m)\) equals the
evenly-spaced throughput \(\bar{T}_{\mathrm{bal}}(m)\) in
\eqref{eq:balanced_k_cycle_throughput}.
\end{appaproposition}

Note that, across levels, the ordering need not be monotone: a cycle at a higher eviction level can outperform one at a lower level when the lower-level cycle clusters its active stages into consecutive positions.

\begin{example}[Cross-level non-monotonicity]
\label{ex:cross_level}
Let $l_0=2$, $l_1=10$, so $w_j = 3+j$.
\begin{enumerate}
\item \emph{Contiguous-block, level~4} (six occupied stages in positions $0$--$5$). By \eqref{eq:ci_def},
\[
c_4 = \frac{M}{5\,w_4 + \sum_{j=5}^{9}w_j} = \frac{M}{5\cdot 7 + (8{+}9{+}10{+}11{+}12)} = \frac{M}{85}.
\]
\item \emph{Balanced, level~5} (five occupied stages at positions $1,3,5,7,9$; $k=5$, $d=2$). By \eqref{eq:balanced_k_cycle_throughput},
\[
\bar{T}_{\mathrm{bal}}(5) = \frac{M}{2(w_1{+}w_3{+}w_5{+}w_7{+}w_9)} = \frac{M}{2(4{+}6{+}8{+}10{+}12)} = \frac{M}{80}.
\]
\end{enumerate}
Since $c_4 = M/85 < M/80 = \bar{T}_{\mathrm{bal}}(5)$, the level-4 cycle achieves \emph{lower} throughput despite having one more occupied stage. Contiguous clustering concentrates requests in adjacent heavy stages ($5w_4 = 35$ dominates the denominator), while even spacing samples stages of mixed weight ($w_1 = 4$ through $w_9 = 12$), yielding a smaller memory-time total.
\end{example}

\begin{proof}[Proof of Proposition~\ref{prop:gap_vector_envelope}]
If \(m=1\), the single predecessor gap is \(g_0=l_1\), and the memory boundary
gives \(y_0=M/(l_0+l_1)\) at the block-end representative. This mass completes
once every \(l_1\) iterations, so the throughput is
\(M/[l_1(l_0+l_1)]\), which agrees with~\eqref{eq:gap_vector_throughput}.
Hence assume \(m\ge2\).

The block-end representative of a periodic orbit with a fixed relative support
pattern must be balanced. Otherwise \(G^{\mathrm{sc}}>0\), and Lemma~\ref{lem:A8} would
force another support loss in finite time, contradicting periodicity with \(m\)
live positions. Thus all normalized live masses are equal:
\[
u_h=\frac{y_h}{g_h}=c,\qquad h=0,\ldots,m-1.
\]
Thus \(y_h=cg_h\). With \(S_h=\sum_{t=0}^h g_t\), the memory-boundary equation
\eqref{eq:block_boundary} gives
\[
M
=
\sum_{h=0}^{m-1}(l_0+S_h)y_h
=
c\left(l_0l_1+\sum_{h=0}^{m-1}g_hS_h\right).
\]
During the block that begins with oldest live mass \(y_{m-1}=cg_{m-1}\), exactly
\(g_{m-1}\) physical iterations elapse before the next block-end, and
\(cg_{m-1}\) requests complete at the beginning of that block. Over one full
circuit of block-end transitions, the elapsed physical time is
\(\sum_h g_h=l_1\), and the total completed mass is \(\sum_h cg_h=cl_1\).
Hence the time-averaged throughput of the periodic orbit is \(c\).

This also makes the \(l_1\)-step closure equations explicit. Let
\((g^{(s)},y^{(s)})\) denote the \(s\)-th block-end representative along the
orbit. Corollary~\ref{cor:A5} and the balance relation \(y_h=cg_h\) give
\[
g^{(s+1)}=(g^{(s)}_{m-1},g^{(s)}_0,\ldots,g^{(s)}_{m-2}),
\qquad
y^{(s+1)}=(cg^{(s)}_{m-1},cg^{(s)}_0,\ldots,cg^{(s)}_{m-2}).
\]
After \(m\) block-end transitions, these equations return both the gap vector
and the live-mass vector to their initial order. The total elapsed physical
time over those \(m\) block-end transitions is \(\sum_h g_h=l_1\), so the
physical stage labels also return to their initial positions. Thus the
block-end representative satisfies the same closure condition as the physical
\(l_1\)-step cycle.

It remains to simplify the denominator. Since
\[
\sum_{h=0}^{m-1}g_hS_h
=
\sum_{h=0}^{m-1}\sum_{t=0}^{h}g_hg_t
=
\frac12\left(\left(\sum_{h=0}^{m-1}g_h\right)^2+\sum_{h=0}^{m-1}g_h^2\right)
=
\frac12\left(l_1^2+\sum_{h=0}^{m-1}g_h^2\right),
\]
we obtain~\eqref{eq:gap_vector_throughput}.

For fixed \(m\), the denominator in~\eqref{eq:gap_vector_throughput} is increasing
in \(\sum_h g_h^2\). Among positive integer vectors with sum \(l_1\), the sum of
squares is maximized by making one gap as large as possible and all others equal
to one:
\[
\sum_h g_h^2\le (l_1-m+1)^2+(m-1),
\]
with equality exactly at cyclic permutations of
\((l_1-m+1,1,\ldots,1)\). This is the contiguous-block support geometry, and
substituting this value into~\eqref{eq:gap_vector_throughput} gives
\(c_{l_1-m}\).

The sum of squares is minimized when the gaps differ by at most one. Indeed, if
two gaps satisfy \(g_a\ge g_b+2\), replacing them by \(g_a-1\) and \(g_b+1\)
preserves the total sum and decreases the squared sum by
\[
g_a^2+g_b^2-(g_a-1)^2-(g_b+1)^2
=2(g_a-g_b-1)>0.
\]
Iterating this balancing operation yields \(r\) gaps equal to \(a+1\) and
\(m-r\) gaps equal to \(a\), where \(l_1=am+r\). This proves the upper endpoint
\(\bar{T}_{\mathrm{even}}(m)\). Such a gap vector defines an as-evenly-spaced
support pattern on the circular pipeline. For any positive gap vector, the
balanced state \(y_h=cg_h\) satisfies \(u_h\equiv c\), so the normalized
block-end map in Corollary~\ref{cor:A5} preserves the balanced representative after
rotation. The fixed-support induction in Proposition~\ref{prop:A4} then lifts this
balanced block-end orbit to the physical LPF dynamics. The closure equations are
also satisfied: after one block-end transition the gap vector is cyclically
rotated, after \(m\) block-end transitions it returns to its original order, and
the elapsed physical time is \(\sum_{h=0}^{m-1}g_h=l_1\). Thus every live
position has advanced by one full pipeline cycle, so the physical state returns
to the original state and the balanced block-end orbit is a physical limit
cycle. The lower endpoint is the contiguous-block formula~\eqref{eq:ci_def};
when \(r=0\), all gaps are equal to \(a=l_1/m\), yielding the evenly-spaced
formula~\eqref{eq:balanced_k_cycle_throughput}.
\end{proof}

\subsection{Worst-Case Limit Cycle}
This subsection remains within the saturated-input continuous model of
Theorem~\ref{thm:fcfs}: state coordinates are nonnegative request masses, and a
partial LPF eviction may trim an arbitrary submass of a stage because requests
within the same stage are indistinguishable. The proof of
Lemma~\ref{lem:monotonicity} relies on the following continuous accounting
identity, which we call the \emph{waste decomposition}. Let
\[
W = \sum_{j=0}^{l_1-1} w_j = l_1(2l_0+l_1+1)/2
\]
denote the total memory-time cost required to process a request through all $l_1$ decoding stages.
Over a limit cycle of period $p$, the memory constraint implies $\sum_{j} w_j x^{n}_{j} = M$ for each iteration $n$.
Summing this identity over the $p$ iterations of the cycle and decomposing the total memory-time usage by completed mass versus evicted mass yields
\begin{equation}\label{eq:throughput_decomp}
pM = p\,\bar{T}(\mathcal{C})\,W + E(\mathcal{C}),
\end{equation}
where $E(\mathcal{C}) \ge 0$ denotes the total memory-time consumed by requests that are eventually evicted.
Rearranging gives
\[
\bar{T}(\mathcal{C}) = x^* - \frac{E(\mathcal{C})}{pW},
\]
where $x^* = M/W$ is the throughput of the eviction-free fixed point.
Thus, any limit cycle with eviction satisfies $\bar{T}(\mathcal{C}) < x^*$. The throughput loss relative to the fixed point is proportional to the waste $E(\mathcal{C})$.
The maximal-eviction cycle maximizes this waste, since requests move through the decoding pipeline as a single synchronized batch and eviction trims the batch at every stage.
This regime yields throughput $\bar{T}_{l_1-1}=M/[l_1(l_0+l_1)]$.

\noindent The proof of Lemma~\ref{lem:monotonicity} is deferred to Appendix~\ref{app:aux_lemma_proofs}.

\StartAppendixBTheoremNumbering
\section{Proof of Theorem~\ref{thm:gcd}: GCD Stability Condition and Extensions}
\label{app:theorem2}

Under the LPF convention in Section~\ref{sec:protocol} and the proportional
admission convention in Section~\ref{sec:multi_class}, Theorem~\ref{thm:gcd}
is stated for the two-class \commoninput{} normalization, where the mechanism
is most transparent and the complete proof can be written without additional
notation. This case contains the essential ideas of the argument: the scalar
cohort recurrence, the limiting survival polynomial, the
unit-circle/root-of-unity characterization, the \finiteinput{} root perturbation,
and the Lyapunov argument for global stability.
Sections~\ref{sec:K_type_extension} and~\ref{sec:hetero_prompt_extension} then
discuss the corresponding \(K\)-class and \heteroinput{} extensions through the
appropriate cohort reductions and weighted survival
polynomials. In the two-class
\commoninput{} setting,
(i)~if $\gcd(l_{1,1}, l_{1,2}) > 1$, the eviction-free equilibrium is
unstable; (ii)~if $\gcd(l_{1,1}, l_{1,2}) = 1$, the system converges globally
to equilibrium from any feasible initial active state in the saturated-input
model.

\begin{appblemma}[Entry into the proportional-cohort manifold]
\label{lem:proportional_cohort_entry}
Consider the continuous saturated-input multi-class dynamics under
per-iteration proportional admission \(a_k^n=p_k a^n\) and proportional
within-stage LPF eviction. Let \(d=\max_k l_{1,k}\). Starting from any feasible
active state, after at most \(d\) iterations the state lies on the
proportional-cohort manifold: for every cohort age \(j\), there is an aggregate
cohort mass \(y_j^n\ge 0\) such that
\[
x_{k,j}^n=p_k y_j^n \quad \text{for all classes } k \text{ with } l_{1,k}>j,
\qquad
x_{k,j}^n=0 \quad \text{for } l_{1,k}\le j.
\]
This manifold is forward invariant.
\end{appblemma}

\noindent The proof is deferred to Appendix~\ref{app:aux_lemma_proofs}.

For the global-convergence statement, Lemma~\ref{lem:proportional_cohort_entry}
allows us to start without loss of generality from the proportional-cohort
manifold.

\paragraph{Proof roadmap.}
Sections~\ref{sec:app_b_dynamics}--\ref{sec:app_b_decomposition} derive the
two-class \commoninput{} admission recurrence, its characteristic polynomial, and
the leading-order decomposition that separates the limiting survival polynomial
from the \finiteinput{} correction. Sections~\ref{sec:app_b_roots}--\ref{sec:finite_prompt_appendix}
analyze the limiting roots and show how the \finiteinput{} correction moves the
root-of-unity modes. Sections~\ref{sec:local_stability}--\ref{sec:global_convergence}
prove local and global stability in the coprime case, completing the proof of
Theorem~\ref{thm:gcd}. Section~\ref{sec:K_type_extension} then gives the
corresponding \(K\)-class cohort reduction and survival-polynomial
substitutions, while Section~\ref{sec:hetero_prompt_extension} explains the
\heteroinput{} weighted-polynomial substitution. These extension sections show
how the same mechanism carries over, rather than introducing additional formal
theorem statements. The later subsections collect \finiteinput{} threshold
refinements and pulse-cycle results used to interpret the non-coprime case.

\subsection{System Dynamics and Characteristic Polynomial}
\label{sec:app_b_dynamics}

Consider two request classes with common \inputlen{} $l_0$ but different decoding lengths $l_{1,1} < l_{1,2}$. Let $p = \lambda_1/(\lambda_1 + \lambda_2)$ and $q = 1-p$.

The token balance at time $n$ yields the recurrence:
\begin{align*}
(l_0+1)x^{n+1}_{0} =\; &(l_0+l_{1,2})q \cdot x^{n-l_{1,2}+1}_{0} + [(l_0+l_{1,1})p - q]x^{n-l_{1,1}+1}_{0} \\
&- \sum_{i=n-l_{1,1}+2}^{n}x^i_0 - q\sum_{i=n-l_{1,2}+2}^{n-l_{1,1}}x^i_0.
\end{align*}
Subtracting consecutive equations and defining $D^n = x^{n+1}_{0} - x^n_{0}$:
\[
(l_0+1)D^n + \sum_{m=1}^{l_{1,1}-1}(l_0+m+1)D^{n-m} + q\sum_{m=l_{1,1}}^{l_{1,2}-1}(l_0+m+1)D^{n-m} = 0.
\]

The characteristic polynomial is:
\begin{equation}
F(z) = (l_0+1)z^{l_{1,2}-1} + \sum_{m=1}^{l_{1,1}-1}(l_0+m+1)z^{l_{1,2}-1-m} + q\sum_{m=l_{1,1}}^{l_{1,2}-1}(l_0+m+1)z^{l_{1,2}-1-m}.
\label{eq:characteristic_app}
\end{equation}

\subsection{Limiting Polynomial and Key Decomposition}
\label{sec:app_b_decomposition}

As $l_0 \to \infty$, define the limiting polynomial:
\[
A(z) = \lim_{l_0 \to \infty} \frac{F(z)}{l_0+1} = z^{l_{1,2}-1} + \sum_{m=1}^{l_{1,1}-1}z^{l_{1,2}-1-m} + q\sum_{m=l_{1,1}}^{l_{1,2}-1}z^{l_{1,2}-1-m}.
\]
Summing geometric series yields the closed form:
\begin{equation}
(1-z)A(z) = -z^{l_{1,2}} + pz^{l_{1,2}-l_{1,1}} + q.
\label{eq:closed_form}
\end{equation}

The characteristic polynomial admits a decomposition that separates leading-order root structure from the $O(1/l_0)$ correction. Define
\[
B(z) = F(z) - (l_0+1)A(z) = \sum_{m=1}^{l_{1,1}-1}mz^{l_{1,2}-1-m} + q\sum_{m=l_{1,1}}^{l_{1,2}-1}mz^{l_{1,2}-1-m}.
\]
We claim that $B(z) = (l_{1,2}-1)A(z) - z A'(z)$. To verify, compute:
\begin{align*}
z A'(z) &= (l_{1,2}-1)z^{l_{1,2}-1} + \sum_{m=1}^{l_{1,1}-1}(l_{1,2}-1-m)z^{l_{1,2}-1-m} + q\sum_{m=l_{1,1}}^{l_{1,2}-1}(l_{1,2}-1-m)z^{l_{1,2}-1-m} \\
&= (l_{1,2}-1)A(z) - B(z).
\end{align*}
Therefore:
\begin{equation}
F(z) = (l_0+1)A(z) + B(z) = (l_0 + l_{1,2})A(z) - z A'(z).
\label{eq:decomposition}
\end{equation}
This decomposition reduces the stability problem to analyzing the roots of $A(z)$ and their perturbation under the $O(1/l_0)$ correction $z A'(z)$.

\subsection{Root Structure of the Limiting Polynomial}
\label{sec:app_b_roots}

We establish four lemmas characterizing the roots of $A(z)$.

\begin{appblemma}[Boundedness]
\label{lem:bounded}
All roots of $(1-z)A(z) = 0$ satisfy $|z| \leq 1$.
\end{appblemma}

\noindent The proof is deferred to Appendix~\ref{app:aux_lemma_proofs}.

\begin{appblemma}[Unit Circle Structure]
\label{lem:unit_circle}
If $|z| = 1$ and $(1-z)A(z) = 0$, then $z^g = 1$ where $g = \gcd(l_{1,1}, l_{1,2})$.
\end{appblemma}

\noindent The proof is deferred to Appendix~\ref{app:aux_lemma_proofs}.

\begin{appblemma}[Exclusion of Unity]
\label{lem:unity}
$z = 1$ is not a root of $A(z)$.
\end{appblemma}

\noindent The proof is deferred to Appendix~\ref{app:aux_lemma_proofs}.

\begin{appblemma}[Non-Trivial $g$-th Roots]
\label{lem:primitive}
If $g = \gcd(l_{1,1}, l_{1,2}) > 1$, every non-trivial $g$-th root of unity $\omega = e^{2\pi ik/g}$ for $k = 1, \ldots, g-1$ is a root of $A(z)$.
\end{appblemma}

\noindent The proof is deferred to Appendix~\ref{app:aux_lemma_proofs}.

Combining these lemmas: when $g = 1$, Lemmas~\ref{lem:unit_circle} and~\ref{lem:unity} imply that all roots of $A(z)$ lie strictly inside the unit circle; when $g > 1$, Lemma~\ref{lem:primitive} shows there are exactly $g-1$ roots on the unit circle, located at the non-trivial $g$-th roots of unity. Figure~\ref{fig:root_locus} in Section~\ref{sec:multi_class} visualizes this structure: panel~(a) shows the coprime case with all roots inside the unit circle, while panel~(b) shows the non-coprime case with a root at $z = -1$ drifting outside.

\begin{appbexample}[Root structure verification]
\label{ex:root_verification}
For the coprime case ($l_{1,1} = 2$, $l_{1,2} = 3$, $p = q = 1/2$), the limiting polynomial is
\[
A(z) = z^2 + z + \tfrac{1}{2}, \quad \text{roots } z = \frac{-1 \pm \mathrm{i}}{2}, \quad |z| = \frac{1}{\sqrt{2}} \simeq 0.707.
\]
All roots lie strictly inside the unit circle (Lemma~\ref{lem:bounded}), $A(1) = 5/2 > 0$ (Lemma~\ref{lem:unity}), and since $\gcd(2,3) = 1$ no root reaches the boundary (Lemma~\ref{lem:unit_circle}).

For the non-coprime case ($l_{1,1} = 2$, $l_{1,2} = 4$, $\gcd = 2$):
\[
A(z) = z^3 + z^2 + \tfrac{1}{2}z + \tfrac{1}{2} = (z + 1)\!\left(z^2 + \tfrac{1}{2}\right).
\]
The root $z = -1$ sits on the unit circle, as Lemma~\ref{lem:primitive} predicts: $\omega = e^{\mathrm{i}\pi} = -1$ is the unique non-trivial $g$-th root of unity for $g = 2$. The remaining roots $z = \pm \mathrm{i}/\sqrt{2}$ satisfy $|z| = 1/\sqrt{2} < 1$. Thus only the GCD-predicted root reaches the boundary, and it is exactly the root that the IFT will push outside the unit circle for finite $l_0$ (Theorem~\ref{thm:instability}).
\end{appbexample}

\begin{figure}[ht]
\centering
\includegraphics[width=\textwidth]{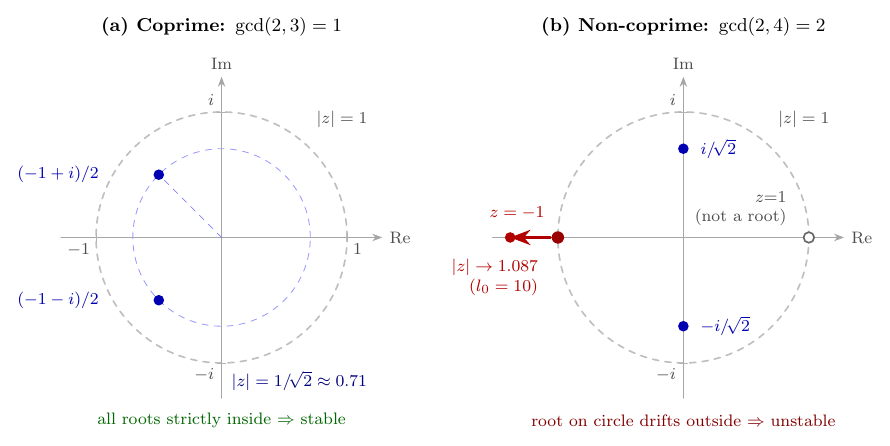}
\caption{Root structure of the limiting polynomial $A(z)$ for the two-class examples ($p_1 = p_2 = 1/2$). (a)~Coprime decode lengths ($l_{1,1} = 2$, $l_{1,2} = 3$): the polynomial $A(z) = z^2 + z + \tfrac{1}{2}$ has roots at $(-1 \pm i)/2$ with $|z| = 1/\sqrt{2} \approx 0.71$ (dashed blue circle), strictly inside the unit circle (dashed gray). The eviction-free equilibrium is asymptotically stable. (b)~Non-coprime decode lengths ($l_{1,1} = 2$, $l_{1,2} = 4$, $g = 2$): the polynomial factors as $A(z) = (z + 1)(z^2 + \tfrac{1}{2})$, placing a root at $z = -1$ on the unit circle (red dot). For finite $l_0 = 10$, the Implicit Function Theorem shows this root drifts to $|z| \approx 1.087$ (arrow, not to scale), establishing instability. The two interior roots (blue dots) at $\pm i/\sqrt{2}$ remain harmless. The point $z = 1$ (hollow marker) is never a root of $A(z)$ (Lemma~iii).}
\label{fig:root_locus}
\end{figure}

\subsection[Instability for Non-Coprime Case]{Instability for Non-Coprime Case ($g > 1$)}

\begin{appbtheorem}[Instability]
\label{thm:instability}
Let $\gcd(l_{1,1}, l_{1,2}) = g > 1$. For sufficiently large $l_0$, the characteristic equation $F(z) = 0$ has $g-1$ roots with $|z| > 1$.
\end{appbtheorem}

\begin{proof}
Using decomposition~\eqref{eq:decomposition}, define:
\[
G(z, \epsilon) = A(z) - \epsilon z A'(z), \quad \epsilon = \frac{1}{l_0 + l_{1,2}}.
\]
For $\epsilon > 0$: $G(z, \epsilon) = 0 \iff F(z) = 0$ (after dividing~\eqref{eq:decomposition} by $l_0 + l_{1,2}$).

At each non-trivial $g$-th root $\omega$, the Implicit Function Theorem applies: $G(\omega, 0) = A(\omega) = 0$ by Lemma~\ref{lem:primitive}, and $\frac{\partial G}{\partial z}(\omega, 0) = A'(\omega) \neq 0$, as we now verify.

To compute $A'(\omega)$, differentiate~\eqref{eq:closed_form} with respect to $z$:
\[
-A(z) + (1-z)A'(z) = -l_{1,2}z^{l_{1,2}-1} + p(l_{1,2}-l_{1,1})z^{l_{1,2}-l_{1,1}-1}.
\]
At a non-trivial $g$-th root $\omega$ where $A(\omega) = 0$ and $\omega^{l_{1,2}} = \omega^{l_{1,2}-l_{1,1}} = 1$ (by Lemma~\ref{lem:primitive}):
\begin{align*}
(1-\omega)A'(\omega) &= -l_{1,2}\omega^{l_{1,2}-1} + p(l_{1,2}-l_{1,1})\omega^{l_{1,2}-l_{1,1}-1} \\
&= -l_{1,2}\omega^{-1} + p(l_{1,2}-l_{1,1})\omega^{-1} \quad \text{(since $\omega^{l_{1,2}} = \omega^{l_{1,2}-l_{1,1}} = 1$)} \\
&= \omega^{-1}[-l_{1,2} + p(l_{1,2}-l_{1,1})] \\
&= \omega^{-1}[-l_{1,2} + pl_{1,2} - pl_{1,1}] \\
&= -\omega^{-1}[(1-p)l_{1,2} + pl_{1,1}] \\
&= -\omega^{-1}[ql_{1,2} + pl_{1,1}].
\end{align*}
Since $p, q > 0$ and $l_{1,1}, l_{1,2} \geq 1$, we have $ql_{1,2} + pl_{1,1} > 0$. Also $\omega \neq 1$ (non-trivial root), so $1 - \omega \neq 0$. Therefore:
\[
A'(\omega) = \frac{-\omega^{-1}(pl_{1,1} + ql_{1,2})}{1-\omega} \neq 0.
\]

By IFT, there exists a smooth function $z(\epsilon)$ with $z(0) = \omega$ and $G(z(\epsilon), \epsilon) = 0$. Differentiating $G(z(\epsilon), \epsilon) = 0$ with respect to $\epsilon$ at $\epsilon = 0$:
\[
z'(0) = -\frac{\partial G/\partial \epsilon}{\partial G/\partial z}\bigg|_{(\omega,0)} = -\frac{-\omega A'(\omega)}{A'(\omega)} = \omega.
\]
Taylor expansion gives $z(\epsilon) = \omega(1 + \epsilon) + r(\epsilon)$ where $r(\epsilon) = O(\epsilon^2)$. Let $\zeta(\epsilon) = \omega^{-1}r(\epsilon)$; then $\zeta(\epsilon) = O(\epsilon^2)$ and:
\[
|z(\epsilon)| = |1 + \epsilon + \zeta(\epsilon)|.
\]
For sufficiently small $\epsilon>0$, there exists a constant $C>0$ such that $|\zeta(\epsilon)| \leq C\epsilon^2$. Therefore:
\[
|z(\epsilon)| \geq 1 + \epsilon - C\epsilon^2 > 1.
\]
Applying this argument to each of the $g-1$ distinct non-trivial $g$-th roots of unity yields $g-1$ distinct roots of $F(z)=0$ with $|z|>1$ for all sufficiently large $l_0$.
\end{proof}

\begin{appblemma}[Complementary Root Count in the Non-Coprime Case]
\label{lem:noncoprime_complementary_roots}
Let \(\gcd(l_{1,1},l_{1,2})=g>1\). For sufficiently large \(l_0\), the
characteristic equation \(F(z)=0\) has exactly \(g-1\) roots outside the unit
disk, all satisfying \(|z|-1=\Theta(1/l_0)\), and the remaining
\(l_{1,2}-g\) roots satisfy \(|z|<1\). Roots are counted with algebraic
multiplicity.
\end{appblemma}

\noindent The proof is deferred to Appendix~\ref{app:aux_lemma_proofs}.

\begin{appblemma}[Generic finite exit from a local no-eviction neighborhood]
\label{lem:generic_finite_exit}
Assume \(\gcd(l_{1,1},l_{1,2})=g>1\) and \(l_0\) is sufficiently large. There
exists a neighborhood \(\mathcal{N}\) of the eviction-free equilibrium, within
the proportional-cohort manifold, such that LPF eviction is inactive throughout
\(\mathcal{N}\). While the trajectory remains in \(\mathcal{N}\), the exact
saturated-input dynamics coincide with the no-eviction affine stage-state map.
For every initial perturbation in \(\mathcal{N}\) with nonzero projection onto
the unstable eigenspace, the trajectory exits \(\mathcal{N}\) in finite time.
The exceptional perturbations are contained in a proper linear subspace and
therefore have Lebesgue measure zero in the local state coordinates.
\end{appblemma}

\noindent The proof is deferred to Appendix~\ref{app:aux_lemma_proofs}.

Lemma~\ref{lem:noncoprime_complementary_roots} gives the exact unstable/stable
root count used in Theorem~\ref{thm:gcd}. Lemma~\ref{lem:generic_finite_exit}
then connects this spectral instability to the local dynamical implication:
outside the measure-zero set of perturbations with zero unstable projection, a
non-coprime system cannot remain forever in the neighborhood where eviction is
inactive and the no-eviction recurrence is exact.
The later pulse-cycle instance in Section~\ref{sec:pulse_cycle_return}
then makes this non-coprime recurrent-eviction behavior explicit through
a period-\(g\) pulse cycle and its local return map.

\subsection[Local Stability for Coprime Case]{Local Stability for Coprime Case ($g = 1$)}
\label{sec:local_stability}

The local stability proof does not require a generic simplicity assumption on
the roots of \(A\). We use Rouch\'e's theorem to count roots with multiplicity
inside a contour strictly contained in the unit disk.

\begin{appbtheorem}[Local Stability]
\label{thm:local_stability}
Let $\gcd(l_{1,1}, l_{1,2}) = 1$. For every $p \in (0,1)$ and for sufficiently large
$l_0$, all roots of $F(z) = 0$ satisfy $|z| < 1$.
\end{appbtheorem}

\begin{proof}
When $g = 1$, Lemmas~\ref{lem:unit_circle} and~\ref{lem:unity} imply all roots
of $A(z)$ satisfy $|\alpha_j| < 1$. Let
$r = \max_j |\alpha_j| < 1$, and choose $\rho$ with $r<\rho<1$. On the circle
$|z|=\rho$, the polynomial $A$ has no zeros, so
\[
m_\rho:=\min_{|z|=\rho}|A(z)|>0,\qquad
M_\rho:=\max_{|z|=\rho}|z A'(z)|<\infty .
\]
Using decomposition~\eqref{eq:decomposition} with
\(\epsilon = 1/(l_0+l_{1,2})\), the finite equation is
\[
A(z)-\epsilon z A'(z)=0 .
\]
For all sufficiently large \(l_0\), \(\epsilon M_\rho<m_\rho\) on
\(|z|=\rho\). Rouch\'e's theorem then implies that
\(A-\epsilon z A'\) and \(A\) have the same number of roots in
\(|z|<\rho\). Since \(A\) has degree \(l_{1,2}-1\), this accounts for all
roots of \(F(z)=0\), and they all satisfy \(|z|<\rho<1\).
\end{proof}

\begin{appbexample}[IFT Root Drift]\label{ex:ift_drift}
For the two-class examples ($l_{1,1} = 2$, $l_{1,2} = 3$ or $4$, $p = q = 1/2$), Figure~\ref{fig:ift_perturbation} plots the spectral radius $|z_{\max}|$ of $F(z)$ as $l_0$ increases. In the non-coprime case ($l_{1,1} = 2$, $l_{1,2} = 4$, $\gcd = 2$), the root at $\omega = -1$ drifts to $|z| \simeq 1 + 1/(l_0 + 4)$, remaining above $1$ for all finite $l_0$. In the coprime case ($l_{1,1} = 2$, $l_{1,2} = 3$), $|z_{\max}| \to 1/\sqrt{2} \simeq 0.707$ as $l_0 \to \infty$, with a safety margin that persists at every finite $l_0$.
\end{appbexample}

\begin{figure}[ht]
\centering
\includegraphics[width=\textwidth]{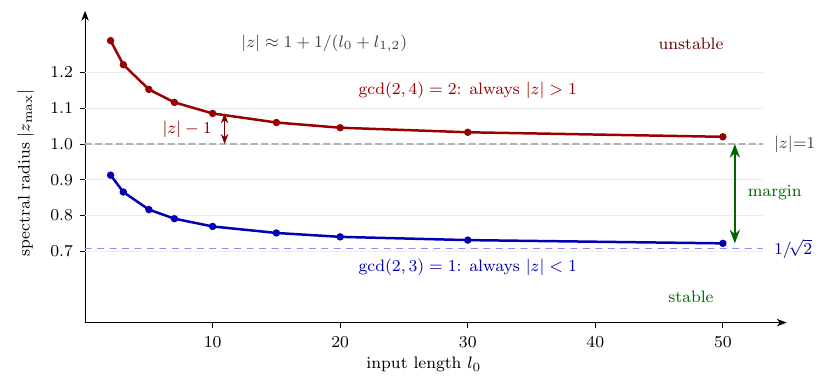}
\caption{Spectral radius of $F(z)$ versus \inputlen{} $l_0$ ($p_1 = p_2 = 1/2$). Non-coprime case ($l_{1,1} = 2$, $l_{1,2} = 4$, red): the root at $z = -1$ of $A(z)$ drifts to $|z| > 1$ for every finite $l_0$ (Theorem~\ref{thm:instability}). Coprime case ($l_{1,1} = 2$, $l_{1,2} = 3$, blue): all roots remain strictly inside the unit circle (Theorem~\ref{thm:local_stability}), with a safety margin that persists as $l_0 \to \infty$. Both curves approach their limiting values ($1$ and $1/\sqrt{2}$) as $l_0 \to \infty$, confirming the IFT first-order approximation $|z| \approx |\alpha|(1 + 1/(l_0 + l_{1,2}))$.}
\label{fig:ift_perturbation}
\end{figure}

\subsection{Global Convergence for Coprime Case}
\label{sec:global_convergence}

Local stability (Theorem~\ref{thm:local_stability}) shows the equilibrium is stable to small perturbations. We now prove global convergence: the system reaches equilibrium from \emph{any feasible initial active state} in the saturated-input model, even when eviction occurs during transients. Figure~\ref{fig:gcd_trajectories}(a) in Section~\ref{sec:multi_class} illustrates this convergence for the coprime case $\gcd(2,3) = 1$.

Throughout this section, we work in the asymptotic regime $l_0 \to \infty$: memory scales as $M = \beta l_0$ for fixed $\beta > 0$; decoding lengths $l_{1,1}, l_{1,2}$ are constants; and $\gcd(l_{1,1}, l_{1,2}) = 1$. All asymptotic notation is with respect to $l_0 \to \infty$. As $l_0$ grows, eviction perturbations become $O(1/l_0)$ while state variables remain $O(1)$, so the contractive linear dynamics dominate.

\subsubsection{Normalized System}

We write $X^n_s$ for the number of requests at stage $s$ at time $n$ (equivalently $x^{n}_{s}$ in the earlier notation). The original state $X^n \in \mathbb{R}^{l_{1,2}}$ has non-uniform equilibrium: $X_s^* = x^*$ for $s < l_{1,1}$ and $X_s^* = qx^*$ for $s \geq l_{1,1}$, where $x^*$ denotes the equilibrium admission rate. (At equilibrium, a fraction $p$ of admissions are class~1 and $q$ are class~2. For $s < l_{1,1}$, both classes contribute, giving $px^* + qx^* = x^*$ requests at stage $s$. For $s \geq l_{1,1}$, only class 2 requests remain, giving $qx^*$.) Define the normalized system:
\[
Y^n_s = \begin{cases}
X^n_s & s < l_{1,1} \\
X^n_s / q & s \geq l_{1,1}
\end{cases}
\]
with weights $V_s = l_0 + s + 1$ for $s < l_{1,1}$ and $V_s = q(l_0 + s + 1)$ for $s \geq l_{1,1}$. The equilibrium becomes uniform: $Y^* = [x^*, \ldots, x^*]^\top$.

The equilibrium admission rate $x^*$ satisfies:
\[
x^* = \frac{M}{\sum_s V_s} = \frac{\beta l_0}{\sum_{s=0}^{l_{1,1}-1}(l_0+s+1) + q\sum_{s=l_{1,1}}^{l_{1,2}-1}(l_0+s+1)}.
\]
The denominator expands to $(l_{1,1} + q(l_{1,2}-l_{1,1}))l_0 + \Theta(1) = (pl_{1,1} + ql_{1,2})l_0 + \Theta(1)$. Letting $\kappa = pl_{1,1} + ql_{1,2}$, we have:
\[
x^* = \frac{\beta l_0}{\kappa l_0 + \Theta(1)} = \frac{\beta}{\kappa} + O(1/l_0) = \Theta(1).
\]

\subsubsection{Linear Dynamics and Lyapunov Function}

The deviation $y^n = Y^n - Y^*$ evolves under the transition matrix $\Phi \in \mathbb{R}^{l_{1,2} \times l_{1,2}}$:
\[
\Phi = \begin{pmatrix}
-V_1/V_0 & -V_2/V_0 & \cdots & -V_{l_{1,2}-1}/V_0 & 0 \\
1 & 0 & \cdots & 0 & 0 \\
\vdots & & \ddots & & \vdots \\
0 & 0 & \cdots & 1 & 0
\end{pmatrix}.
\]
Since the last column of $\Phi$ is zero, $z = 0$ is an eigenvalue. The remaining eigenvalues are determined by the $(l_{1,2}-1) \times (l_{1,2}-1)$ leading principal submatrix, which is a companion matrix whose characteristic polynomial is $F(z)/(l_0+1)$. Hence the nonzero eigenvalues of $\Phi$ are exactly the $l_{1,2} - 1$ roots of $F(z) = 0$.

By Theorem~\ref{thm:local_stability} and the eigenvalue correspondence above,
there exists \(l_{0,\star}\) such that, for every \(l_0\ge l_{0,\star}\), all
eigenvalues of \(\Phi\) satisfy \(|z_j|\le \rho<1\) for a fixed \(\rho\)
independent of \(l_0\). We fix such a pair \((l_{0,\star},\rho)\) below.

\begin{appblemma}[Lyapunov Matrix Bounds]
\label{lem:lyapunov_bounds}
For $l_0 \ge l_{0,\star}$, let $\Phi=\Phi(l_0)$ and let $P=P(l_0)$ be the unique positive definite solution to the discrete Lyapunov equation $\Phi^\top P \Phi - P = -I$. There exists $C_P>0$, independent of $l_0$, such that $I \preceq P \preceq C_P I$.
\end{appblemma}

\noindent The proof is deferred to Appendix~\ref{app:aux_lemma_proofs}.

The Lyapunov function $L(y) = y^\top P y$ satisfies $\|y\|_2^2 \leq L(y) \leq C_P \|y\|_2^2$ and $L(\Phi y) - L(y) = -\|y\|_2^2$.

\begin{appbexample}[Transition matrix and Lyapunov bounds]
\label{ex:transition_matrix}
Continuing the coprime example ($l_{1,1} = 2$, $l_{1,2} = 3$, $l_0 = 10$, $p = q = 1/2$, $M = 118$), the normalized weights are $V_0 = 11$, $V_1 = 12$, $V_2 = 13/2$, giving $\sum V_s = 59/2$ and equilibrium $x^* = 2 \cdot 118/59 = 4$. The transition matrix is
\[
\Phi = \begin{pmatrix}
-12/11 & -13/22 & 0 \\
1 & 0 & 0 \\
0 & 1 & 0
\end{pmatrix},
\]
a companion matrix whose nonzero eigenvalues satisfy $z^2 + \frac{12}{11}z + \frac{13}{22} = 0$, equivalently $F(z)/(l_0+1) = 0$. The roots are $z = (-12 \pm \mathrm{i}\sqrt{142})/22$ with modulus $|z| = \sqrt{13/22} \simeq 0.769$, consistent with the local-stability bound in Theorem~\ref{thm:local_stability}.

Solving the Lyapunov equation $\Phi^\top P \Phi - P = -I$ yields $C_P = \|P\|_2 \simeq 10.8$. The lower bound $P \succeq I$ is tight: $\lambda_{\min}(P) = 1$ (the $k = 0$ term in the series $P = \sum_k (\Phi^k)^\top \Phi^k$ already contributes $I$). The resulting convergence rate $1 - 1/C_P \simeq 0.907$ per step is conservative compared to the observed spectral rate $\rho^2 \simeq 0.59$ (see Example~\ref{ex:lyapunov_illustration}).
\end{appbexample}

\subsubsection{Eviction Bounds}

\begin{appblemma}[Physical Bounds]
\label{lem:physical}
The total number of in-flight requests is $O(1)$, and the number of evicted requests per step is $O(1/l_0)$.
\end{appblemma}

\noindent The proof is deferred to Appendix~\ref{app:aux_lemma_proofs}.

The actual evolution is $y^n = \Phi y^{n-1} + e^n$, where $e^n \in \mathbb{R}^{l_{1,2}}$ captures the deviation caused by eviction.

\begin{appblemma}[Perturbation Bound]
\label{lem:perturbation}
There exists a constant $C_e>0$, independent of $l_0$, such that $\|e^n\|_2 \le C_e/l_0$ for all $n$.
\end{appblemma}

\noindent The proof is deferred to Appendix~\ref{app:aux_lemma_proofs}.

Figure~\ref{fig:perturbation_bound} illustrates the mechanism behind Lemmas~\ref{lem:physical}--\ref{lem:perturbation}. Panel~(a) contrasts token growth per step ($O(1)$) with the memory freed by a single eviction ($\geq l_0 + 1$): at $l_0 = 100$, one eviction clears roughly $9$ times the overflow.
Panel~(b) plots the resulting upper bound on evicted requests per step as a
function of \(l_0\). This bound is uniform over all time steps \(n\) and ensures
\(\|e^n\|_2 \leq C_e/l_0\), which is small enough for the Lyapunov contraction
to dominate the perturbation when \(l_0\) is large.

\begin{figure}[ht]
\centering
\includegraphics[width=\textwidth]{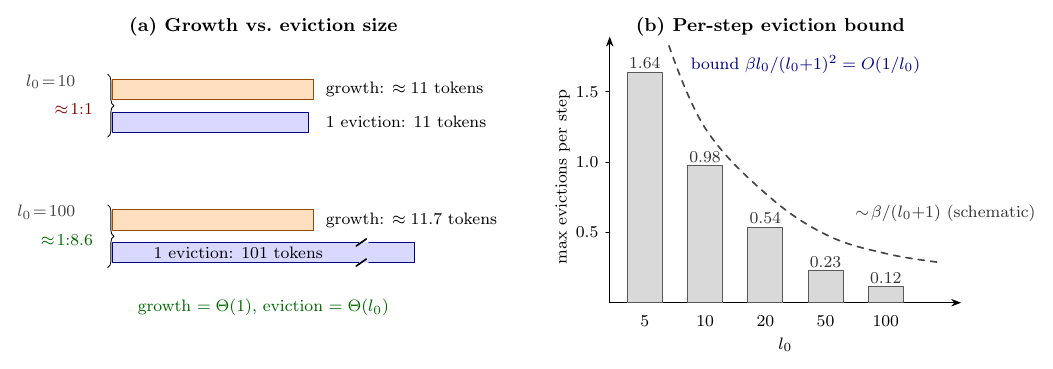}
\caption{Eviction perturbation bound (Lemmas~\ref{lem:physical}--\ref{lem:perturbation}). (a)~Token growth per step ($\Theta(1)$) versus memory freed by one eviction ($\geq l_0 + 1$). At $l_0 = 100$, one eviction clears $\approx\!9$ times the overflow, so each step triggers at most a fraction of an eviction. (b)~Upper bound on evicted requests per step as a function of $l_0$ ($\beta = 11.8$). The bound $\beta l_0/(l_0+1)^2 = O(1/l_0)$ is uniform over all time steps $n$; it guarantees $\|e_n\|_2 \leq C_e/l_0$ for a constant $C_e$ independent of $n$.}
\label{fig:perturbation_bound}
\end{figure}

\subsubsection{Convergence Proof}

\begin{appbtheorem}[Global Convergence]
\label{thm:global}
Let $\gcd(l_{1,1}, l_{1,2}) = 1$ and $M = \beta l_0$. For every \(p\in(0,1)\) and
sufficiently large $l_0$, the system converges exponentially to equilibrium
from any feasible initial active state in the saturated-input model.
\end{appbtheorem}

\begin{proof}
We construct the neighborhood explicitly. Define:
\[
\delta = \frac{\beta}{4 C_2 \kappa} > 0, \quad \text{where } C_2 = \sqrt{l_{1,2}} \cdot \max_s V_s / l_0.
\]
Since $\max_s V_s \le l_0 + l_{1,2}$, we have $C_2 \le \sqrt{l_{1,2}}(1 + l_{1,2}/l_0)$ and in particular $C_2 \le 2\sqrt{l_{1,2}}$ for large $l_0$. Thus $\delta=\Theta(1)$ and is effectively independent of $l_0$. Define the Lyapunov sublevel set $\mathcal{N} = \{y : L(y) < \delta^2\}$. Since $L(y) \geq \|y\|_2^2$ (Lemma~\ref{lem:lyapunov_bounds}), every $y \in \mathcal{N}$ satisfies $\|y\|_2 < \delta$.

We first show that no eviction occurs inside $\mathcal{N}$. The post-shift memory $E^S$ (before admission) is:
\[
E^S = \sum_{s=0}^{l_{1,2}-2} V_{s+1} Y^n_s = \sum_{s=0}^{l_{1,2}-2} V_{s+1}(x^* + y_{n,s}).
\]
At equilibrium, $E^{S,*} = \sum_{s=0}^{l_{1,2}-2} V_{s+1} x^* = M - V_0 x^*$. The deviation satisfies:
\[
|E^S - E^{S,*}| = \left| \sum_{s=0}^{l_{1,2}-2} V_{s+1} y_{n,s} \right| \leq \max_s V_s \cdot \sqrt{l_{1,2}} \|y^n\|_2 \leq C_2 l_0 \|y^n\|_2.
\]
When $\|y^n\|_2 \leq \delta$:
\begin{align*}
E^S - M &\leq -V_0 x^* + C_2 l_0 \delta \\
&= -\frac{(l_0+1)\beta}{\kappa + O(1/l_0)} + C_2 l_0 \cdot \frac{\beta}{4 C_2 \kappa} \\
&\leq -\frac{\beta l_0}{2\kappa} + \frac{\beta l_0}{4\kappa} = -\frac{\beta l_0}{4\kappa} < 0.
\end{align*}
Hence, no eviction occurs inside $\mathcal{N}$ for sufficiently large $l_0$.

Outside $\mathcal{N}$, the Lyapunov energy strictly decreases. The one-step change is:
\begin{align*}
\Delta L_n &= L(y^n) - L(y^{n-1}) = (\Phi y^{n-1} + e^n)^\top P (\Phi y^{n-1} + e^n) - (y^{n-1})^\top P y^{n-1} \\
&= (y^{n-1})^\top (\Phi^\top P \Phi - P) y^{n-1} + 2 (e^n)^\top P \Phi y^{n-1} + (e^n)^\top P e^n \\
&= -\|y^{n-1}\|_2^2 + R_n,
\end{align*}
where the residual term is $R_n = 2(e^n)^\top P \Phi y^{n-1} + (e^n)^\top P e^n$. By Lemmas~\ref{lem:lyapunov_bounds}, \ref{lem:physical}, and~\ref{lem:perturbation}:
\[
|R_n| \leq 2\|e^n\|_2 \cdot \|P\|_2 \cdot \|\Phi\|_2 \cdot \|y^{n-1}\|_2 + \|P\|_2 \|e^n\|_2^2 \leq \frac{C_3}{l_0} \|y^{n-1}\|_2 + \frac{C_4}{l_0^2},
\]
where we used $\|P\|_2 \le C_P$ (Lemma~\ref{lem:lyapunov_bounds}) and $\|e^n\|_2 \le C_e/l_0$ (Lemma~\ref{lem:perturbation}). Moreover, for all sufficiently large $l_0$, the coefficients satisfy $|V_s/V_0|\le 2$, so $\|\Phi\|_2 \le \|\Phi\|_F \le \sqrt{5(l_{1,2}-1)}=:C_\Phi$. Thus $C_3 = 2 C_P C_\Phi C_e$ and $C_4 = C_P C_e^2$.

Outside $\mathcal{N}$, $L(y^{n-1}) \geq \delta^2$, so $\|y^{n-1}\|_2^2 \geq L(y^{n-1})/C_P \geq \delta^2/C_P$. Hence:
\[
\Delta L_n \leq -\frac{\delta^2}{C_P} + \frac{C_3 M_y}{l_0} + \frac{C_4}{l_0^2},
\]
where $M_y = \sup_n \|y^n\|_2 = O(1)$. Indeed, $\sum_s X^n_s\le M/(l_0+1)=O(1)$ implies $\|Y^n\|_2 \le \|Y^n\|_1 \le \frac{1}{\min(1,q)}\sum_s X^n_s=O(1)$, while $\|Y^*\|_2=\sqrt{l_{1,2}}\,x^*=\Theta(1)$, so $\|y^n\|_2\le \|Y^n\|_2+\|Y^*\|_2=O(1)$ uniformly in $n$. Choose $l_{0,3}$ such that $(C_3 M_y + C_4)/l_0 \leq \delta^2/(2C_P)$ for $l_0 \geq l_{0,3}$. Then $\Delta L_n \leq -\delta^2/(2C_P) < 0$.

The system enters $\mathcal{N}$ in finite time: the initial Lyapunov value satisfies $L(y_0) \leq C_P M_y^2$, and the guaranteed decrease of $\delta^2/(2C_P)$ per step gives entry within $T^* = \lceil 2 C_P^2 M_y^2 / \delta^2 \rceil$ steps. Once inside, the system remains in $\mathcal{N}$: since $e^n = 0$ inside $\mathcal{N}$ (no eviction occurs there), $L(y^n) = L(y^{n-1}) - \|y^{n-1}\|_2^2 < L(y^{n-1})$, so the sublevel set is invariant.

Finally, the convergence inside $\mathcal{N}$ is exponential. With $e^n = 0$, the dynamics reduce to $y^n = \Phi y^{n-1}$. Using $L(\Phi y)-L(y)=-\|y\|_2^2$ and $L(y)\le C_P\|y\|_2^2$:
\[
L(y^n) = L(y^{n-1})-\|y^{n-1}\|_2^2 \le \left(1-\frac{1}{C_P}\right)L(y^{n-1}).
\]
Iterating gives $L(y^n)\le (1-1/C_P)^{n-n_0}L(y^{n_0})$, hence $\|y^n\|_2\to 0$ geometrically and the convergence is exponential.
\end{proof}

The following example traces each phase of the convergence mechanism on a concrete instance.

\begin{appbexample}[Lyapunov convergence illustration]
\label{ex:lyapunov_illustration}
Using the coprime parameters of Example~\ref{ex:coprime} ($l_{1,1} = 2$, $l_{1,2} = 3$, $l_0 = 10$, $p = q = 1/2$, $M = 118$), the transition matrix $\Phi \in \mathbb{R}^{3 \times 3}$ has spectral radius $\rho \simeq 0.769$, the discrete Lyapunov equation $\Phi^\top P \Phi - P = -I$ gives $C_P \simeq 10.8$, and the eviction-free threshold is $\delta \simeq 0.57$ (so $\delta^2 \simeq 0.32$).

Near equilibrium, the post-shift memory
\(E^{S,*}=V_1x^*+V_2x^*=74\) consumes \(63\%\) of \(M=118\), leaving
\(V_0x^*=44\) units for new admissions. Under the worst perturbation
\(\|y\|=\delta\), the memory deviation is at most
\(C_2l_0\delta\simeq12\) units, so the admission margin remains \(32\) units
above zero (Figure~\ref{fig:lyapunov_step1}).

Starting from $y_0 = (1, 0, 0)^\top$, the energy $L(y_0) = 8.70$ decreases by $\|y^{n-1}\|^2$ at each step. The guaranteed decrease outside $\mathcal{N}$ exceeds $\delta^2/(2C_P) \simeq 0.015$ per step, but the actual decreases range from $0.19$ to $2.55$. The energy crosses $\delta^2$ at $T^* = 8$ (Figure~\ref{fig:lyapunov_step23}).

Inside $\mathcal{N}$, no eviction occurs, so the pure linear dynamics
$y^n = \Phi y^{n-1}$ take over. The trajectory remains in
\(\mathcal{N}\) and converges geometrically:
\[
L(y^n)\le (1-1/C_P)^{n-n_0}L(y^{n_0}).
\]
This gives a theoretical rate of \(1-1/C_P\simeq0.907\) per step, while the
observed decay is closer to \(\rho^2\simeq0.59\). The state spirals to the
origin in the \((y_0,y_1)\)-plane (Figure~\ref{fig:lyapunov_step45}).
\end{appbexample}

\begin{figure}[ht]
\centering
\includegraphics[width=\textwidth]{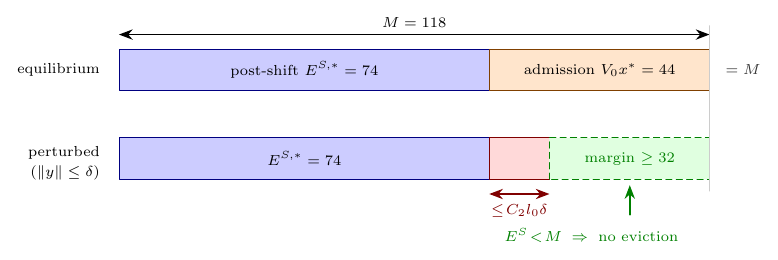}
\caption{Step~1 illustration (Example~\ref{ex:lyapunov_illustration}): post-shift memory budget. At equilibrium, admitting $x^* = 4$ requests fills memory to exactly $M$. Under worst-case perturbation $\|y\| \leq \delta$, the deviation $C_2 l_0 \delta \approx 12$ leaves a margin of $32$ units, so no eviction is triggered inside $\mathcal{N}$.}
\label{fig:lyapunov_step1}
\end{figure}

\begin{figure}[ht]
\centering
\includegraphics[width=\textwidth]{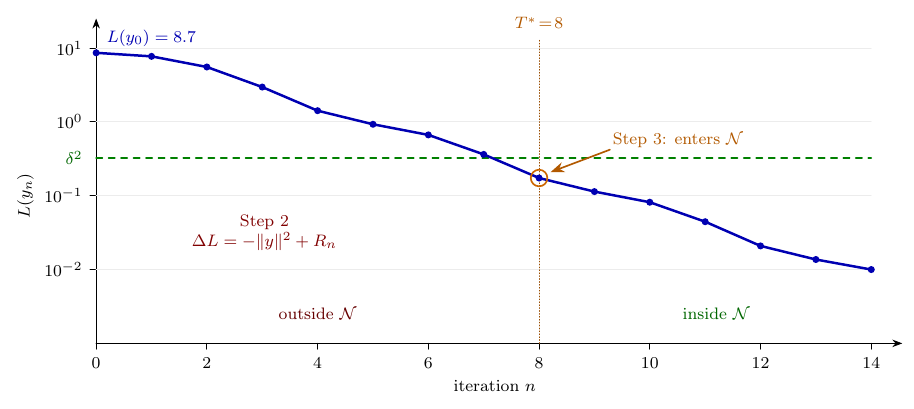}
\caption{Steps~2--3 illustration (Example~\ref{ex:lyapunov_illustration}): energy descent on log scale. Outside $\mathcal{N}$ (left of the dotted line at $T^*$), the Lyapunov energy $L(y_n)$ decreases at each step (Step~2): the actual decreases ($0.19$--$2.55$) far exceed the guaranteed minimum $\delta^2/(2C_P) \approx 0.015$. The energy crosses $\delta^2$ at $T^* = 8$ (Step~3, orange circle).}
\label{fig:lyapunov_step23}
\end{figure}

\begin{figure}[ht]
\centering
\includegraphics[width=\textwidth]{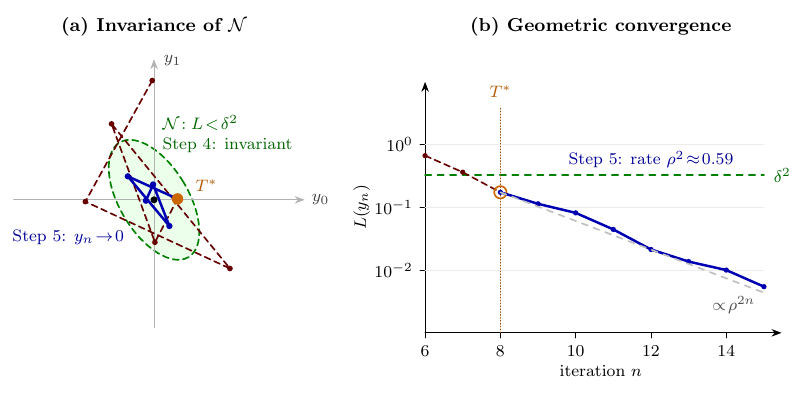}
\caption{Steps~4--5 illustration (Example~\ref{ex:lyapunov_illustration}). (a)~State-space trajectory in the $(y_0, y_1)$ plane. The system enters $\mathcal{N} = \{L < \delta^2\}$ (green dashed ellipse) at $T^* = 8$ (orange) and remains inside (Step~4, invariance). Blue segments show the trajectory spiraling to the origin under pure linear dynamics (Step~5). (b)~Energy $L(y_n)$ on log scale near $T^*$. Inside $\mathcal{N}$, the decay tracks the geometric rate $\rho^2 \approx 0.59$ (gray dashed).}
\label{fig:lyapunov_step45}
\end{figure}

The mechanism behind this convergence is the scaling of eviction perturbations. When $M = \beta l_0$, the number of in-flight requests is $O(1)$, so each perturbation is $O(1/l_0)$. As $l_0 \to \infty$, these perturbations vanish relative to the $O(1)$ state variables, and the contractive linear dynamics dominate.

\subsection{Finite-Input Spectral Drift and Stability Threshold}
\label{sec:finite_prompt_appendix}

This subsection proves the \finiteinput{} statements used in
Section~\ref{sec:sensitivity_analysis}. Throughout this subsection, set
\(a:=l_{1,1}\), \(n:=l_{1,2}\), \(q:=1-p\), and \(\theta_n:=a/n\).  The proof has
three parts. First, we quantify how each limiting root moves when
\(\varepsilon=(l_0+n)^{-1}\) is positive. Second, we identify the
limiting root closest to the unit circle. Third, we balance these two
effects to obtain the finite-\(l_0\) stability threshold.

We write the corresponding finite-\(l_0\) characteristic equation as
\(F_\varepsilon(z)\), where
\[
F_\varepsilon(z)
=
\frac{1}{\varepsilon}A(z)-z A'(z).
\]

\begin{appblemma}[Multiple roots are away from the threshold shell]
\label{lem:multiple_roots_separated}
Fix \(p\in(0,1)\). There are constants \(c_p>0\) and \(N_p\) such that, for
every coprime pair \(1\le a<n\), \(n\ge N_p\), any non-unit multiple root
\(\zeta\) of \(P_n(z):=z^n-pz^{n-a}-q\) satisfies
\[
|\zeta|\le 1-\frac{c_p}{n}.
\]
Hence possible multiple roots of \(A\) are separated from the
\(O(n^{-3})\) threshold shell in which the two dominant roots cross the unit
circle.
\end{appblemma}

\noindent The proof is deferred to Appendix~\ref{app:aux_lemma_proofs}.

\begin{appbtheorem}[Spectral drift under finite \inputlen{}]
\label{thm:spectral_drift}
Assume \(\gcd(l_{1,1},l_{1,2})=1\). Let \(\alpha\) be a simple root of the
limiting polynomial \(A\), i.e., \(A(\alpha)=0\) and \(A'(\alpha)\neq0\).
Then there exists a unique root branch \(z_\alpha(\varepsilon)\) of
\(F_\varepsilon\) with \(z_\alpha(0)=\alpha\), and constants
\(\varepsilon_\alpha>0\) and \(K_\alpha>0\), such that for all
\(0\le\varepsilon\le\varepsilon_\alpha\),
\begin{equation}
\label{eq:spectral_drift}
\left|z_{\alpha}(\varepsilon)-\alpha(1+\varepsilon)-c_2(\alpha)\varepsilon^2\right|
\le
K_\alpha\varepsilon^3,
\qquad
c_2(\alpha)=
\frac{\alpha}{2}\left(\frac{\alpha A''(\alpha)}{A'(\alpha)}+2\right).
\end{equation}
Along coprime families with \(l_{1,2}\to\infty\), the two dominant limiting
branches are simple. Let \(\rho(\varepsilon)\) denote the spectral radius of the
corresponding finite-\(l_0\) linearization. On each dominant branch, the
second-order modulus correction is positive. Moreover, for every fixed
threshold-scale constant \(E>0\), there exist constants \(K_\rho>0\) and
\(N_\rho\ge1\) such that, for all \(l_{1,2}\ge N_\rho\) and all
\(0\le\varepsilon\le E(l_{1,2})^{-3}\),
\begin{equation}
\label{eq:rho_l0}
\left|\rho(\varepsilon)-\rho_\infty(1+\varepsilon)\right|
\le
K_\rho\,l_{1,2}\varepsilon^2.
\end{equation}
\end{appbtheorem}

\begin{appbproposition}[Spectral gap scaling]
\label{prop:spectral_gap_scaling}
Along any coprime family with \(l_{1,2}\to\infty\), with
\(\theta_n:=l_{1,1}/l_{1,2}\), there exist constants
\(K_{\mathrm{gap}}>0\) and \(N_{\mathrm{gap}}\ge1\) such that, for all
\(l_{1,2}\ge N_{\mathrm{gap}}\),
\begin{equation}
\label{eq:tight_spectral_gap}
\left|1-\rho_\infty-c(\theta_n,p)(l_{1,2})^{-3}\right|
\le
K_{\mathrm{gap}}(l_{1,2})^{-4},
\qquad
c(\theta_n,p)=
\frac{2\pi^2p(1-p)}{((1-p)+p\theta_n)^3}.
\end{equation}
\end{appbproposition}

\begin{appbcorollary}[Stability threshold]
\label{cor:l0_threshold}
Along any coprime family with \(l_{1,2}\to\infty\), there exist constants
\(K_S>0\) and \(N_0\ge1\) such that, for all \(l_{1,2}\ge N_0\), the minimum
stable \inputlen{}
\[
l_{0,\min}
:=
\min\Bigl\{l_0\in\mathbb{Z}_{\ge0}:
\rho\bigl((l_0+l_{1,2})^{-1}\bigr)<1
\Bigr\}
\]
satisfies
\[
\left|(l_{0,\min}+l_{1,2})(1-\rho_\infty)-1\right|
\le
K_S(l_{1,2})^{-2}.
\]
\end{appbcorollary}
\[
P_n(z):=z^n-pz^{n-a}-q .
\]
Since \(A(1)=pa+qn>0\), the unit root of \(P_n\) is removed by the
\((1-z)\) factor in~\eqref{eq:closed_form}. Thus the roots of \(A\) are
exactly the non-unit roots of \(P_n\).

\begin{appblemma}[Near-unit logarithmic chart]
\label{lem:near_unit_chart}
Fix \(p\in(0,1)\). There exist constants \(\delta,\eta,c>0\), independent of
the coprime pair \(1\le a<n\), such that every root \(z\ne1\) of
\(P_n\) with \(|z|\ge e^{-\eta/n}\) can be represented as
\[
z=\omega e^{-\mu/n},
\qquad
\omega^n=1,\qquad \Im\mu\in[-\pi,\pi],
\]
after choosing an \(n\)-th root \(\omega\) nearest to the argument of
\(z\). If \(x\in(-\pi,\pi]\) is defined by \(\omega^{-a}=e^{\mathrm{i}x}\), then
\[
H(\mu,x,\theta_n)=0,\qquad
\theta_n=a/n,
\]
where
\[
H(\mu,x,\theta):=e^{-\mu}-p e^{\mathrm{i}x}e^{-(1-\theta)\mu}-q .
\]
For \(|x|\le\delta\) and \(\Re\mu\le\eta\), this equation has the unique
solution \(\mu=\mu(x,\theta)\) near zero. For \(|x|\ge\delta\), every such
near-unit solution satisfies \(\Re\mu\ge c\).
\end{appblemma}

\noindent The proof is deferred to Appendix~\ref{app:aux_lemma_proofs}.

\subsubsection{Proof of Proposition~\ref{prop:spectral_gap_scaling}}
\label{sec:tight_bound_proof}

We next evaluate the limiting spectral gap.  The limiting roots of \(A\) are
most transparent in logarithmic coordinates around the \(n\)-th roots of
unity. For each
\(m\in\{0,\ldots,n-1\}\), let
\[
x_m:=\text{the principal representative of }\frac{2\pi m}{n}\pmod{2\pi},
\qquad x_m\in(-\pi,\pi].
\]
Because \(\gcd(a,n)=1\), multiplication by \(a\) is invertible modulo \(n\).
Hence there is a unique \(u_m\pmod n\) such that
\[
a u_m\equiv -m\pmod n.
\]
Define
\[
\omega_{m,n}:=e^{2\pi \mathrm{i} u_m/n}.
\]
Then \(\omega_{m,n}^n=1\) and \(\omega_{m,n}^{-a}=e^{\mathrm{i}x_m}\).

For \(m\ne0\), the near-unit root chart in Lemma~\ref{lem:near_unit_chart}
parameterizes the root branch associated with this phase as
\[
z=\omega_{m,n}\exp(-\mu/n).
\]
Substituting this expression into \(P_n(z)=0\) gives
\[
H(\mu,x_m,\theta_n)=0,
\qquad
\theta_n:=a/n,
\]
where
\[
H(\mu,x,\theta)
:=
e^{-\mu}
-p e^{\mathrm{i}x}e^{-(1-\theta)\mu}
-q .
\]
At phase \(x=0\), the solution is \(\mu=0\) for every
\(\theta\in[0,1]\). Moreover,
\[
H_\mu(0,0,\theta)=-(q+p\theta)\neq0.
\]
Thus the implicit-function theorem gives a smooth local branch
\(\mu(x,\theta)\), uniformly for \(\theta\in[0,1]\),
such that
\[
H(\mu(x,\theta),x,\theta)=0,
\qquad
\mu(0,\theta)=0.
\]

Let \(D(\theta):=q+p\theta\). Differentiating implicitly at \(x=0\) gives
\[
\mu_x(0,\theta)=-i\frac{p}{D(\theta)}
\]
and
\[
\mu_{xx}(0,\theta)=\frac{pq}{D(\theta)^3}.
\]
Consequently, near the boundary phase \(x=0\),
\[
\Re\mu(x,\theta)
=
\frac{pq}{2D(\theta)^3}x^2+O(|x|^3),
\]
uniformly in \(\theta\). Branches bounded away from \(x=0\) are separated from
the unit circle by Lemma~\ref{lem:near_unit_chart}: for \(|x|\ge\delta\),
near-unit roots satisfy \(\Re\mu\ge c_\delta>0\), and therefore
\(|z|\le e^{-c_\delta/n}=1-\Theta(n^{-1})\). Roots outside the
near-unit chart have an even larger radial gap.

Because \(a\) and \(n\) are coprime, the nonzero phase set is
\[
\{\pm 2\pi/n,\pm4\pi/n,\ldots\}.
\]
After shrinking \(\delta\), the Taylor remainder is at most one half of the
quadratic term for \(|x|\le\delta\), uniformly in \(\theta\in[0,1]\). Hence
among phases with \(|x|\le\delta\), the roots closest to the unit circle are
exactly those with \(|x|=2\pi/n\). For these branches,
\[
\Re\mu(2\pi/n,\theta_n)
=
\frac{2\pi^2pq}{D(\theta_n)^3}n^{-2}
+O(n^{-3}).
\]
Therefore
\[
\rho_\infty
=
\exp\!\left(
-\frac{1}{n}\Re\mu(2\pi/n,\theta_n)
\right)
=
1-
\frac{2\pi^2pq}{(q+p\theta_n)^3}n^{-3}
+O(n^{-4}).
\]
Returning to the original notation \(\theta_n=l_{1,1}/l_{1,2}\) and \(n=l_{1,2}\) gives
\[
1-\rho_\infty
=
\frac{2\pi^2p(1-p)}
{((1-p)+p\theta_n)^3}
\,(l_{1,2})^{-3}
+O((l_{1,2})^{-4}),
\]
which is~\eqref{eq:tight_spectral_gap}.

The same estimates also give the separation needed later. If
\(x_k=2\pi k/n\) and \(2\le k\le \delta n/(2\pi)\), then
\[
\Re\mu(x_k,\theta_n)-\Re\mu(x_1,\theta_n)
\ge c(k^2-1)n^{-2},
\]
so the corresponding moduli are below the dominant modulus by at least
\(c n^{-3}\). If \(|x_k|\ge\delta\), Lemma~\ref{lem:near_unit_chart} gives
\(1-|\alpha_k|\ge c_\delta n^{-1}\). Thus every non-dominant limiting branch
is separated from the dominant modulus by at least \(c n^{-3}\).

\subsubsection{Proof of Theorem~\ref{thm:spectral_drift}}

Let \(\alpha\) be a simple root of \(A\).  The key point is that the
\finiteinput{} correction enters the characteristic equation as a small
radial perturbation of the limiting polynomial.  Dividing
\eqref{eq:decomposition} by \(l_0+n\), the finite-\(l_0\) characteristic
equation is
\[
G(z,\varepsilon):=A(z)-\varepsilon z A'(z)=0,
\qquad
\varepsilon=\frac{1}{l_0+n}.
\]
At \((z,\varepsilon)=(\alpha,0)\), we have
\(G(\alpha,0)=0\) and \(G_z(\alpha,0)=A'(\alpha)\neq0\).
The analytic implicit-function theorem therefore identifies a unique root
branch \(z_\alpha(\varepsilon)\), analytic in \(\varepsilon\), with
\(z_\alpha(0)=\alpha\).

Differentiating \(G(z_\alpha(\varepsilon),\varepsilon)=0\) at
\(\varepsilon=0\) gives
\[
z_\alpha'(0)=\alpha .
\]
Thus the first-order finite-\(l_0\) effect pushes the root radially outward:
the direction of motion is the root itself.  To obtain the stated remainder,
expand one order further.  Write
\[
z_\alpha(\varepsilon)
=\alpha+\alpha\varepsilon+c_2(\alpha)\varepsilon^2+O(\varepsilon^3).
\]
Substituting this expansion into
\(A(z)-\varepsilon z A'(z)=0\) and matching the
\(\varepsilon^2\) terms gives
\[
c_2(\alpha)
=
\frac{\alpha}{2}
\left(
\frac{\alpha A''(\alpha)}{A'(\alpha)}+2
\right).
\]
Analyticity of the branch gives the stated \(O(\varepsilon^3)\) remainder on a
sufficiently small branch neighborhood.

To use this expansion at the threshold scale, we need a uniform bound on the
second-order coefficient along the dominant roots. For such a root, write
\[
\alpha=\omega_{m,n}\exp(-\mu(x_m,\theta_n)/n),
\qquad |x_m|=2\pi/n.
\]
Since \(\mu(x_m,\theta_n)=O(n^{-1})\), we have
\[
\alpha^a=e^{-ix_m}e^{-\theta_n\mu(x_m,\theta_n)}=1+O(n^{-1}).
\]
Moreover,
\[
P_n'(\alpha)
=
\alpha^{n-a-1}\{n\alpha^a-p(n-a)\},
\]
and therefore
\[
n\alpha^a-p(n-a)
=
n\{q+p\theta_n+O(n^{-1})\}.
\]
Since \(q+p\theta_n\ge q>0\), \(|P_n'(\alpha)|\ge c n\) for large \(n\). Since
\(\alpha\ne1\), \(P_n(\alpha)=0\), and \(A=-P_n/(1-z)\) by
\eqref{eq:closed_form}, we have
\[
A'(\alpha)=-\frac{P_n'(\alpha)}{1-\alpha}.
\]
Moreover, \(|1-\alpha|=O(n^{-1})\) on the dominant branches, so
\(|A'(\alpha)|\ge c'n^2\). In particular, the dominant roots are simple
uniformly along the coprime family. The same calculation gives
\[
P_n''(\alpha)
=
\alpha^{n-a-2}
\{n(n-1)\alpha^a-p(n-a)(n-a-1)\}
=O(n^2),
\]
so
\[
\left|\frac{P_n''(\alpha)}{P_n'(\alpha)}\right|=O(n)
\]
uniformly on the dominant branches. Also
\[
|1-\alpha|^{-1}=O(n).
\]
Indeed, \(\omega_{m,n}\ne1\), the nearest nontrivial \(n\)-th root of unity is
at distance \(2\sin(\pi/n)\) from \(1\), and
\(\alpha-\omega_{m,n}=O(n^{-2})\) on the dominant branches. Since the constant
sign in \(A=-P_n/(1-z)\) cancels in logarithmic derivatives,
\[
\frac{A''(\alpha)}{A'(\alpha)}
=
\frac{P_n''(\alpha)}{P_n'(\alpha)}
+\frac{2}{1-\alpha}
=O(n).
\]
Consequently, along each dominant branch,
\[
|z_\alpha(\varepsilon)|
=
|\alpha|\bigl(1+\varepsilon+O(n\varepsilon^2)\bigr).
\]
The same expansion also identifies the sign of the second-order modulus
correction. Write
\[
z_\alpha(\varepsilon)
=
\alpha\left(1+\varepsilon+d_\alpha\varepsilon^2+O(\varepsilon^3)\right),
\qquad
d_\alpha:=\frac{c_2(\alpha)}{\alpha}
=1+\frac{\alpha A''(\alpha)}{2A'(\alpha)}.
\]
Then
\[
|z_\alpha(\varepsilon)|
=
|\alpha|\left(1+\varepsilon+\Re(d_\alpha)\varepsilon^2
+O(\varepsilon^3)\right).
\]
We check \(\Re(d_\alpha)>0\) on the dominant branches directly. Using
\(A=-P_n/(1-z)\) at a non-unit root,
\[
d_\alpha
=
\frac{1}{1-\alpha}
+\frac{\alpha P_n''(\alpha)}{2P_n'(\alpha)} .
\]
Let
\[
r_\alpha:=\alpha^a .
\]
The derivative formulas above give the exact identity
\[
\frac{\alpha P_n''(\alpha)}{P_n'(\alpha)}
=
n\,
\frac{r_\alpha-p(1-\theta_n)^2}
{r_\alpha-p(1-\theta_n)}
-1 .
\]
For dominant roots, \(r_\alpha=e^{-ix_m}e^{-\theta_n\mu(x_m,\theta_n)}\) with
\(|x_m|=2\pi/n\), so \(r_\alpha=1+O(n^{-1})\). Hence
\[
\frac{r_\alpha-p(1-\theta_n)^2}
{r_\alpha-p(1-\theta_n)}
=
\frac{q+2p\theta_n-p\theta_n^2}{q+p\theta_n}
+O(n^{-1}).
\]
The real part of the leading ratio is bounded below by a positive constant
uniformly over \(\theta_n\in[0,1]\), because \(q>0\) and
\[
q+2p\theta_n-p\theta_n^2 \ge q .
\]
The remaining term has nonnegative real part. Indeed, if
\(\alpha=re^{\mathrm{i}\varphi}\) with \(r<1\), then
\[
\Re\frac{1}{1-\alpha}
=
\frac{1-r\cos\varphi}{|1-\alpha|^2}
>0 .
\]
Consequently the \(n\)-order positive term above cannot be offset by
\((1-\alpha)^{-1}\). There is a constant \(c_*>0\) such that
\[
\Re(d_\alpha)\ge c_* n
\]
on the dominant branches for all sufficiently large \(n\). This proves that
the second-order modulus correction is positive.

We now pass from the branch expansion to the spectral radius. No genericity
assumption on \(p\) is needed. The threshold-scale argument only uses the two
dominant phases \(|x_m|=2\pi/n\), whose roots are simple uniformly by the
calculation above. Possible multiple roots are non-dominant by
Lemma~\ref{lem:multiple_roots_separated}; they are already \(\Omega(n^{-1})\)
inside the unit disk and therefore cannot affect an \(O(n^{-3})\) crossing.

For \(0<|x_m|\le\delta\), the same logarithmic-coordinate calculation,
formalized in Lemma~\ref{lem:uniform_log_perturbation} below, gives
\[
\nu_m(\eta)=\mu(x_m,\theta_n)-\eta+O(\eta^2),
\qquad \eta=n\varepsilon,
\]
uniformly in \(m,a,n\). Hence
\[
|z_m(\varepsilon)|
=
|\alpha_m|\{1+\varepsilon+O(n\varepsilon^2)\}
\]
for all small-phase branches. Among these branches, Proposition~\ref{prop:spectral_gap_scaling}
shows that the two dominant phases are separated from all other small phases by
at least \(c n^{-3}\) in limiting modulus. At the threshold scale
\(\varepsilon\le E n^{-3}\), the perturbation error \(O(n\varepsilon^2)\) is
only \(O(n^{-5})\), so the dominant branches remain the only candidates for the
spectral radius near the crossing. For \(|x_m|\ge\delta\),
Lemma~\ref{lem:near_unit_chart} gives \(\Re\mu\ge c_\delta\), so the limiting
roots are at radial distance \(\Omega(n^{-1})\) from the unit circle. By
Rouch\'e's theorem on contours enclosing the corresponding root clusters
\(A-\varepsilon zA'\) has the same multiplicities as \(A\) there, and those
finite-\(\varepsilon\) clusters remain at distance \(\Omega(n^{-1})\) from the
unit circle for \(\varepsilon\le E n^{-3}\), after reducing \(E\) if needed.
Combining these estimates yields
\[
\left|\rho(\varepsilon)-\rho_\infty(1+\varepsilon)\right|
\le K_\rho n\varepsilon^2
\]
for \(0\le\varepsilon\le E n^{-3}\), which is~\eqref{eq:rho_l0}.

\begin{appblemma}[Uniform \finiteinput{} perturbation in logarithmic coordinates]
\label{lem:uniform_log_perturbation}
Fix \(p\in(0,1)\). There are constants \(\delta,\eta_0,C,N>0\), depending only
on \(p\), such that the following holds for all coprime \(1\le a<n\), \(n\ge N\).
For every nonzero phase \(0<|x_m|\le\delta\), the \finiteinput{} equation has a
unique logarithmic branch
\[
z_{m,n}(\eta)=\omega_{m,n}\exp(-\nu_{m,n}(\eta)/n),
\qquad
\eta:=n\varepsilon,
\]
with \(\nu_{m,n}(0)=\mu(x_m,\theta_n)\), for \(0\le\eta\le\eta_0\), and
\[
\nu_{m,n}(\eta)
=
\mu(x_m,\theta_n)-\eta+R_{m,n}(\eta),
\qquad
|R_{m,n}(\eta)|\le C\eta^2 .
\]
\end{appblemma}

\noindent The proof is deferred to Appendix~\ref{app:aux_lemma_proofs}.

\begin{appblemma}[No stable islands before the asymptotic threshold]
\label{lem:no_stable_islands}
Let \(\Delta_n:=1-\rho_\infty\). There are constants \(C,N>0\), depending only
on \(p\), such that for all coprime \(1\le a<n\), \(n\ge N\), the \finiteinput{}
polynomial is unstable for every physical input value satisfying
\[
\Delta_n+C n\Delta_n^2
\le
\varepsilon
\le
\frac1n .
\]
Consequently, the first stable integer \inputlen{} has the same asymptotic
location as the local crossing of the two dominant branches.
\end{appblemma}

\noindent The proof is deferred to Appendix~\ref{app:aux_lemma_proofs}.

\subsubsection{Proof of Corollary~\ref{cor:l0_threshold}}

Let \(\Delta_n:=1-\rho_\infty\). Proposition~\ref{prop:spectral_gap_scaling}
gives \(\Delta_n=\Theta(n^{-3})\). By Theorem~\ref{thm:spectral_drift},
\[
\rho(\varepsilon)
=
\rho_\infty(1+\varepsilon)+O(n\varepsilon^2).
\]
The proof of Proposition~\ref{prop:spectral_gap_scaling} shows that every
non-dominant limiting branch is below the dominant modulus by at least
\(c n^{-3}\). At the threshold scale \(\varepsilon=O(\Delta_n)=O(n^{-3})\),
the finite-\(l_0\) perturbation error is
\[
O(n\varepsilon^2)=O(n\Delta_n^2)=O(n^{-5}),
\]
which is smaller than this separation. Thus the dominant branches alone
determine the first crossing of the unit circle near the stability threshold:
for a sufficiently large constant \(C\), a dominant branch has modulus larger
than one whenever
\[
\varepsilon\ge \Delta_n+C n\Delta_n^2
\]
within the threshold-scale neighborhood. Lemma~\ref{lem:no_stable_islands}
extends this instability to the whole physical range above the crossing window,
up to \(\varepsilon\le1/n\). Conversely, all roots lie strictly inside the unit
disk whenever
\[
\varepsilon\le \Delta_n-C n\Delta_n^2 .
\]
The stability boundary is therefore determined up to an \(O(n\Delta_n^2)\)
error by
\[
1=\rho_\infty(1+\varepsilon)+O(n\varepsilon^2).
\]
Since the boundary has \(\varepsilon=\Theta(\Delta_n)\), this gives
\[
\varepsilon
=
\Delta_n+O(n^{-5}),
\qquad
\frac{\Delta_n}{\varepsilon}
=
1+O(n^{-2}).
\]
Because \(\varepsilon=(l_0+n)^{-1}\) is strictly decreasing in \(l_0\), the
instability interval in Lemma~\ref{lem:no_stable_islands} covers every physical
integer input length below the local crossing window. Integer rounding in
\(l_0\) does not affect this order. If the real threshold
is \(s=l_0+n\), replacing \(s\) by a neighboring integer changes
\(s\Delta_n\) by at most \(\Delta_n=O(n^{-3})\), which is absorbed by the
claimed \(O(n^{-2})\) tolerance. Using
\(\varepsilon=(l_{0,\min}+n)^{-1}\), with
Lemma~\ref{lem:no_stable_islands} ruling out earlier stable islands, yields
\[
(l_{0,\min}+n)(1-\rho_\infty)=1+O(n^{-2}),
\]
which is the claimed bound after enlarging the constant.

This corollary implies Proposition~\ref{prop:finite_prompt_threshold}. Indeed,
Proposition~\ref{prop:spectral_gap_scaling} gives
\[
1-\rho_\infty
=
\frac{2\pi^2p(1-p)}{((1-p)+p\theta_n)^3}\,n^{-3}
+O(n^{-4}),
\qquad
\theta_n=a/n.
\]
Here and above, for each fixed \(p\in(0,1)\), the constants in the
\(O(\cdot)\) terms are uniform in \(n\) and in the coprime sequence.
The corollary gives
\[
l_{0,\min}+n
=
\frac{1+O(n^{-2})}{1-\rho_\infty}.
\]
Substituting the spectral-gap expansion and using
\(\theta_n\to\theta\in[0,1)\) yields
\[
l_{0,\min}
=
\frac{((1-p)+p\theta)^3}{2\pi^2p(1-p)}\,n^3
+o(n^3).
\]
Since \(n=l_{1,2}\), this is exactly the asymptotic threshold stated in
Proposition~\ref{prop:finite_prompt_threshold}.

\subsection{Extension to Multiple Request Types}
\label{sec:K_type_extension}

We first explain the \(K\geq 3\) extension for request classes with decoding
lengths \(l_{1,1}<\cdots<l_{1,K}\), proportions \(p_1,\ldots,p_K\), and a common
\inputlen{} \(l_0\). Heterogeneous \inputlens{} are handled separately in
Section~\ref{sec:hetero_prompt_extension}. Write \(d:=l_{1,K}\). The argument
uses the same embedded linearization as in the two-class case; the new object is
the survival weight \(\alpha_s\), which replaces the two-class piecewise
normalization. The main substitutions are the following: the class
composition at a stage is summarized by \(\alpha_s\), the memory weight becomes
\(V_s=\alpha_s(l_0+s+1)\), and the two-class polynomial is replaced by the
survival polynomial in~\eqref{eq:K_type_limiting}. The following paragraphs
give the corresponding root-location, \finiteinput{} perturbation, and
global-convergence checks under these substitutions.

\subsubsection{Cohort reduction}

Under per-iteration proportional admission \(a^n_k=p_k a^n\) and
proportional within-stage eviction, the full class-stage occupancy process
collapses to a \(d\)-dimensional \emph{cohort manifold}:
\[
x^n_{k,s}=p_k Y^n_s\;\;(s<l_{1,k}),\qquad x^n_{k,s}=0\;\;(s\ge l_{1,k}),
\]
for a single vector \(Y^n\in\mathbb{R}_+^d\). Shift preserves within-stage
proportions, proportional eviction scales all classes in a mixed stage by the same
factor, and proportional admission injects a new mixed cohort. From any feasible
initial active state, after at most \(d\) iterations all pre-existing requests
have completed or been evicted; thereafter every surviving cohort was admitted
under the proportional rule, so the process lies on this manifold.

Define the \emph{survival weight} and \emph{effective stage weight}:
\[
\alpha_s:=\sum_{k:l_{1,k}>s}p_k,\qquad V_s:=\alpha_s(l_0+s+1),\qquad s=0,\ldots,d-1.
\]
The survival weight \(\alpha_s\) is nonincreasing, with \(\alpha_0=1\) and
\(\alpha_{d-1}=p_K>0\); it drops exactly at the decoding lengths. Total memory
is \(\sum_{s=0}^{d-1}V_s Y^n_s\), and the eviction-free equilibrium is
\(Y^*=x^*\mathbf{1}\), where \(x^*=M/\sum_s V_s\).

\subsubsection{Characteristic polynomial}

On the cohort manifold, the linearized no-eviction dynamics have the same
companion form as in the two-class case:
\(y^n=\Phi y^{n-1}\), with first-row entries \(-V_{j+1}/V_0\). The nonzero
eigenvalues are the roots of
\[
F_{l_0}(z) = \sum_{s=0}^{d-1}V_sz^{d-1-s} = (l_0+d)A(z)-z A'(z),
\]
where the limiting polynomial is
\begin{equation}\label{eq:K_type_limiting}
(1-z)A(z) = -z^d + \sum_{k=1}^K p_kz^{d-l_{1,k}}.
\end{equation}

\subsubsection[Root structure of A]{Root structure of $A(z)$}

All roots of \(A\) lie in the closed unit disk. Indeed, if
\(|z|>1\), then
\[
|z|^d\le\sum_k p_k|z|^{d-l_{1,k}}<|z|^d,
\]
a contradiction. To locate the unit-circle roots, note that
\(\gcd(d-l_{1,1},\ldots,d-l_{1,K-1},d)=g\): any divisor of \(d\) and
\(d-l_{1,i}\) also divides \(l_{1,i}\). If \(|z|=1\) and \(A(z)=0\),
equality in the triangle inequality forces \(z^{d-l_{1,k}}=1\) for every
\(k\), and B\'ezout's identity gives \(z^g=1\). Since
\(A(1)=\sum_k p_k l_{1,k}>0\), only non-trivial \(g\)-th roots can be roots of
\(A\). At each such root \(\omega\), differentiating
\eqref{eq:K_type_limiting} gives
\[
(1-\omega)A'(\omega)=-\omega^{-1}\sum_k p_k l_{1,k},
\]
so \(A'(\omega)\neq0\).

\subsubsection[Instability when g is greater than one]{Instability when $g>1$}

Each non-trivial \(g\)-th root \(\omega\) satisfies \(A(\omega)=0\) and
\(A'(\omega)\neq0\). Applying the implicit-function theorem to
\[
G(z,\varepsilon)=A(z)-\varepsilon z A'(z),
\qquad
\varepsilon=\frac{1}{l_0+d},
\]
gives \(z'(0)=\omega\), and hence
\[
z(\varepsilon)=\omega(1+\varepsilon)+O(\varepsilon^2).
\]
Thus the \finiteinput{} correction induces the same outward radial drift as in
the two-class system. Each unit-circle root moves outside the unit disk for
sufficiently small positive \(\varepsilon\), producing \(g-1\) unstable
eigenvalues.

\subsubsection[Local stability when g equals one]{Local stability when $g=1$}

When \(g=1\), all roots of \(A\) lie strictly inside the unit circle. Let
\(r_*:=\max_j|\alpha_j|<1\), and choose \(\rho\in(r_*,1)\). On
\(|z|=\rho\), \(|A(z)|\ge m_\rho>0\), while
\(|\varepsilon z A'(z)|\le \varepsilon M_\rho\). For \(l_0\) large
enough that \(\varepsilon M_\rho<m_\rho\), Rouch\'e's theorem implies that
\(F_{l_0}/(l_0+d)\) and \(A\) have the same number of roots in
\(|z|<\rho\), namely all \(d-1\). Therefore
\(\rho(\Phi)\le\rho<1\).

\subsubsection[Global convergence when g equals one]{Global convergence when $g=1$}

The global-convergence check follows the structure of
Section~\ref{sec:global_convergence} after replacing the two-class normalization
weights by \(\alpha_s\). The only points that require checking are the
population bound, the size of the eviction perturbation, and the no-eviction
neighborhood.
All three checks use the same estimates as in the two-class proof. Since
\(\alpha_s\ge p_K>0\), the in-service population
\[
N^n=\sum_s\alpha_s Y^n_s\le \frac{M}{l_0+1}
\]
is \(O(1)\), uniformly in \(K\). Each shift increases memory by at most \(N^n\)
tokens. Restoring feasibility therefore requires evicting only \(O(1/l_0)\)
requests per step, so the perturbation from the no-eviction linearized dynamics
satisfies \(\|e^n\|_2=O(1/l_0)\) in the \(Y\)-coordinates. The Schur bound for
\(\Phi\) gives the Lyapunov function
\(\mathcal{L}(y)=y^\top Py\) with \(I\preceq P\preceq C_P I\); here \(C_P\) may
depend on the fixed \(K\)-class instance but not on \(l_0\). Therefore
\[
\mathcal{L}(y^n)-\mathcal{L}(y^{n-1})\le -\|y^{n-1}\|_2^2+O(1/l_0).
\]
For \(\|y^{n-1}\|_2\ge\delta\), this drift is strictly negative when \(l_0\)
is large enough. The no-eviction neighborhood is obtained exactly as in the
two-class proof: the equilibrium leaves a stage-0 admission margin of order
\(l_0\), while the memory deviation is at most \(O(l_0\|y\|_2)\). Choosing
\(\delta>0\) small enough therefore rules out eviction whenever
\(\|y\|_2<\delta\). Since \(\mathcal{L}\ge0\), the embedded trajectory reaches
this neighborhood in finite time. Thereafter \(e^n=0\), and the Schur-stable
linear dynamics imply exponential convergence.

\subsubsection[Spectral gap scaling for K at least three]{Spectral gap scaling for $K\geq 3$}

The log-coordinate parameterization of Section~\ref{sec:tight_bound_proof}
extends to \(K\) classes. For the polynomial~\eqref{eq:K_type_limiting}, write
\(n:=l_{1,K}\) and parameterize the \(u\)-th branch as
\(z=\omega_u\exp(-\mu/n)\), where \(\omega_u\) is an \(n\)-th root of
unity. The branch equation becomes
\[
\sum_{k=1}^K p_k e^{\mathrm{i}x_k(u)}e^{\beta_{k,n}\mu} = 1,
\qquad x_k(u) := -\frac{2\pi u l_{1,k}}{n}\!\!\pmod{2\pi},
\]
with \(\beta_{k,n}:=l_{1,k}/n\) and \(x_K(u)=0\). The implicit-function theorem
gives a local branch \(\mu(x)\) with
\[
\Re\mu(x) = \frac{1}{2D}\sum_{k=1}^K p_k\Bigl(x_k - \beta_k\bar{x}\Bigr)^2 + O(\|x\|^3),
\qquad D:=\sum_k p_k\beta_k,\quad \bar{x}:=\frac{\sum_k p_k x_k}{D}.
\]
The quadratic term is positive definite; in particular, the \(k=K\) term
contributes \(p_K\bar{x}^2\). Thus the spectral gap is governed by the smallest
nonzero value of \(\Re\mu\) over the discrete phases \(x(u)\),
\(u=1,\ldots,n-1\).

\noindent{\bf Fixed shorter lengths.} If \(l_{1,1},\ldots,l_{1,K-1}\) are fixed and
\(n\to\infty\), the branch \(u^*\) satisfying
\(\gcd(l_{1,1},\ldots,l_{1,K-1})\,u^*\equiv 1\pmod n\) yields phases
\(x_k=O(1/n)\). The resulting gap has the cubic scaling
\[
1-\rho_\infty = c_{\mathrm{fix}}\,n^{-3}+O(n^{-4}),
\quad
c_{\mathrm{fix}} = \frac{2\pi^2}{p_K}\sum_{k=1}^{K-1}p_k(a_k^\#)^2 + \frac{2\pi^2}{p_K^2}\Bigl(\sum_{k=1}^{K-1}p_k a_k^\#\Bigr)^2,
\]
where \(a_k^\#:=l_{1,k}/\gcd(l_{1,1},\ldots,l_{1,K-1})\). For \(K=2\), this reduces to
\(2\pi^2p/q^2\).

\noindent{\bf Bounded resonance.} More generally, suppose there exist a fixed
integer \(\bar u\) and bounded integers \(c_1,\ldots,c_{K-1}\) such that
\(\bar u l_{1,k}\equiv c_k\pmod n\) for all \(k\). Then
\[
1-\rho_\infty = c_{\mathrm{res}}\,n^{-3}+O(n^{-4}),
\quad
c_{\mathrm{res}} = \frac{2\pi^2}{D}\sum_{k=1}^K p_k\Bigl(c_k - \beta_k\frac{\sum_j p_j c_j}{D}\Bigr)^2,
\]
with \(c_K:=0\). This includes near-consecutive families such as
\((n,n{+}1,\ldots,n{+}K{-}1)\).

\noindent{\bf Generic ratios.} When no bounded resonance exists, the minimizing
branch index grows with \(n\). The gap exponent then depends on the simultaneous
Diophantine approximation properties of
\((\beta_1,\ldots,\beta_{K-1})\). The numerical evidence is consistent with the
generic exponent \(1+2/(K-1)\), giving \(\alpha=2\) for \(K=3\) and
\(\alpha=5/3\) for \(K=4\). A proof of this generic exponent remains open.

\subsection{Heterogeneous \Inputterm{} Lengths}
\label{sec:hetero_prompt_extension}

We now justify Remark~\ref{rem:hetero_prompt}. The \commoninput{} proof above
sets \(l_{0,1}=l_{0,2}\) only to keep the recurrence and polynomial notation
light. The \heteroinput{} case differs only in the memory weights carried
by surviving cohorts, not in the completion phases. If heterogeneous inputs
scale as
\[
l_{0,k}=r_k L+O(1),\qquad r_k>0,
\]
then the same cohort-balance derivation replaces the \commoninput{} limiting
polynomial by a weighted survival polynomial. For \(K\) classes, let
\(R=\sum_{k=1}^K p_k r_k\) and \(l_{1,K}=\max_k l_{1,k}\). The limiting
polynomial can be written as
\[
A_r(z)
=
\sum_{m=0}^{l_{1,K}-1}
\left(\sum_{k:\,l_{1,k}>m} p_k r_k\right)
z^{l_{1,K}-1-m},
\]
or equivalently,
\[
(1-z)A_r(z)
=
-Rz^{l_{1,K}}
+\sum_{k=1}^K p_k r_k\,z^{l_{1,K}-l_{1,k}}.
\]
Thus heterogeneous inputs do not change the timing exponents
\(l_{1,K}-l_{1,k}\); they only replace the class probabilities \(p_k\) by
positive weights \(p_k r_k\) in the survival coefficients. The triangle
inequality argument used in Lemmas~\ref{lem:bounded}--\ref{lem:primitive}
therefore carries over after normalizing by \(R\): equality on the unit
circle still forces all active phases \(z^{l_{1,k}}\) to agree, which is
exactly the same root-of-unity condition governed by
\(\gcd(l_{1,1},\ldots,l_{1,K})\).

The remaining steps follow the same outline as in the \commoninput{} proof, with
constants changed by the weights \(p_k r_k\). For \(g>1\),
the non-trivial \(g\)-th roots still satisfy \(A_r(\omega)=0\), and
differentiating the weighted closed form gives
\[
(1-\omega)A_r'(\omega)
=-\omega^{-1}\sum_{k=1}^K p_k r_k l_{1,k}\neq0.
\]
The corresponding finite-\(L\) perturbation has a positive radial first-order
coefficient, whose value depends on \(p_k\) and \(r_k\), so the same
root-of-unity mechanism produces outward drift. For \(g=1\), all
non-equilibrium limiting roots are strictly inside the unit disk, so Rouch\'e's
theorem keeps the finite-\(L\) roots inside for sufficiently large \(L\). The
global Lyapunov step is checked by the same estimates: because each \(r_k>0\)
is fixed, all effective stage weights are \(\Theta(L)\), the active population
remains \(O(1)\), and each eviction perturbation is \(O(1/L)\).
Thus heterogeneous inputs change constants through the positive ratios \(r_k\),
while preserving the spectral and Lyapunov mechanisms used in the two-class
\commoninput{} case.

\begin{remark}[Stable initial conditions in non-coprime systems]
When \(g>1\), some initial conditions remain stable because the instability is
spectral rather than global across all directions. The general solution admits a
Vandermonde decomposition into the equilibrium mode \(z_0=1\), unstable
modes \(z_1,\ldots,z_{g-1}\) with \(|z_j|>1\), and stable
modes \(z_g,\ldots,z_{l_{1,2}-1}\) with \(|z_j|<1\). Stability
holds exactly when the initial perturbation has no projection onto the unstable
subspace. The equilibrium \(X^*=x^*(1,\ldots,1)^\top\) itself lies outside that
unstable subspace.
\end{remark}

\subsection{Pulse Cycle}
\label{sec:pulse_cycle_return}

This subsection gives a complementary finite-\(l_0\) calculation for the
non-coprime case \(g>1\). The finite-exit result above shows that synchronized
GCD modes push generic local perturbations out of the no-eviction neighborhood;
the instance below makes the resulting recurrent-eviction behavior explicit
by exhibiting a two-class \commoninput{} pulse cycle and computing its local
return map.
Recall that we require
\[
a^n_1 = p a^n,\qquad a^n_2 = q a^n
\]
at every iteration. If LPF partially evicts a mixed stage, it removes classes 1
and 2 in proportion to their occupancies in that stage.

\begin{appblemma}[Cohort reduction for the strict mixing rule]
\label{lem:strict_cohort}
Consider the two-class \commoninput{} system with decoding lengths \(l_{1,1}<l_{1,2}\),
and write \(a:=l_{1,1}\) and \(b:=l_{1,2}\). Under strict per-iteration proportional
admission and proportional within-stage eviction, the class-stage process admits
a scalar cohort representation \(y^{n}_{j}\) such that
\[
x^n_{1,j}=
\begin{cases}
p\,y^{n}_{j}, & j<a,\\
0, & j\ge a,
\end{cases}
\qquad
x^n_{2,j}=q\,y^{n}_{j},\qquad 0\le j\le b-1.
\]
Moreover, the normalized variables of the coprime proof satisfy
\[
Y^n_j=y^{n}_{j},
\qquad 0\le j\le b-1,
\]
and the resulting cohort dynamics have effective weights
\[
V_j=
\begin{cases}
l_0+1+j, & 0\le j\le a-1,\\
q(l_0+1+j), & a\le j\le b-1.
\end{cases}
\]
\end{appblemma}

\noindent The proof is deferred to Appendix~\ref{app:aux_lemma_proofs}.

By Lemma~\ref{lem:strict_cohort}, the strict mixing rule can be analyzed through
the scalar cohort variables and effective weights \(V_j\). We use this reduced
representation to construct the recurrent pulse cycle.

\begin{appbproposition}[Two-class \commoninput{} exact pulse cycle]
\label{prop:exact_pulse_cycle}
Assume \(g=\gcd(a,b)>1\), and write \(h:=b/g\). For
\(r=0,\ldots,g-1\), define
\[
J_r=\{r,r+g,\ldots,r+(h-1)g\},
\qquad
\beta=\frac{M}{\sum_{m=0}^{h-1}V_{g-1+mg}},
\]
and
\[
\alpha_r
=
\frac{M-\beta\sum_{m=1}^{h-1}V_{r+mg}}{V_r}.
\]
Let \(Z^{(r)}\in\mathbb R^b\) be the state whose only nonzero coordinates are
\[
Z^{(r)}_{r}=\alpha_r,
\qquad
Z^{(r)}_{r+mg}=\beta,\qquad m=1,\ldots,h-1.
\]
Then
\[
Z^{(0)}\to Z^{(1)}\to \cdots \to Z^{(g-1)}\to Z^{(0)}
\]
is an exact period-\(g\) orbit. Its average throughput is
\[
\bar{T}_{\mathrm{pulse}}
=
\frac{2M}{a(2l_0+a+g)+q(b-a)(2l_0+a+b+g)}.
\]
\end{appbproposition}

\begin{proof}
Write \(u=l_0+1\). For \(r=0,\ldots,g-1\), let
\[
S_r:=\sum_{m=1}^{h-1}V_{r+mg}.
\]
Then \(\alpha_r=(M-\beta S_r)/(u+r)\). Within each branch of the weight
profile, \(V_{j+g}-V_j\) is constant, so
\[
S_r=S_0+c r,
\qquad
c=\frac{a}{g}-1+q\!\left(\frac{b}{g}-\frac{a}{g}\right)\ge 0.
\]
Hence
\[
\alpha_r-\alpha_{r+1}
=
\frac{M-\beta S_0+\beta c u}{(u+r)(u+r+1)}>0
\qquad
\text{for }r=0,\ldots,g-2,
\]
because \(M-\beta S_0=V_0\alpha_0>0\). Also,
\[
\alpha_{g-1}
=
\frac{M-\beta\sum_{m=1}^{h-1}V_{g-1+mg}}{V_{g-1}}
=
\beta.
\]

Fix \(r\le g-2\). Starting from \(Z^{(r)}\), the execute step shifts the
support from \(J_r\) to \(J_{r+1}\). The youngest occupied age \(r+1\) carries
mass \(\alpha_r\), and all older occupied ages in \(J_{r+1}\) carry mass
\(\beta\). The post-execute memory exceeds \(M\) by exactly
\[
V_{r+1}(\alpha_r-\alpha_{r+1})>0.
\]
This overflow is smaller than the full youngest-stage load
\(V_{r+1}\alpha_r\). Hence LPF trims only the youngest occupied age, reducing
its mass from \(\alpha_r\) to \(\alpha_{r+1}\). No new admission occurs, so the
next state is \(Z^{(r+1)}\).

At phase \(r=g-1\), every occupied age already carries mass \(\beta\). Execute
removes the age-\((b-1)\) cohort, shifts the remaining \(h-1\) cohorts to ages
\(g,2g,\ldots,(h-1)g\), and creates slack
\[
M-\sum_{m=1}^{h-1}V_{mg}\beta = V_0\alpha_0.
\]
No eviction is needed. Admission therefore inserts a new age-0 cohort of mass
\(\alpha_0\), reconstructing \(Z^{(0)}\).

Finally,
\[
\sum_{m=0}^{h-1}V_{g-1+mg}
=
\frac{a(2l_0+a+g)+q(b-a)(2l_0+a+b+g)}{2g}.
\]
Since one cohort of mass \(\beta\) completes every \(g\) iterations,
\[
\bar{T}_{\mathrm{pulse}}=\frac{\beta}{g}
=
\frac{2M}{a(2l_0+a+g)+q(b-a)(2l_0+a+b+g)}.
\]
\end{proof}

\begin{appbproposition}[Local return map for the two-class pulse cycle]
\label{prop:pulse_return_map}
In a neighborhood of the pulse cycle from
Proposition~\ref{prop:exact_pulse_cycle}, the phase maps are affine. For
\(r=0,\ldots,g-2\), if
\[
z=(z_0,\ldots,z_{h-1})
\]
denotes the masses on the support \(J_r\), then
\[
\Phi_r(z_0,\ldots,z_{h-1})
=
\left(
\frac{M-\sum_{m=1}^{h-1}V_{r+1+mg}z_m}{V_{r+1}},
z_1,\ldots,z_{h-1}
\right).
\]
At the wrap phase,
\[
\Phi_{g-1}(z_0,\ldots,z_{h-1})
=
\left(
\frac{M-\sum_{m=1}^{h-1}V_{mg}z_{m-1}}{V_0},
z_0,\ldots,z_{h-2}
\right).
\]
Let \(\mathcal R=\Phi_{g-1}\circ\cdots\circ\Phi_0\) be the \(g\)-step return
map at phase \(0\). Then \(D\mathcal R\) has one exact zero eigenvalue, and its
remaining eigenvalues are the roots of the compressed two-class polynomial
\begin{equation}
\label{eq:compressed_pulse_poly_app}
\widehat F(z)
=
\sum_{m=0}^{a/g-1}\Bigl(\frac{l_0}{g}+1+m\Bigr)z^{h-1-m}
+
q\sum_{m=a/g}^{h-1}\Bigl(\frac{l_0}{g}+1+m\Bigr)z^{h-1-m}.
\end{equation}
\end{appbproposition}

\begin{proof}
The first point is local branch consistency. The strict inequalities that define
the pulse cycle also define an open neighborhood on which the LPF branch does
not change.
The strict inequalities \(\alpha_r>\alpha_{r+1}\) for \(r\le g-2\), together
with \(\alpha_0>0\), persist in a small neighborhood of each phase state. On
that neighborhood, the LPF trim pattern is fixed: for \(r\le g-2\), only the
youngest occupied age is partially trimmed; for \(r=g-1\), no eviction occurs
and admission fills the slack at age \(0\). The displayed affine maps follow
directly from execute, LPF trim, and admission.

We next reduce the \(g\)-phase return map to the nontrivial coordinates.
Now write
\[
x=(x_0,\ldots,x_{h-2}) := (z_1,\ldots,z_{h-1}),
\qquad
W_m:=V_{g-1+mg},\quad m=0,\ldots,h-1.
\]
Every non-wrap phase deletes the current youngest coordinate and leaves the
older coordinates unchanged. Therefore the \(g\)-step return map is determined
by \(x\) alone:
\[
\mathcal R_{\mathrm{red}}(x_0,\ldots,x_{h-2})
=
\left(
\frac{M-\sum_{m=1}^{h-1}W_m x_{m-1}}{W_0},
x_0,\ldots,x_{h-3}
\right).
\]
Its Jacobian is the companion matrix
\[
K=
\begin{pmatrix}
-W_1/W_0 & -W_2/W_0 & \cdots & -W_{h-1}/W_0 \\
1 & 0 & \cdots & 0 \\
0 & 1 & \ddots & 0 \\
\vdots & & \ddots & \vdots \\
0 & 0 & \cdots & 1 & 0
\end{pmatrix},
\]
so the full Jacobian has block form
\[
D\mathcal R=
\begin{pmatrix}
0 & *\\
0 & K
\end{pmatrix},
\]
This block-triangular form gives one exact zero eigenvalue. The remaining
\(h-1\) eigenvalues are the roots of the characteristic polynomial of \(K\).

It remains only to identify this companion polynomial. The dictionary is the
compressed coprime two-class dictionary
\[
\bar a:=a/g,\qquad \bar b:=b/g=h,\qquad \ell_0:=l_0/g,
\]
with effective weights sampled along the \(g\)-spaced pulse support,
\(W_m:=V_{g-1+mg}\). In these variables,
\[
\frac{W_m}{g}=\frac{V_{g-1+mg}}{g}
=
\begin{cases}
\ell_0+1+m, & 0\le m\le \bar a-1,\\
q(\ell_0+1+m), & \bar a\le m\le \bar b-1,
\end{cases}
\]
so the companion polynomial of \(K\) is exactly
\[
\widehat F(z)
=
\sum_{m=0}^{\bar a-1}(\ell_0+1+m)z^{\bar b-1-m}
+
q\sum_{m=\bar a}^{\bar b-1}(\ell_0+1+m)z^{\bar b-1-m},
\]
which is~\eqref{eq:compressed_pulse_poly_app}.
\end{proof}

\begin{appbcorollary}[Local pulse stability]
\label{cor:pulse_local}
The pulse cycle from Proposition~\ref{prop:exact_pulse_cycle} is locally
asymptotically stable if and only if the polynomial \(\widehat F\)
in~\eqref{eq:compressed_pulse_poly_app} is Schur stable. In particular, for
generic \(q\in(0,1)\) and sufficiently large \(l_0\), the pulse cycle is locally
asymptotically stable.
\end{appbcorollary}

\begin{proof}
The local return map is affine in a neighborhood of the pulse cycle, so local
asymptotic stability is equivalent to \(\rho(D\mathcal R)<1\).
Proposition~\ref{prop:pulse_return_map} reduces this condition to
\(\rho(K)<1\), equivalently Schur stability of \(\widehat F\). Thus the
compressed polynomial gives the local stability criterion for this recurrent-
eviction pulse cycle. The final claim
follows because the compressed pair
\((\bar a,\bar b)=(a/g,b/g)\) is coprime, and the proof of
Theorem~\ref{thm:local_stability} depends only on the coefficient profile of the
polynomial. Applying that argument to \(\widehat F\) gives Schur stability for
generic \(q\) and sufficiently large \(\ell_0=l_0/g\), hence for sufficiently
large \(l_0\).
\end{proof}

Corollary~\ref{cor:pulse_local} turns the return-map calculation into a local
stability criterion. The following example spells out this criterion for the
smallest non-coprime pair.

\begin{appbexample}[Pulse return map for \((2,4)\)]
\label{ex:pulse_return_map_24}
Let \(a=2\), \(b=4\), and \(g=2\). At phase \(0\), the state is \(z=(x,y)\) on ages \(0\) and \(2\), and
\[
\Phi_0(x,y)
=
\left(
\frac{M-V_3 y}{V_1},
y
\right)
=
\left(
\frac{M-q(l_0+4)y}{l_0+2},
y
\right).
\]
At phase \(1\), the state is \((u,v)\) on ages \(1\) and \(3\), and
\[
\Phi_1(u,v)
=
\left(
\frac{M-V_2 u}{V_0},
u
\right)
=
\left(
\frac{M-q(l_0+3)u}{l_0+1},
u
\right).
\]
Therefore
\[
D(\Phi_1\circ\Phi_0)
=
\begin{pmatrix}
0 & \dfrac{V_2V_3}{V_0V_1}\\[8pt]
0 & -\dfrac{V_3}{V_1}
\end{pmatrix},
\]
so the eigenvalues are
\[
z_1=0,
\qquad
z_2=-\frac{V_3}{V_1}
=
-\frac{q(l_0+4)}{l_0+2}.
\]
Hence the exact local stability condition is
\[
q<\frac{l_0+2}{l_0+4}.
\]
For the symmetric case \(p=q=1/2\), this inequality holds for every \(l_0\ge 1\).
\end{appbexample}

\subsection{Summary}

This appendix proves Theorem~\ref{thm:gcd} in the two-class \commoninput{}
normalization and explains the corresponding extension arguments:
\begin{enumerate}[nosep]
\item The characteristic polynomial decomposes as $F(z) = (l_0+l_{1,2})A(z) - z A'(z)$.
\item Roots of $A(z)$ on the unit circle are exactly the non-trivial $g$-th roots of unity (Lemmas~\ref{lem:bounded}--\ref{lem:primitive}).
\item For $g > 1$: IFT shows these roots drift outside the unit circle, causing instability (Theorem~\ref{thm:instability}).
\item For $g = 1$: All roots remain inside the unit circle; Lyapunov analysis proves global convergence from any feasible initial active state in the saturated-input model (Theorems~\ref{thm:local_stability} and~\ref{thm:global}).
\item The non-coprime two-class system also admits an exact period-\(g\) pulse cycle whose local return-map eigenvalues are given by a compressed coprime polynomial (Section~\ref{sec:pulse_cycle_return}).
\item The $K\geq 3$ and \heteroinput{} discussions use the same survival-polynomial argument with positive class weights (Sections~\ref{sec:K_type_extension}--\ref{sec:hetero_prompt_extension}).
\end{enumerate}

\section{Proofs of Auxiliary Lemmas}
\label{app:aux_lemma_proofs}
\makeatletter
\providecommand{\ECsetlemmaheadingorigin}[2]{\global\@namedef{EClemmaorigin@#1}{#2}}
\providecommand{\EClemmaheadingorigin}[1]{%
  \@ifundefined{EClemmaorigin@#1}{}{\space(\@nameuse{EClemmaorigin@#1})}%
}
\providecommand{\EClemmaheadingref}[2]{\ref{#1}\EClemmaheadingorigin{#1}}
\makeatother

\ECsetlemmaheadingorigin{lem:monotonicity}{main text}

This appendix collects the proofs of the lemmas used in the main text and Appendices~\ref{app:theorem1} and~\ref{app:theorem2}. The lemma statements are kept in the main proof flow, while the technical details are gathered here.

\subsection{Proof of Lemma~\EClemmaheadingref{lem:monotonicity}{1}}

\begin{proof}[Proof of Lemma~\EClemmaheadingref{lem:monotonicity}{1}]
Let $\mathcal{C}=\{X_1,\ldots,X_p\}$ be a limit cycle of period~$p$ on the memory boundary ($M^n=M$ at every iteration). Write $X^n=(x^{n}_{0},\ldots,x^{n}_{l_1-1})$ and let $y_n := x^{n}_{l_1-1}$ denote the number of completions at iteration~$n$, so that $\bar{T}(\mathcal{C})=\frac{1}{p}\sum_{n=1}^{p} y_n$.
By the waste decomposition~\eqref{eq:throughput_decomp}, $\bar{T}(\mathcal{C}) = x^* - E(\mathcal{C})/(pW) \le x^*$, with equality only when no eviction occurs ($E = 0$).
All subpopulations below are submasses of this continuous state. When a
subpopulation cuts through a stage, the within-stage tie-breaking is fixed so
that the non-protected submass is trimmed before the protected submass; this is
an admissible representative of the continuous LPF convention in
Section~\ref{sec:model}.

We now lower bound $\bar{T}(\mathcal{C})$ by a direct counting argument. Define
\[
m := \frac{M}{w_{l_1-1}}=\frac{M}{l_0+l_1}.
\]
Fix any starting index $t\in\{1,\ldots,p\}$ and consider the $l_1$ consecutive iterations $t,t{+}1,\ldots,t{+}l_1{-}1$ (with indices taken modulo~$p$). We claim that at least $m$ requests complete over this block:
\begin{equation}\label{eq:l1_window_lb}
\sum_{r=0}^{l_1-1} y_{t+r} \;\ge\; m.
\end{equation}
To see this, observe that at time~$t$ the total number of active requests is at least~$m$ because
\[
M=\sum_{j=0}^{l_1-1} w_j x^{t}_{j} \le w_{l_1-1}\sum_{j=0}^{l_1-1} x^{t}_{j}.
\]
Consider any subpopulation of exactly $m$ active request mass at time~$t$ with the \emph{largest} stage indices.\footnote{If the cutoff falls within a stage, take an arbitrary $m$-mass subpopulation from that stage; requests within a stage are indistinguishable in the continuous model.} Call this the \emph{protected} mass. At any future iteration before it completes, each protected unit occupies at most $w_{l_1-1}$ tokens, so the protected subpopulation uses at most $m\,w_{l_1-1}=M$ tokens at all times (including immediately after \textit{Execute}). Evicting \emph{only} non-protected mass therefore suffices to restore feasibility after any \textit{Execute} step: if we hypothetically removed every non-protected unit, the remaining protected mass would still fit in memory.

Under LPF, evictions occur in increasing order of stage index, so protected mass can only be evicted after all less-progressed non-protected mass has been removed. At the cutoff stage that may contain both protected and non-protected mass, the amount to remove never exceeds the non-protected mass, because the protected mass alone is feasible; by the continuous within-stage convention fixed above, the non-protected mass is evicted first. Consequently, the protected subpopulation is never evicted and advances one stage per iteration until completion. Because each unit traverses $l_1$ stages, all protected mass completes within the next $l_1$ iterations, proving~\eqref{eq:l1_window_lb}.

Summing~\eqref{eq:l1_window_lb} over $t=1,\ldots,p$ and noting that each completion $y_n$ appears in exactly $l_1$ of the resulting window sums gives
\[
l_1\sum_{n=1}^{p} y_n \;\ge\; pm.
\]
Dividing by $pl_1$ yields
\[
\bar{T}(\mathcal{C})=\frac{1}{p}\sum_{n=1}^{p} y_n \;\ge\; \frac{m}{l_1}=\frac{M}{l_1(l_0+l_1)}=\bar{T}_{l_1-1}.
\]

We next characterize when equality holds. Define the total number of active requests at time~$n$ by
\[
N_n := \sum_{j=0}^{l_1-1} x^{n}_{j}.
\]
Since $w_j\le w_{l_1-1}$ for all $j$, the memory constraint implies
\begin{equation}\label{eq:Nn_lb_m}
N_n \ge \frac{M}{w_{l_1-1}}=m,
\end{equation}
with equality if and only if $x^{n}_{0}=\cdots=x^{n}_{l_1-2}=0$ (i.e., all active requests are at stage~$l_1-1$).

We claim that if there exists some index $t$ with $y_t>0$ and $N_t>m$, the corresponding length-$l_1$ completion window is \emph{strictly} larger than~$m$:
\begin{equation}\label{eq:l1_window_strict}
\sum_{r=0}^{l_1-1} y_{t+r} \;>\; m.
\end{equation}
Indeed, set
\[
\delta := \min\{y_t,\; N_t-m\},
\]
which is strictly positive under $y_t>0$ and $N_t>m$. Consider the subpopulation of size $m{+}\delta$ consisting of the active request mass at time~$t$ with the largest stage indices, using the same within-stage convention as above. This enlarged protected subpopulation includes all stage-$(l_1{-}1)$ mass, so exactly $y_t$ of it completes at iteration~$t$. Immediately after the \textit{Execute} step at iteration~$t$, the remaining enlarged protected mass equals
\[
(m+\delta)-y_t \le m,
\]
where we used $\delta \le y_t$. Since every unit occupies at most $w_{l_1-1}$ tokens, the remaining protected mass fits in memory: its total footprint is at most $m\,w_{l_1-1}=M$ (and thereafter it only decreases as further completions occur). Consequently, evicting only non-protected mass suffices to restore feasibility after any subsequent \textit{Execute} step. Under LPF, non-protected mass (which is weakly less progressed at time~$t$ and remains so under stage advancement) is evicted before protected mass, and the continuous within-stage convention removes non-protected mass first whenever the eviction front reaches the cutoff stage. Hence none of the $m+\delta$ enlarged protected mass is ever evicted, and all of it completes within the next $l_1$ iterations. This proves~\eqref{eq:l1_window_strict}.

If \(\mathcal{C}\) is not the maximal-eviction cycle, it never visits a state
with \(N_t=m\). Indeed, if \(N_t=m\) for some~\(t\), then by the equality case
of~\eqref{eq:Nn_lb_m} we must have \(X_t=(0,\ldots,0,m)\). Starting from this
state, the execute-evict-admit map deterministically generates the
maximal-eviction cycle, so \(\mathcal{C}\) itself would be maximal, a
contradiction. Hence \(N_t>m\) for all~\(t\) on any non-maximal limit cycle.
By~\eqref{eq:l1_window_lb}, there exists \(t\) with \(y_t>0\); applying
\eqref{eq:l1_window_strict} at such~\(t\) gives one strict window sum. Summing
window sums over \(t=1,\ldots,p\) then yields
\(l_1\sum_{n=1}^{p}y_n>pm\), i.e.,
\(\bar{T}(\mathcal{C})>m/l_1\).

Finally, the maximal-eviction cycle achieves $\bar{T}=m/l_1$, so equality holds if and only if $\mathcal{C}$ is the maximal-eviction cycle.
\end{proof}

\subsection{Proof of Lemma~\EClemmaheadingref{lm:A1}{A.1}}

\begin{proof}[Proof of Lemma~\EClemmaheadingref{lm:A1}{A.1}]
Under eviction-free operation, the memory boundary is preserved and the stage-0
admission satisfies
\begin{equation}\label{eq:balance}
(l_0+1)x^{n+1}_{0}
=
(l_0+l_1)x^{n}_{l_1-1}-\sum_{j=0}^{l_1-2}x^{n}_{j}.
\end{equation}
Rewriting gives
\begin{equation}\label{eq:balance_appendix_re}
x^{n+1}_{0}
=
x^{n}_{l_1-1}
+
\frac{l_1x^{n}_{l_1-1}-\sum_{j=0}^{l_1-1}x^{n}_{j}}{l_0+1}.
\end{equation}

For \(j\ge 1\), the shift relation \(x^{n+1}_{j}=x^{n}_{j-1}\) preserves all
non-final stage values. Therefore the only new coordinate is
\(x^{n+1}_{0}\).

If the maximum is attained at some stage \(j\le l_1-2\), then the same value
reappears at stage \(j+1\), so \(\bar x^{n+1}\ge \bar x^n\). If the maximum is
attained at the final stage, then \(x^{n}_{l_1-1}\ge \sum_j x^{n}_{j}/l_1\),
and \eqref{eq:balance_appendix_re} gives
\[
x^{n+1}_{0}\ge x^{n}_{l_1-1}=\bar x^n.
\]
Hence \(\bar x^{n+1}\ge \bar x^n\).

The minimum is similar. If it is attained at some stage \(j\le l_1-2\), it
shifts unchanged to stage \(j+1\), so \(\underline x^{n+1}\le \underline x^n\).
If it is attained at the final stage, then \(x^{n}_{l_1-1}\le \sum_j x^{n}_{j}/l_1\),
and \eqref{eq:balance_appendix_re} gives
\[
x^{n+1}_{0}\le x^{n}_{l_1-1}=\underline x^n.
\]
Hence \(\underline x^{n+1}\le \underline x^n\), and so \(G^{n+1}\ge G^n\).
\end{proof}

\subsection{Proof of Lemma~\EClemmaheadingref{lem:A2}{A.2}}

\begin{proof}[Proof of Lemma~\EClemmaheadingref{lem:A2}{A.2}]
Suppose the system stays eviction-free. The argument has three parts. First, the
running maximum must converge if no eviction ever occurs. Second, whenever a
near-maximum cohort reaches the final stage, the balance equation creates a
new admission that jumps a fixed positive amount above the current maximum.
Third, convergence forces such near-maximum final-stage events to occur along
late renewals, which contradicts the fixed jump.

By Lemma~\ref{lm:A1}, the running
maximum \(\bar x^n\) is non-decreasing. The memory constraint implies
\[
\bar x^n(l_0+1)\le \sum_{j=0}^{l_1-1}(l_0+1+j)x^{n}_{j}=M.
\]
Thus \(\bar x^n\) is bounded and converges; in particular
\(\bar x^{n+1}-\bar x^n\to0\).

We next quantify the jump that occurs when the final stage is close to the
running maximum. Fix a time \(n\), write
\(a_n:=\bar x^n-x^n_{l_1-1}\ge0\), and set
\[
S_n:=\sum_{j=0}^{l_1-1}(\bar x^n-x^n_j).
\]
Since \(G^n\ge G^0\), we have \(S_n\ge G^0\). Rewriting
\eqref{eq:balance_appendix_re} relative to \(\bar x^n\) gives
\begin{equation}\label{eq:first_support_renewal_increment}
x^{n+1}_0-\bar x^n
=
\frac{S_n-(l_0+l_1+1)a_n}{l_0+1}.
\end{equation}
Consequently, whenever the final-stage coordinate is within
\(G^0/[2(l_0+l_1+1)]\) of the running maximum, the next admission exceeds the
current maximum by at least \(G^0/[2(l_0+1)]\).

It remains to connect this local jump estimate to the long-run trajectory. A
cohort that attains the running maximum has two possibilities before it reaches
the final stage. If no larger admission appears first, then the cohort reaches
stage \(l_1-1\) still attaining the running maximum, and the preceding
paragraph gives the fixed positive jump. If a larger admission appears first,
then the running maximum has already increased, and we restart the same
argument from this newer maximizing cohort. Thus the only way to avoid the
fixed jump would be to keep renewing the running maximum before the tracked
cohort reaches the final stage.

That last possibility is also incompatible with convergence. Since
\(\bar x^n\) is monotone and convergent, for every fixed window length \(q\),
\[
\bar x^{n+q}-\bar x^n\to0 .
\]
Apply this with \(q=l_1-1\) along sufficiently late renewal times. During the
next \(l_1-1\) iterations, the tracked cohort advances from its admission stage
to the final stage, while the running maximum can increase only by
\(o(1)\). Therefore, for all sufficiently late renewals, that cohort reaches the
final stage within \(G^0/[2(l_0+l_1+1)]\) of the then-current running maximum.
The jump estimate \eqref{eq:first_support_renewal_increment} then forces a
new increase of at least \(G^0/[2(l_0+1)]\), contradicting
\(\bar x^{n+1}-\bar x^n\to0\).

Hence the eviction-free regime must break down in finite time. The resulting
LPF eviction creates the first eviction-induced support loss in the
relative-position representation, moving the trajectory from full support to
the next lower live-support level. Since the fixed point is the unique
eviction-free boundary state with \(G^0=0\), it is unstable.
\end{proof}

\subsection{Proof of Lemma~\EClemmaheadingref{lem:A3}{A.3}}

\begin{proof}[Proof of Lemma~\EClemmaheadingref{lem:A3}{A.3}]
Zero mass shifts as zero, so \(P_r\) remains empty until it next reaches the
final stage. Starting from
the boundary \(M\), every other live position increases its footprint by one
token during execute, while \(P_r\)'s completion releases zero memory. The
post-execution memory is therefore
\[
\tilde M^n=M+\sum_{q\neq r} z^{n}_{q}>M.
\]
In the continuous formulation of Section~\ref{sec:model}, LPF eviction may trim
a fractional amount and restores the total to exactly~\(M\), leaving no slack
for admission: the new stage-0 batch is zero. That batch is exactly the mass
that would have been assigned to \(P_r\) after wrap-around, so
\(z^{n+1}_{r}=0\). Repeating the argument each time \(P_r\) returns to
stage~0 shows that \(P_r\) never revives.
\end{proof}

\subsection{Proof of Lemma~\EClemmaheadingref{lem:A6}{A.6}}

\begin{proof}[Proof of Lemma~\EClemmaheadingref{lem:A6}{A.6}]
Let
\[
\underline u:=\min_h u_h,
\qquad
\overline u:=\max_h u_h.
\]
The shifted coordinates \(u_0,\dots,u_{m-2}\) reappear unchanged in \(Y'\), so
only the loss of the oldest value \(u_{m-1}\) can reduce the spread.

If \(u_{m-1}\) is neither the unique maximum nor the unique minimum, then both
\(\underline u\) and \(\overline u\) remain among the shifted coordinates, so the
spread cannot decrease.

If \(u_{m-1}=\overline u\), then every term \(u_{m-1}-u_h\) in
\eqref{eq:block_map_u} is nonnegative, hence \(u'_0\ge \overline u\). The old
minimum remains among the shifted coordinates, so
\[
G^{\mathrm{sc}}(Y')\ge \overline u-\underline u=G^{\mathrm{sc}}(Y).
\]
If \(G^{\mathrm{sc}}(Y)>0\), then some \(u_h<\overline u\), so the numerator in
\eqref{eq:block_map_u} is strictly positive and \(u'_0>\overline u\); thus the
spread strictly grows.

The case \(u_{m-1}=\underline u\) is symmetric: every term in the numerator of
\eqref{eq:block_map_u} is nonpositive, so \(u'_0\le \underline u\), while the
old maximum remains among the shifted coordinates. Therefore the spread again
cannot decrease, and it is strict when \(G^{\mathrm{sc}}(Y)>0\).
\end{proof}

\subsection{Proof of Lemma~\EClemmaheadingref{lem:A8}{A.8}}

\begin{proof}[Proof of Lemma~\EClemmaheadingref{lem:A8}{A.8}]
Suppose, toward contradiction, that the support pattern remains fixed forever.
The proof mirrors Lemma~\ref{lem:A2}, but now on the block-end chain and in the
normalized coordinates \(u_h=y_h/g_h\). We first obtain a convergent running
maximum, then show that an oldest coordinate close to that maximum creates a
uniform positive jump, and finally use convergence to force such close-oldest
events along late renewals.

Index the sampled states by \(r\), and write
\[
U^r:=\max_h u^r_h,\qquad G^r:=G^{\mathrm{sc}}(Y_r).
\]
The proof of Lemma~\ref{lem:A6} shows that \(U^r\) is non-decreasing. Since
\(u^r_h\in[0,M/(l_0+1)]\), \(U^r\) is bounded and therefore converges; hence
\(U^{r+1}-U^r\to0\). Also \(G^r\ge G^0>0\) by Lemma~\ref{lem:A6}.

We now isolate the one-step estimate that drives the contradiction.
At a sampled step \(r\), let the oldest normalized coordinate be
\(u^r_{m-1}\), and set \(a_r:=U^r-u^r_{m-1}\). Define
\[
S_r:=\sum_{h=0}^{m-2}g_h\,(U^r-u^r_h).
\]
If \(a_r\le G^0/[2(l_0+l_1)]\), then the oldest coordinate is not the
minimum, so a minimum-attaining coordinate survives among
\(h=0,\ldots,m-2\), and \(S_r\ge G^r\ge G^0\). Rewriting
\eqref{eq:block_map_u} relative to \(U^r\) gives
\begin{equation}\label{eq:sampled_renewal_increment}
u^{r+1}_0-U^r
=
\frac{S_r-(l_0+l_1)a_r}{l_0+g_{m-1}} .
\end{equation}
Thus, whenever the oldest coordinate is within
\(G^0/[2(l_0+l_1)]\) of the running maximum, the new coordinate exceeds the
current maximum by at least \(G^0/[2(l_0+l_1)]\).

It remains to show that the hypothesis of this estimate must hold arbitrarily
late. Follow any coordinate that attains the running maximum. If it reaches the
oldest slot before a larger newly created coordinate appears, then the estimate
above gives the fixed positive jump. If a larger coordinate appears first, then
\(U^r\) has already increased and we restart the tracking argument from this
new coordinate. Hence avoiding the fixed jump would require infinitely many
such renewals.

But infinite late renewals are also impossible under convergence. Since
\(U^r\) is monotone and convergent, for the fixed window length \(m-1\),
\[
U^{r+m-1}-U^r\to0 .
\]
After a renewal at sampled time \(r\), the renewing coordinate reaches the
oldest slot within at most \(m-1\) sampled steps unless an even larger renewal
occurs first, in which case we restart from that later and larger coordinate.
Along sufficiently late renewals, the total increase of \(U^r\) during this
window is less than \(G^0/[2(l_0+l_1)]\). Therefore the tracked coordinate
reaches the oldest slot within \(G^0/[2(l_0+l_1)]\) of the then-current running
maximum, and \eqref{eq:sampled_renewal_increment} again forces a fixed positive
jump. This contradicts \(U^{r+1}-U^r\to0\).

Therefore the fixed-support affine continuation cannot remain in the positive
orthant forever. The final step is to translate this affine contradiction back
to the physical LPF dynamics. On a fixed support pattern, Proposition~\ref{prop:A4}
gives the affine block-end update only on the branch where every live
coordinate stays positive and LPF trims only the prescribed least-progressed
mass during the intervening eviction block. Leaving the positive orthant is
therefore exactly reaching a branch boundary: some live mass is exhausted by
LPF. At that boundary \(y_h=g_hu_h=0\) for at least one live position, and by
Lemma~\ref{lem:A3} the emptied relative position cannot revive. Thus the next
sampled state lies on a strictly smaller support pattern after finitely many
sampled steps.
\end{proof}

\subsection{Proof of Lemma~\EClemmaheadingref{lem:A9}{A.9}}

\begin{proof}[Proof of Lemma~\EClemmaheadingref{lem:A9}{A.9}]
If \(|I|=1\), the trajectory is already on the single-live-position stratum and
cannot be captured by a nonmaximal orbit, so the claim is empty. Assume
\(|I|\ge2\). We prove first that, for each fixed finite horizon, the set of
initial states captured by a nonmaximal balanced orbit by that horizon has
relative Lebesgue measure zero in \(\mathcal B_I\). The desired claim then
follows by taking the countable union over horizons. The finite-horizon proof
has four steps. First, finite-time capture can be checked at a block-end
representative of the eventual balanced orbit. Second, the LPF dynamics on a
fixed finite horizon split the initial face into finitely many affine branch
cells, after discarding lower-dimensional branch boundaries. Third, on each
nondegenerate branch cell, the corresponding branch map is built from affine
isomorphisms and affine submersions, so preimages of relative-measure-zero
target sets remain relative-measure-zero. Fourth, the nonmaximal balanced
targets form only finitely many points in each positive-dimensional target
face.

\emph{Step 1: finite-time capture and boundary faces.}
Here ``captured'' means that the trajectory equals a state on the
balanced orbit from some finite time onward. For finite-time capture at an
arbitrary physical phase, passing to the next block-end gives the corresponding
balanced sampled representative. By Lemma~\ref{lem:A3} the live-position set can only
shrink, and on any eventual live set of size \(m\ge2\), the block-end chain is
governed by Proposition~\ref{prop:A4}: either \(G^{\mathrm{sc}}=0\), in which case the
sampled state is the balanced representative and determines the full periodic
orbit, or \(G^{\mathrm{sc}}>0\), in which case Lemma~\ref{lem:A8} forces another support
loss in finite time. In the rest of the proof, ``measure zero'' always means
zero relative Lebesgue measure on the affine hull of the current
memory-boundary face.

\emph{Step 2: finite-horizon branch cells.}
Fix a finite horizon \(T\). Up to time \(T\), only finitely many LPF branches
are possible: at each step, a branch records the current support pattern and,
whenever LPF evicts, the ordered set of stages exhausted by eviction together
with the unique stage that is partially trimmed. On any fixed branch, all LPF
ordering inequalities and positivity inequalities are fixed, so the
execute--evict--admit update is affine on the relative interior of that branch
cell.

We first remove the branch boundaries. These are the states where one of the
strict branch inequalities becomes an equality: an eviction depth exactly
exhausts a stage, the eviction stops with no positive partially trimmed
survivor, or the partial trim has zero size. For a fixed step and a fixed
branch history, each such equality is affine in the current branch coordinates,
so it defines a proper affine face unless it is empty. There are only finitely
many such faces up to horizon \(T\). Their preimages in the initial face also
have measure zero by the pullback argument of Step~3 below. Thus, after
discarding a finite union of relative-measure-zero sets, it suffices to work on
the relative interiors of nondegenerate branch cells.

On such a cell, we group fixed-support physical steps into the exact block-end
map from Proposition~\ref{prop:A4} and treat support-loss steps separately;
equivalently, the finite-horizon map is a finite composition of elementary
execute--evict--admit affine pieces on the same LPF branch. The only maps that
can appear in this composition are the following two types: affine
isomorphisms on a fixed boundary stratum, and affine submersions onto a
lower-dimensional boundary stratum.

\emph{Step 3: pullback of measure-zero sets through branch maps.}
We now verify these two map types and their null-set pullback property.
\begin{itemize}[nosep]
\item In the eviction-free regime, the boundary map is affine and invertible.
It shifts \(x_j\) to \(x_{j+1}\) for \(j<l_1-1\), and the new admission is
\[
x'_0=\frac{(l_0+l_1)x_{l_1-1}-\sum_{j=0}^{l_1-2}x_j}{l_0+1}.
\]
The associated linear map has determinant
\(\pm(l_0+l_1)/(l_0+1)\neq0\). It preserves the weighted memory hyperplane
\(\sum_j w_jx_j=M\), so its restriction to the boundary affine hull is an
affine isomorphism.
\item On a fixed support pattern, the block-end map in normalized coordinates
\(u_h=y_h/g_h\) is affine and invertible. It shifts old normalized masses and
has
\[
\frac{\partial u'_0}{\partial u_{m-1}}
=\frac{l_0+l_1}{l_0+g_{m-1}}\neq0,
\]
so its full Jacobian is nonsingular. This map sends the old boundary hyperplane
to the rotated new boundary hyperplane, hence restricts to an affine
isomorphism between the corresponding block-end strata.
\item At a support-loss branch, the map to the lower-dimensional stratum is an
affine submersion. Let \(\tilde x\) be the post-execute, pre-eviction vector on
a fixed nondegenerate LPF branch. Let \(E\) be the live positions fully
exhausted by LPF, let \(q\) be the unique position partially trimmed with
trim amount \(0<\tau<\tilde x_q\), and let \(J\) be the surviving live positions
after the step. For a genuine support-loss branch \(E\neq\emptyset\); if
\(E=\emptyset\), the branch remains on the same support pattern. Thus
\(q\in J\). Set \(H:=J\setminus\{q\}\). Branches with no positive partially
trimmed survivor, with zero trim, or with \(\tau=\tilde x_q\), are exactly the
measure-zero branch-boundary cases already discarded. On the target boundary face,
use \(z=(x'_h)_{h\in H}\) as local affine coordinates; the remaining target
coordinate is determined by
\[
x'_q=\frac{M-\sum_{h\in H} w_h z_h}{w_q}.
\]
On the relative interior of the source branch, the pre-eviction masses
\((\tilde x_h)_{h\in H}\) are distinct affine functions of the source state.
By the LPF order, the trim does not reach any position in \(H\) before it stops
partway through position \(q\), so the positions in \(H\) are not trimmed during
this eviction step.
Complete them to a source affine chart \((z,\eta)\) as follows. Use the affine
image of the source memory-boundary equation to solve for the non-target
coordinate \(\tilde x_q\), and let \(\eta\) collect the remaining non-target
source coordinates, including fully exhausted masses and, if present, the mass
that completes during the execute step. This chart is valid on a sufficiently
small relative neighborhood inside the branch cell because all LPF ordering and
positivity inequalities are strict there. The partial-trim amount follows
affinely from the prescribed overflow, the completed mass, and the exhausted
masses. In these coordinates, the target free coordinates are exactly the
surviving copied-or-shifted masses \(x'_h=\tilde x_h=z_h\), \(h\in H\). Hence,
after affine changes of coordinates, the branch has the normal form
\[
(z,\eta)\longmapsto z,
\]
where \(z\in\mathbb R^{|H|}=\mathbb R^{|J|-1}\) are target coordinates and
\(\eta\) collects the discarded fiber coordinates. Its Jacobian contains an
identity block of size \(|J|-1\), so it has full row rank and the branch is a
submersion.
\end{itemize}
It remains to justify why these maps preserve relative nullity under preimage.
Affine isomorphisms are bi-Lipschitz on the affine hulls of the relevant
strata, so they preserve and reflect measure-zero sets. For an affine
submersion, affine changes of source and target coordinates put the map in the
normal form \((z,\eta)\mapsto z\) on each bounded polyhedral branch cell. Hence
for a relative-measure-zero target set \(N\), its preimage has measure
\[
\lambda\{(z,\eta):z\in N\}
=\int_N \lambda(\text{fiber over }z)\,dz=0
\]
by Fubini, because the memory boundary makes the branch cell bounded and each
fiber has finite measure. Therefore the preimage of a relative-measure-zero set
under any finite composition of these branch maps also has relative measure
zero.

\emph{Step 4: finite balanced targets and countable union over horizons.}
For any fixed relative support pattern and physical phase with \(m\ge2\) live
positions, the balanced orbit contributes a single point of the corresponding
stratum. Indeed, if the live masses are \(y_0,\ldots,y_{m-1}\) with predecessor
gaps \(g_0,\ldots,g_{m-1}\), balance requires
\[
\frac{y_0}{g_0}=\frac{y_1}{g_1}=\cdots=\frac{y_{m-1}}{g_{m-1}},
\]
so \(y_h=cg_h\), and the memory boundary fixes \(c\). There are finitely many
relative support patterns and at most \(l_1\) physical phases. Therefore the
target nonmaximal balanced states form a finite set in each stratum, hence a
measure-zero set because the target stratum has dimension \(m-1\ge1\).

Fix the horizon \(T\). By Steps~2 and~3, on every nondegenerate branch cell,
the preimage of this finite balanced target set has relative measure zero in
the initial face. The discarded branch boundaries also have relative measure
zero, and there are only finitely many branch sequences up to \(T\). Hence the
set captured by a nonmaximal balanced orbit by time \(T\) has relative measure
zero in \(\mathcal B_I\). Finally, finite-time capture means capture by some
finite horizon, so the full exact-capture set is the countable union over
\(T=0,1,2,\ldots\) of these finite-horizon null sets. Countable subadditivity
preserves measure zero and proves the claim.
\end{proof}

\subsection{Proof of Lemma~\EClemmaheadingref{lem:max_cycle_local_absorption}{A.10}}

\begin{proof}[Proof of Lemma~\EClemmaheadingref{lem:max_cycle_local_absorption}{A.10}]
Choose
\[
\epsilon_0<\frac12\min_{r\ne s}\|C_r-C_s\|
\]
and replace the given \(\epsilon\) by \(\min\{\epsilon,\epsilon_0\}\). Then a
state within this distance of \(\mathcal C_{\max}\) has a unique closest phase.
Fix such a phase \(r\). Write the coordinate at stage \(r\) as the principal
coordinate and all other coordinates as residual coordinates. By the
memory-boundary equation, the principal coordinate differs from \(M/w_r\) by a
linear combination of the residual coordinates. Thus, by taking
\(\delta\) small enough, the total residual memory can be made arbitrarily
small and the principal mass remains bounded away from zero, uniformly over
the finitely many phases.
More concretely, choose this total residual memory smaller than a fixed
fraction of \(\min_r M/w_r\); then, even if all higher-stage residual cohorts
complete during an execute step, the nonfinal phase still has strictly positive
overflow and therefore no new admission can occur before LPF eviction restores
the boundary.

We claim that for sufficiently small phase neighborhoods the trajectory cannot
leave the prescribed \(\epsilon\)-tube before it reaches the single-live-position
stratum. If \(r<l_1-1\), the principal cohort advances to stage \(r+1\). At
the cycle point \(C_r\), this advancement creates overflow
\[
(w_{r+1}-w_r)\frac{M}{w_r}=\frac{M}{w_r}>0.
\]
For sufficiently small residual mass, the same strict overflow persists after
accounting for any residual cohorts that complete during the execute step.
Continuous LPF then restores the memory boundary exactly. It removes any
lower-progress residual cohorts before it can trim the advanced principal
cohort, and if residual cohorts at higher stages remain, their total memory is
still of the same small order as the initial residual memory. The advanced
principal coordinate is then determined by the boundary equation and therefore
differs from \(M/w_{r+1}\) by only the memory contribution of those higher-stage
residuals. If \(r=l_1-1\), the principal cohort completes and the admission
step creates the next principal cohort at stage \(0\); any residual cohorts
again have only the small memory inherited from the initial residuals, and the
stage-0 principal mass is fixed by the boundary equation up to that same small
residual memory.

No new residual cohort is created in this local tube: at nonfinal phases the
strict overflow makes the continuous LPF step fill the boundary by eviction,
so there is no additional admission, while at the final phase the new admitted
cohort is precisely the next principal cohort. Existing residual cohorts either
are evicted by LPF when they lag behind the principal cohort or complete after
at most \(l_1\) physical iterations when they lie ahead. Hence, for small enough
\(\delta\), all iterates before absorption remain within the prescribed
\(\epsilon\)-tube around the corresponding cycle phase.

Now take \(X^0\notin\mathcal E\) in this \(\delta\)-tube. If the trajectory is
not already on the single-live-position stratum, then it has at least two live
positions and cannot stay on, or be captured by, any nonmaximal balanced orbit:
the eviction-free fixed point and all intermediate balanced limit cycles are
included in \(\mathcal E\). Theorem~\ref{thm:A1_position} therefore forces
support losses until only one live position remains. Once a single live
position remains, the memory-boundary equation fixes its mass uniquely at the
current phase, namely \(M/w_r\), so the state is exactly \(C_r\). From then on
the trajectory follows \(\mathcal C_{\max}\). Thus the trajectory reaches
\(\mathcal C_{\max}\) in finite time, remains within \(\epsilon\) of the cycle
at all times, and then has distance zero from the cycle.
\end{proof}

\subsection{Proof of Lemma~\EClemmaheadingref{lem:proportional_cohort_entry}{2}}

\begin{proof}[Proof of Lemma~\EClemmaheadingref{lem:proportional_cohort_entry}{2}]
Each cohort present at time zero has age at least zero and therefore either
completes or is evicted within at most \(d\) iterations. Hence, after \(d\)
iterations, every live cohort was admitted after time zero. Such a newly
admitted cohort enters with class masses \(p_k a^n\), so it has the proportional
composition stated above at age zero. One execution step only advances the
cohort age and removes classes whose decoding length has ended. If LPF evicts
part of a mixed stage, the proportional within-stage rule multiplies every live
class mass in that stage by the same factor, so the ratios \(p_k\) among the
surviving classes are preserved. Therefore all live cohorts after time \(d\)
have the displayed form, and the same argument shows that the set of such
states is forward invariant.
\end{proof}

\subsection{Proof of Lemma~\EClemmaheadingref{lem:bounded}{3}}

\begin{proof}[Proof of Lemma~\EClemmaheadingref{lem:bounded}{3}]
From~\eqref{eq:closed_form}, any root $z \neq 1$ satisfies $z^{l_{1,2}} = pz^{l_{1,2}-l_{1,1}} + q$. Taking absolute values and applying the triangle inequality:
\[
|z|^{l_{1,2}} = |pz^{l_{1,2}-l_{1,1}} + q| \leq p|z|^{l_{1,2}-l_{1,1}} + q.
\]
Define $g(x) = x^{l_{1,2}} - px^{l_{1,2}-l_{1,1}} - q$ for $x \geq 0$. First, $g(1) = 1 - p - q = 0$. Second, for $x > 1$:
\[
g'(x) = l_{1,2} x^{l_{1,2}-1} - p(l_{1,2}-l_{1,1}) x^{l_{1,2}-l_{1,1}-1} = x^{l_{1,2}-l_{1,1}-1}[l_{1,2} x^{l_{1,1}} - p(l_{1,2}-l_{1,1})].
\]
Since $x > 1$ implies $x^{l_{1,1}} > 1$, the bracket exceeds $l_{1,2} - p(l_{1,2}-l_{1,1}) = ql_{1,2} + pl_{1,1} > 0$, so $g'(x) > 0$ for $x > 1$. By the mean value theorem, $g(x) > g(1) = 0$ for all $x > 1$.

Now suppose $|z| > 1$. Then $g(|z|) > 0$, i.e., $|z|^{l_{1,2}} > p|z|^{l_{1,2}-l_{1,1}} + q$. This contradicts the bound $|z|^{l_{1,2}} \leq p|z|^{l_{1,2}-l_{1,1}} + q$. Hence $|z| \leq 1$.
\end{proof}

\subsection{Proof of Lemma~\EClemmaheadingref{lem:unit_circle}{4}}

\begin{proof}[Proof of Lemma~\EClemmaheadingref{lem:unit_circle}{4}]
For $|z| = 1$, we have $|z|^{l_{1,2}} = 1$ and $|pz^{l_{1,2}-l_{1,1}} + q| \leq p + q = 1$ (by triangle inequality). Since $(1-z)A(z) = 0$ requires equality in~\eqref{eq:closed_form}, we need $|z^{l_{1,2}}| = |pz^{l_{1,2}-l_{1,1}} + q|$, i.e., $1 = |pz^{l_{1,2}-l_{1,1}} + q|$.

Since $|pz^{l_{1,2}-l_{1,1}}|=p>0$ and $|q|=q>0$, triangle-inequality equality requires $z^{l_{1,2}-l_{1,1}}$ to be a positive real of modulus one, hence $z^{l_{1,2}-l_{1,1}}=1$. Then $z^{l_{1,2}}=p+q=1$, so
\begin{align*}
z^{l_{1,2}} &= 1, \\
z^{l_{1,2} - l_{1,1}} &= 1.
\end{align*}
By the definition of GCD, $g = \gcd(l_{1,1}, l_{1,2}) = \gcd(l_{1,2}, l_{1,2} - l_{1,1})$. Since $z^{l_{1,2}} = 1$ and $z^{l_{1,2}-l_{1,1}} = 1$, we have $z^{\gcd(l_{1,2}, l_{1,2}-l_{1,1})} = 1$, i.e., $z^g = 1$ (by B\'ezout's identity: there exist integers $a, b$ such that $g = a \cdot l_{1,2} + b \cdot (l_{1,2} - l_{1,1})$).
\end{proof}

\subsection{Proof of Lemma~\EClemmaheadingref{lem:unity}{5}}

\begin{proof}[Proof of Lemma~\EClemmaheadingref{lem:unity}{5}]
Direct substitution into the definition of $A(z)$:
\begin{align*}
A(1) &= 1^{l_{1,2}-1} + \sum_{m=1}^{l_{1,1}-1} 1^{l_{1,2}-1-m} + q\sum_{m=l_{1,1}}^{l_{1,2}-1} 1^{l_{1,2}-1-m} \\
&= 1 + (l_{1,1} - 1) \cdot 1 + q \cdot (l_{1,2} - l_{1,1}) \cdot 1 \\
&= l_{1,1} + q(l_{1,2} - l_{1,1}) \\
&= l_{1,1} + (1-p)(l_{1,2} - l_{1,1}) \quad \text{(since $q = 1-p$)} \\
&= l_{1,1} + l_{1,2} - l_{1,1} - p(l_{1,2} - l_{1,1}) \\
&= l_{1,2} - pl_{1,2} + pl_{1,1} \\
&= (1-p)l_{1,2} + pl_{1,1} \\
&= ql_{1,2} + pl_{1,1} > 0.
\end{align*}
Since $p, q > 0$ and $l_{1,1}, l_{1,2} \geq 1$, we have $A(1) > 0$, so $z = 1$ is not a root of $A(z)$.
\end{proof}

\subsection{Proof of Lemma~\EClemmaheadingref{lem:primitive}{6}}

\begin{proof}[Proof of Lemma~\EClemmaheadingref{lem:primitive}{6}]
Write $l_{1,1} = ag$ and $l_{1,2} = bg$ with $\gcd(a,b) = 1$. For any $\omega$ with $\omega^g = 1$ and $\omega \neq 1$: $\omega^{l_{1,1}} = (\omega^g)^a = 1$, $\omega^{l_{1,2}} = (\omega^g)^b = 1$, and $\omega^{l_{1,2}-l_{1,1}} = (\omega^g)^{b-a} = 1$. Substituting into~\eqref{eq:closed_form}: $(1-\omega)A(\omega) = -1 + p + q = 0$. Since $\omega \neq 1$, we have $1 - \omega \neq 0$, and therefore $A(\omega) = 0$.
\end{proof}

\subsection{Proof of Lemma~\EClemmaheadingref{lem:noncoprime_complementary_roots}{7}}

\begin{proof}[Proof of Lemma~\EClemmaheadingref{lem:noncoprime_complementary_roots}{7}]
Write
\[
G_\epsilon(z):=A(z)-\epsilon zA'(z),
\qquad
\epsilon:=\frac{1}{l_0+l_{1,2}} .
\]
By the decomposition~\eqref{eq:decomposition}, the roots of \(F\) are exactly
the roots of \(G_\epsilon\). The degree of \(G_\epsilon\) is \(l_{1,2}-1\) for
all sufficiently small \(\epsilon\), since the leading coefficient is
\(1-\epsilon(l_{1,2}-1)\).

First consider the roots of the limiting polynomial \(A\). By
Lemmas~\ref{lem:bounded}--\ref{lem:primitive}, every root of \(A\) lies in the
closed unit disk, the only roots on the unit circle are the non-trivial
\(g\)-th roots of unity, and \(z=1\) is not a root. Moreover, the computation
in the proof of Theorem~\ref{thm:instability} gives \(A'(\omega)\ne0\) at each
non-trivial \(g\)-th root \(\omega\). Hence these \(g-1\) unit-circle roots are
simple, and the remaining \(l_{1,2}-g\) roots of \(A\), counted with algebraic
multiplicity, lie strictly inside the unit disk.

Let \(r<1\) be the largest modulus of these remaining roots, and choose
\(\rho\) with \(r<\rho<1\). Then \(A\) has no zeros on \(|z|=\rho\), so
\[
m_\rho:=\min_{|z|=\rho}|A(z)|>0,
\qquad
M_\rho:=\max_{|z|=\rho}|zA'(z)|<\infty .
\]
For sufficiently small \(\epsilon\), \(\epsilon M_\rho<m_\rho\). Rouch\'e's
theorem on \(|z|=\rho\) implies that \(G_\epsilon\) and \(A\) have the same
number of roots in \(|z|<\rho\), namely \(l_{1,2}-g\). These roots are stable
because \(\rho<1\).

It remains to locate the roots associated with the unit-circle roots. Because
there are finitely many non-trivial \(g\)-th roots, choose pairwise disjoint
neighborhoods \(U_\omega\) around them. Shrinking a common
\(\epsilon_0>0\) if necessary, the implicit-function argument in
Theorem~\ref{thm:instability} gives, for every \(0<\epsilon<\epsilon_0\) and
each non-trivial \(g\)-th root \(\omega\), a unique zero
\(z_\omega(\epsilon)\in U_\omega\) of \(G_\epsilon\). Moreover
\(\partial_zG_\epsilon(z_\omega(\epsilon))\ne0\), so this zero is simple.
The neighborhoods are disjoint, hence these roots are distinct. They satisfy
\[
z_\omega(\epsilon)=\omega(1+\epsilon)+O(\epsilon^2),
\qquad
|z_\omega(\epsilon)|-1=\Theta(\epsilon)=\Theta(1/l_0)>0 .
\]
Thus \(G_\epsilon\) has \(g-1\) unstable roots. Together with the
\(l_{1,2}-g\) roots inside \(|z|<\rho\), this accounts for all
\(l_{1,2}-1\) roots of \(G_\epsilon\). Therefore no other finite-\(l_0\) roots
exist, and the asserted root count follows.
\end{proof}

\subsection{Proof of Lemma~\EClemmaheadingref{lem:generic_finite_exit}{8}}

\begin{proof}[Proof of Lemma~\EClemmaheadingref{lem:generic_finite_exit}{8}]
Use the normalized two-class cohort coordinates
\[
Y_s=\begin{cases}
X_s, & s<l_{1,1},\\
X_s/q, & s\ge l_{1,1},
\end{cases}
\qquad
V_s=\begin{cases}
l_0+s+1, & s<l_{1,1},\\
q(l_0+s+1), & s\ge l_{1,1}.
\end{cases}
\]
The eviction-free equilibrium is \(Y^*=x^*\mathbf{1}\). Let \(E^S(Y)\) denote
the memory immediately after the execution/completion step and before
admission. At equilibrium, the next admission has mass \(x^*>0\), so
\[
M-E^S(Y^*)=V_0x^*>0 .
\]
By continuity, after shrinking a neighborhood of \(Y^*\) if needed, every
\(Y\in\mathcal{N}\) satisfies \(Y_s\ge0\) and
\[
|E^S(Y)-E^S(Y^*)|<\frac12 V_0x^* .
\]
Therefore \(M-E^S(Y)>\frac12 V_0x^*>0\), the next admission amount
\((M-E^S(Y))/V_0\) is positive, and feasibility is restored by admission alone.
No LPF eviction is triggered anywhere in \(\mathcal{N}\).

On this local no-eviction region, subtracting the equilibrium gives the exact
linear stage-state dynamics
\[
y^{n+1}=\Phi y^n,\qquad y^n:=Y^n-Y^* .
\]
The matrix \(\Phi\) is the companion/shift matrix for the no-eviction dynamics:
it has one zero eigenvalue, and its nonzero eigenvalues are exactly the roots
of \(F(z)=0\). Lemma~\ref{lem:noncoprime_complementary_roots} shows that these
nonzero eigenvalues consist of \(g-1\) unstable roots and \(l_{1,2}-g\) stable
roots. Thus \(\Phi\) has a spectral
splitting \(E^u\oplus E^s\), with \(\dim E^u=g-1\), and the zero eigenvalue is
included in \(E^s\). Let \(P_u\) be the spectral projection onto \(E^u\).

If \(P_u y^0\ne0\), then along the no-eviction dynamics
\[
P_u y^n=\Phi_u^n P_u y^0,
\]
where \(\Phi_u=\Phi|_{E^u}\) has all eigenvalues outside the unit circle. Hence
\(\Phi_u^{-n}\to0\). If \(\{\Phi_u^nP_u y^0\}_{n\ge0}\) were bounded, then
\[
P_u y^0=\Phi_u^{-n}\Phi_u^nP_u y^0\to0,
\]
a contradiction. Therefore the unstable component is unbounded as long as the
linear no-eviction dynamics continue. Since \(\mathcal{N}\) is bounded after
shrinking it if necessary, the trajectory must leave \(\mathcal{N}\) in finite
time.

The exceptional initial perturbations in \(\mathcal{N}\) are contained in
\(\ker P_u\), a proper linear subspace of the local state space. They therefore
have Lebesgue measure zero in local coordinates.
\end{proof}

\subsection{Proof of Lemma~\EClemmaheadingref{lem:lyapunov_bounds}{10}}

\begin{proof}[Proof of Lemma~\EClemmaheadingref{lem:lyapunov_bounds}{10}]
Since all eigenvalues of $\Phi$ satisfy $|z_j| < 1$, the series $P = \sum_{k=0}^{\infty} (\Phi^k)^\top \Phi^k$ converges to the unique solution of $\Phi^\top P \Phi - P = -I$. The identity $P - \Phi^\top P \Phi = I$ gives $P \succeq I$. For the upper bound, note that the entries of $\Phi(l_0)$ are ratios $V_s/V_0$, so $\Phi(l_0)$ depends continuously on $1/(l_0+1)$ and the family $\{\Phi(l_0): l_0\ge l_{0,\star}\}$ has compact closure.

Vectorizing the Lyapunov equation gives:
\[
\big(I - \Phi^\top \otimes \Phi^\top\big)\,\mathrm{vec}(P) = \mathrm{vec}(I).
\]
Since $\rho(\Phi)\le \rho<1$ by the spectral bound fixed before Lemma~\ref{lem:lyapunov_bounds}, we have $\rho(\Phi^\top \otimes \Phi^\top)=\rho(\Phi)^2\le \rho^2<1$, so $I-\Phi^\top \otimes \Phi^\top$ is invertible for all $l_0\ge l_{0,\star}$. The inverse depends continuously on $\Phi$, hence $\|(I-\Phi^\top \otimes \Phi^\top)^{-1}\|_2$ is bounded on the compact closure. Consequently $\|P\|_2$ is uniformly bounded for $l_0\ge l_{0,\star}$, implying $P \preceq C_P I$ for some constant $C_P$ independent of $l_0$.
\end{proof}

\subsection{Proof of Lemma~\EClemmaheadingref{lem:physical}{11}}

\begin{proof}[Proof of Lemma~\EClemmaheadingref{lem:physical}{11}]
Let $N^n = \sum_{s=0}^{l_{1,2}-1} X^n_s$ be the total number of (actual) requests in the system at time $n$. Each request occupies at least $l_0+1$ tokens, so:
\[
N^n \le \frac{M}{l_0+1} = \frac{\beta l_0}{l_0+1} = O(1).
\]
In one iteration, each in-flight request generates one new token during decoding, so the post-shift token increase is at most $N^n = O(1)$. Completions can only reduce memory, and the admission controller never increases memory beyond $M$ (it admits at most the available slack). Therefore, the number of tokens that must be removed by eviction in any step is at most the $O(1)$ overflow created by the shift.

Since evicting one request frees at least $l_0+1$ tokens, the number of evicted requests per step is at most $O(1)/(l_0+1)=O(1/l_0)$.
\end{proof}

\subsection{Proof of Lemma~\EClemmaheadingref{lem:perturbation}{12}}

\begin{proof}[Proof of Lemma~\EClemmaheadingref{lem:perturbation}{12}]
The perturbation $e^n = y^n - \Phi y^{n-1}$ arises when eviction truncates the linear update. At stage~$0$, eviction replaces the ideal admission $\tilde{a}_n<0$ with zero, so $|e^n_0|=|\tilde{a}_n|=(E^S-M)/V_0=O(1)/(l_0+1)=O(1/l_0)$ by Lemma~\ref{lem:physical}. At stages $s\ge 1$, the eviction removes $O(1/l_0)$ requests in total (Lemma~\ref{lem:physical}), giving $\sum_{s\ge 1}|(e^n)_s|\le O(1/l_0)/\min(1,q)=O(1/l_0)$. Hence $\|e^n\|_2\le\|e^n\|_1=O(1/l_0)$.
\end{proof}

\subsection{Proof of Lemma~\EClemmaheadingref{lem:multiple_roots_separated}{13}}

\begin{proof}[Proof of Lemma~\EClemmaheadingref{lem:multiple_roots_separated}{13}]
If \(P_n(\zeta)=P_n'(\zeta)=0\) and \(\zeta\ne1\), then \(\zeta\ne0\) and
\[
n\zeta^{n-1}=p(n-a)\zeta^{n-a-1},
\qquad\text{so}\qquad
\zeta^a=p\left(1-\frac{a}{n}\right).
\]
Therefore
\[
|\zeta|\le p^{1/a}\le p^{1/n}
=\exp\!\left(\frac{\log p}{n}\right)
\le 1-\frac{-\log p}{2n}
\]
for all sufficiently large \(n\). This proves the claim with
\(c_p=-\log(p)/2>0\). Since the dominant threshold shell has radial width
\(O(n^{-3})\) around the unit circle, these multiple roots cannot determine the
first crossing.
\end{proof}

\subsection{Proof of Lemma~\EClemmaheadingref{lem:near_unit_chart}{14}}

\begin{proof}[Proof of Lemma~\EClemmaheadingref{lem:near_unit_chart}{14}]
First note that all roots of \(P_n\) lie in the closed unit disk. Indeed, if
\(R:=|z|>1\), then
\[
R^n=|pz^{n-a}+q|\le pR^{n-a}+q<pR^n+q,
\]
which implies \(R^n<1\), a contradiction. Hence the logarithmic coordinates
used below always satisfy \(\Re\mu\ge0\).

The representation is just the principal logarithm in the sector around a
nearest \(n\)-th root of unity: once \(\omega\) is chosen, set
\(\mu=-n\operatorname{Log}(z/\omega)\), with \(\Im\mu\in[-\pi,\pi]\). Since
\(\omega^n=1\) and \(\omega^{-a}=e^{\mathrm{i}x}\), substituting
\(z=\omega e^{-\mu/n}\) into \(P_n(z)=0\) gives the displayed
equation for \(H\).
If a root lies on a sector boundary and has two nearest \(n\)-th roots of
unity, fix either one by an arbitrary deterministic tie-breaking rule. The
estimates below are invariant under this choice. We use the chart only for
coverage and uniform bounds; a boundary root that admits two labels is counted
once as a root of \(P_n\).

The only point that needs care is uniformity. At \((\mu,x)=(0,0)\),
\[
H_\mu(0,0,\theta)=-(q+p\theta),
\]
which is bounded away from zero uniformly over \(\theta\in[0,1]\), since
\(q>0\). The implicit-function theorem therefore gives a unique branch
\(\mu(x,\theta)\) for \(|x|\le\delta\), with \(\delta\) chosen uniformly in
\(\theta\). We also need that this local branch accounts for all solutions in
the near-unit strip when \(|x|\le\delta\). This follows by compactness. At
\(x=0\) and \(\Re\mu=0\), writing
\[
u=e^{-\mu},\qquad v=e^{-(1-\theta)\mu},
\]
the equation \(H(\mu,0,\theta)=0\) gives \(u=pv+q\) with
\(|u|=|v|=1\). Equality in the triangle inequality forces \(u=v=1\), and
therefore \(\mu=0\) in the strip \(\Im\mu\in[-\pi,\pi]\). Consequently, after
choosing \(\eta\) and \(\delta\) small enough, there are no solutions of
\(H(\mu,x,\theta)=0\) with
\[
0\le\Re\mu\le\eta,\qquad |x|\le\delta,\qquad \Im\mu\in[-\pi,\pi],
\]
outside the fixed IFT neighborhood of \(\mu=0\). Thus the displayed branch is
the unique near-unit solution for small phases.

If the last assertion failed, compactness would give a limit
\((\mu_*,x_*,\theta_*)\) with \(\Re\mu_*=0\), \(|x_*|\ge\delta\), and
\[
e^{-\mu_*}-p e^{\mathrm{i}x_*}e^{-(1-\theta_*)\mu_*}-q=0 .
\]
Writing \(u=e^{-\mu_*}\) and
\(v=e^{\mathrm{i}x_*}e^{-(1-\theta_*)\mu_*}\), we have \(|u|=|v|=1\). The equation
\(u=pv+q\) can hold with \(p,q>0\) and \(|u|=|v|=1\) only when \(u=v=1\), by
equality in the triangle inequality. Hence \(x_*=0\pmod{2\pi}\), contradicting
\(|x_*|\ge\delta\) with \(x_*\in(-\pi,\pi]\). Therefore phases bounded away
from zero have \(\Re\mu\ge c\) uniformly.
\end{proof}

\subsection{Proof of Lemma~\EClemmaheadingref{lem:uniform_log_perturbation}{15}}

\begin{proof}[Proof of Lemma~\EClemmaheadingref{lem:uniform_log_perturbation}{15}]
We keep the removable denominator in \(A=-P_n/(1-z)\), because this is
the term that must be controlled uniformly. In the chart
\(z=\omega_{m,n}e^{-\nu/n}\), set \(t:=1/n\) and
\[
H(\nu,x,\theta)=e^{-\nu}-p e^{\mathrm{i}x}e^{-(1-\theta)\nu}-q .
\]
Since \(A=-H/(1-z)\) and
\(z A'(z)=-n\partial_\nu A(z)\), the equation
\[
A(z)-\varepsilon z A'(z)=0
\]
is equivalent, after multiplying by the nonzero factor \(-(1-z)\), to
\[
\mathcal J(\nu,\eta,x,\theta,t)
:=
H+\eta H_\nu
-t\eta\frac{z}{1-z}H
=0,
\qquad \eta=n\varepsilon .
\]
At \(\eta=0\), this reduces to \(H(\nu,x,\theta)=0\), whose solution is
\(\nu=\mu(x,\theta)\). Moreover,
\[
\mathcal J_\nu(\mu(x,\theta),0,x,\theta,t)
=H_\nu(\mu(x,\theta),x,\theta),
\]
and this derivative is bounded away from zero uniformly for \(|x|\le\delta\)
and \(\theta\in[0,1]\), after shrinking \(\delta\), because
\(H_\nu(0,0,\theta)=-(q+p\theta)\le -q<0\).

The only uniformity issue is the last term in \(\mathcal J\). On the branches
under consideration \(m\ne0\), and the principal
logarithm has \(|\Im\nu|\le\pi\). Thus, on the unit-scale
neighborhood used by the implicit-function theorem,
\[
|1-z|\ge c/n,
\qquad
\left|t\frac{z}{1-z}\right|\le C .
\]
All derivatives of \(H\) are uniformly bounded on this neighborhood. The
analytic implicit-function theorem therefore applies with constants
independent of \(m,a,n\).

Differentiating
\(\mathcal J(\nu_{m,n}(\eta),\eta,x_m,\theta_n,1/n)=0\) at \(\eta=0\) gives
\[
\nu_{m,n}'(0)
=-\frac{\mathcal J_\eta}{\mathcal J_\nu}
=-\frac{H_\nu}{H_\nu}
=-1,
\]
where the term containing \(H\) vanishes because \(H(\mu(x_m,\theta_n),x_m,\theta_n)=0\).
The same uniform bounds control the second \(\eta\)-derivative of
\(\nu_{m,n}\), so Taylor's theorem gives the stated
\(O(\eta^2)\) remainder.
\end{proof}

\subsection{Proof of Lemma~\EClemmaheadingref{lem:no_stable_islands}{16}}

\begin{proof}[Proof of Lemma~\EClemmaheadingref{lem:no_stable_islands}{16}]
The lemma rules out the only global pathology not excluded by the local
expansion alone: stability could, in principle, reappear at a smaller input
length after the dominant local crossing. We prove a stronger root-existence
statement: throughout the physical range above, at least one root lies outside
the unit disk.

Let \(\eta=n\varepsilon\). We split the interval
\([n\Delta_n+C n^2\Delta_n^2,1]\) into three ranges. On the small-\(\eta\)
range, the dominant branch itself is outside. On an intermediate range, we
construct an outside root in a fixed small phase window. On the large-\(\eta\)
range, the product of the root moduli already exceeds one.

First consider \(\eta\le\eta_0\), where \(\eta_0\) is the constant from
Lemma~\ref{lem:uniform_log_perturbation}. Write the remainder in that lemma as
\(|R(\eta)|\le C_R\eta^2\), and reduce \(\eta_0\), if necessary, so that
\(C_R\eta^2\le \eta/4\) on \(0\le\eta\le\eta_0\). On the dominant phase
\(x_1=2\pi/n\),
Lemma~\ref{lem:uniform_log_perturbation} gives the continuation
\[
\begin{aligned}
z_*(\eta)&=\omega_{1,n}\exp(-\nu_*(\eta)/n),\\
\nu_*(\eta)&=\mu(2\pi/n,\theta_n)-\eta+O(\eta^2)
\end{aligned}
\]
for \(0\le\eta\le\eta_0\). Since
\[
\rho_\infty
=
\exp\{-\Re\mu(2\pi/n,\theta_n)/n\},
\]
we have
\[
\Re\mu(2\pi/n,\theta_n)
=
n\Delta_n+O(n\Delta_n^2).
\]
Thus, after enlarging \(C\), the excess
\(\eta-\Re\mu(2\pi/n,\theta_n)\) dominates both error terms. More explicitly,
if \(\eta\) is comparable to \(n\Delta_n\), then
\[
|R(\eta)|\le C_R\eta^2=O(n^2\Delta_n^2),
\]
so choosing \(C\) larger than the conversion and remainder constants gives
\(\eta-\Re\mu(2\pi/n,\theta_n)>|R(\eta)|\). If \(\eta\) is larger, the reduced
bound \(C_R\eta^2\le\eta/4\) gives the same conclusion. Therefore
\[
\Re\nu_*(\eta)<0
\qquad
\text{whenever}\qquad
\varepsilon\ge \Delta_n+C n\Delta_n^2
\quad\text{and}\quad
\eta\le\eta_0 .
\]
The corresponding root satisfies
\(|z_*(\eta)|=\exp(-\Re\nu_*(\eta)/n)>1\), so the system is unstable throughout
this part of the range.

We next cover an intermediate range that is bounded away from both endpoints.
Set
\[
\eta_1:=\frac{1+p}{2}<1,
\]
and reduce \(\eta_0\), if necessary, so that \(0<\eta_0<\eta_1\). For
\(\eta\in[\eta_0,\eta_1]\), consider the reduced logarithmic equation at
phase \(x=0\),
\[
H(\nu,0,\theta)+\eta H_\nu(\nu,0,\theta)=0 .
\]
Writing \(z=e^{-\nu}\) and \(\gamma:=1-\theta\), and using the positive real
branch of \(z^{1-\theta}\) for \(z\ge1\), this equation becomes
\[
f_{\eta,\theta}(z)
:=
(1-\eta)z
-p\{1-\eta(1-\theta)\}z^{1-\theta}
-q
=0 .
\]
For \(z\ge1\), the function \(f_{\eta,\theta}\) is convex, because
\[
f_{\eta,\theta}''(z)
=p\{1-\eta(1-\theta)\}\theta(1-\theta)z^{-\theta-1}\ge0 .
\]
Moreover,
\[
f_{\eta,\theta}(1)=-\eta(q+p\theta)\le-\eta_0 q<0,
\qquad
f_{\eta,\theta}(z)\to\infty
\quad\text{as }z\to\infty ,
\]
uniformly over \((\eta,\theta)\in[\eta_0,\eta_1]\times[0,1]\). To see the
uniformity, use \(z^{1-\theta}\le z\) for \(z\ge1\) to get
\[
f_{\eta,\theta}(z)
\ge
\{q(1-\eta)+p\eta\theta\}z-q
\ge q(1-\eta_1)z-q .
\]
Hence there is a largest root
\(z_*(\eta,\theta)>1\). At this root,
\[
f_{\eta,\theta}'(z_*)
=(1-\eta)\theta+\frac{(1-\theta)q}{z_*}>0 ,
\]
where we used
\[
p\{1-\eta(1-\theta)\}z_*^{1-\theta}
=(1-\eta)z_*-q .
\]
At the endpoint cases \(\theta=0\) and \(\theta=1\), the equation reduces to
the same explicit root \(z_*=(1-\eta)^{-1}\), and the derivative remains
positive. The positive derivative makes this root locally continuous in
\((\eta,\theta)\) by the implicit-function theorem, and the preceding uniform
growth bound keeps it in a compact \(z\)-interval. Therefore compactness over
\((\eta,\theta)\in[\eta_0,\eta_1]\times[0,1]\) gives constants \(r_0,m_0>0\)
such that
\[
\begin{aligned}
z_*(\eta,\theta)&\ge e^{r_0}>1,\\
f_{\eta,\theta}'(z_*(\eta,\theta))&\ge m_0 .
\end{aligned}
\]
Equivalently, the reduced equation has a simple solution
\(\nu_0(\eta,\theta):=-\log z_*(\eta,\theta)\) with
\(\Re\nu_0\le-r_0\), uniformly over this compact parameter set.

We then realize this reduced outside root as an actual finite-\(n\) root. The
derivative lower bound above lets us apply the implicit-function theorem
uniformly once more, now with the phase \(x\) also allowed to vary; the
fractional power is taken using the local analytic branch around the positive
real root \(z_*(\eta,\theta)\). Thus, after
choosing \(x_0>0\) sufficiently small, the reduced equation
\[
H(\nu,x,\theta)+\eta H_\nu(\nu,x,\theta)=0
\]
has a solution \(\nu_0(\eta,\theta,x)\) for every
\(\eta\in[\eta_0,\eta_1]\), \(\theta\in[0,1]\), and
\(x\in[0,2x_0]\), with
\[
\Re\nu_0(\eta,\theta,x)\le-\frac{3r_0}{4}.
\]
Choose \(I=[x_0,2x_0]\subset(0,\delta)\), where \(\delta\) is from
Lemma~\ref{lem:near_unit_chart}. We use the following elementary counting fact:
for all large \(n\), the lattice \(\{2\pi m/n\}\) has at least \(c_I n\) points
in \(I\), while the set of \(n\)th roots satisfying \(|\omega-1|\le s\) has at
most \(C_s s n\) points. Because \(m\mapsto u_m\) is a permutation modulo \(n\),
choosing \(s>0\) small enough leaves at least one phase \(x_m\in I\) with
\(|\omega_{m,n}-1|\ge s\). For this phase and for \(\nu\) in
a fixed neighborhood of \(\nu_0(\eta,\theta_n,x_m)\),
\[
\left|\frac1n\frac{z}{1-z}\right|=O(n^{-1}),
\qquad
z=\omega_{m,n}e^{-\nu/n}.
\]
Thus the full finite-\(n\) logarithmic equation
\[
H+\eta H_\nu
-\frac{\eta}{n}\frac{z}{1-z}H=0
\]
is an \(O(n^{-1})\) perturbation of the reduced equation, uniformly over
\(\eta\in[\eta_0,\eta_1]\). The implicit-function theorem gives a nearby
solution \(\nu_{m,n}(\eta)\) satisfying
\[
\Re\nu_{m,n}(\eta)\le -\frac{r_0}{2}.
\]
The associated root has
\[
|z_{m,n}(\eta)|
=\exp\{-\Re\nu_{m,n}(\eta)/n\}>1 .
\]
Therefore the \finiteinput{} polynomial is unstable throughout
\(\eta\in[\eta_0,\eta_1]\).

For the large-\(\eta\) range, consider \(\eta\in[\eta_1,1]\). The
\finiteinput{} polynomial has degree \(n-1\), constant coefficient \(q\), and
leading coefficient
\[
1-\varepsilon(n-1)=1-\eta+\frac{\eta}{n}.
\]
The product of the moduli of all \(n-1\) roots is therefore
\[
\frac{q}{1-\eta+\eta/n}.
\]
For \(\eta\ge\eta_1=(1+p)/2=1-q/2\), the denominator is at most
\(q/2+1/n\), which is strictly smaller than \(q\) for all sufficiently large
\(n\). Hence the product of the root moduli is larger than one, and at least
one root must lie outside the unit disk.

Combining the three ranges proves instability for every
\(\varepsilon\in[\Delta_n+C n\Delta_n^2,1/n]\). Stable islands before the
asymptotic threshold are therefore impossible.
\end{proof}

\subsection{Proof of Lemma~\EClemmaheadingref{lem:strict_cohort}{17}}

\begin{proof}[Proof of Lemma~\EClemmaheadingref{lem:strict_cohort}{17}]
Admission creates a fresh mixed cohort in the exact proportions \(p:q\). The
execute step advances each surviving cohort by one age; class 1 leaves after age
\(a-1\), while class 2 leaves after age \(b-1\). Proportional within-stage
eviction rescales all classes in a mixed stage by the same factor and therefore
preserves the within-stage composition. Induction over \(n\) gives the displayed
cohort representation. For \(j<a\), the total occupancy is
\(x^n_{1,j}+x^n_{2,j}=y^{n}_{j}\); for \(j\ge a\), only class 2 remains, so
the total occupancy is \(q y^{n}_{j}\). Hence the normalized variables used
above satisfy \(Y^n_j=y^{n}_{j}\), with the effective stage weights stated in
the lemma.
\end{proof}

\section{Numerical Supplement}
\label{app:numerical_details}

This appendix contains the numerical experiments behind Section~\ref{sec:numerics}.
The main text keeps one representative figure per mechanism. The supplementary
figures give the supporting model-based simulations, Vidur simulations,
real-GPU experiments, and parameter sweeps.

Throughout this appendix, ``model-based simulator'' denotes a direct
implementation of the discrete-time memory and admission recursions in
Section~\ref{sec:model}. It keeps the model's state variables, capacity
constraint, greedy admission rule, and LPF eviction rule, and adds stochastic
arrivals only in experiments that explicitly study open-system behavior. These
runs use the paper's state equations as the simulator and isolate the
mathematical mechanism.

``Vidur'' denotes the large-scale LLM-serving simulator of
\citet{agrawal2024vidur}, which reports simulation errors below 5\% for
latency and throughput. Unless otherwise stated, our Vidur experiments use the
Sarathi scheduler~\citep{sarathi2024}, Meta-Llama-3-8B, and a single NVIDIA
A100 GPU. Vidur simulates serving iterations and tracks request-level KV cache
memory, so it serves as an intermediate simulation environment between the
model-based simulator and physical GPU experiments. In the Vidur simulations,
pending requests are considered in arrival order and admitted when memory is
available; the per-iteration admission cap
based on the eviction-free admission rate is added only in the rate-limited
variant. When resident KV cache later exceeds capacity, Vidur records an
eviction event and returns the affected request to the waiting queue. We report
these events as the system-level counterparts of model evictions. The eviction
priority follows
the LPF abstraction used in the model, removing least-progressed active
requests first.

\subsection{Model-based GCD simulations}
\label{app:additional_gcd_simulations}

The model-based simulator implements the discrete-time memory and admission
dynamics of Section~\ref{sec:model}. Time is measured in scheduling iterations,
memory is measured in the same normalized KV cache units as the model, and a
newly admitted request contributes its \inputlen{} KV cache before growing by
one unit per decoding step. Figures~\ref{fig:exp_1} and~\ref{fig:exp_2}
isolate the deterministic GCD mechanism under saturated input: coprime decoding lengths desynchronize
completions and converge to the eviction-free equilibrium, whereas non-coprime
decoding lengths align completions and generate eviction cycles.

\begin{figure}[ht]
    \centering
    \includegraphics[width=\linewidth]{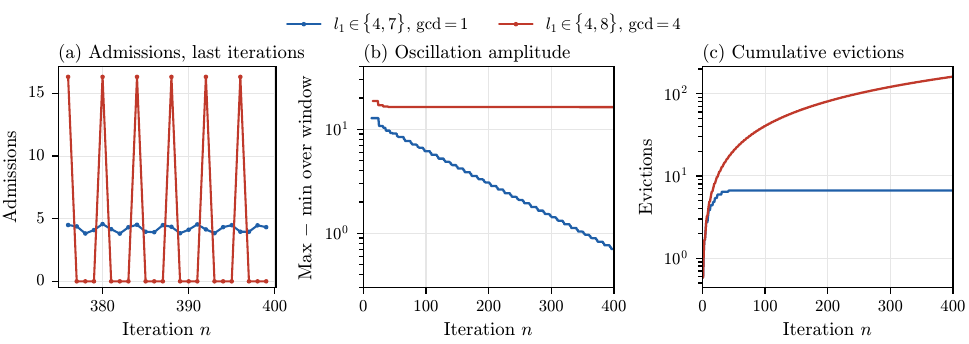}
    \caption{Two-class model-based GCD comparison under greedy admission with
    common \inputlen{} \(l_0=40\) and \(p=1/2\). The coprime decoding-length
    pair \((l_{1,1},l_{1,2})=(4,7)\) converges to an eviction-free
    equilibrium, while the non-coprime pair \((4,8)\) exhibits persistent
    oscillations and eviction. (a)~Admissions over the last 24 iterations. (b)~Oscillation amplitude, measured as the range of admissions over the preceding \(2\max_k l_{1,k}\) iterations, on a log scale: it decays geometrically for the coprime pair and stays constant for the non-coprime pair. (c)~Cumulative evictions on a log scale.}
    \label{fig:exp_1}
\end{figure}

\begin{figure}[ht]
    \centering
    \includegraphics[width=\linewidth]{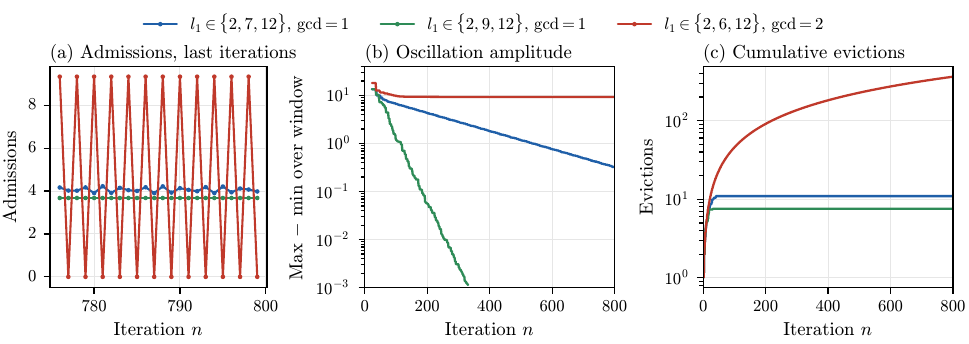}
    \caption{Three-class model-based GCD comparison under greedy admission
    with common \inputlen{} \(l_0=30\) and equal class proportions. The
    coprime decoding-length sets \(l_1\in\{2,7,12\}\) and
    \(l_1\in\{2,9,12\}\) converge to eviction-free equilibria; the non-coprime
    set \(l_1\in\{2,6,12\}\) exhibits persistent oscillations. (a)~Admissions over the last 24 iterations. (b)~Oscillation amplitude, measured as the range of admissions over the preceding \(2\max_k l_{1,k}\) iterations, on a log scale (values below \(10^{-3}\) are not shown). (c)~Cumulative evictions on a log scale.}
    \label{fig:exp_2}
\end{figure}

\FloatBarrier

\subsection{Mixing mechanisms and additional metrics}
\label{app:mixing_details}

Section~\ref{subsec:request_mixing} gives the representative serving-system
mixing result. Figures~\ref{fig:mixing_comparison},
\ref{fig:mixing_summary_cpu}, \ref{fig:throughput_mixing_curve},
\ref{fig:mixing_partial_gcd}, \ref{fig:vidur_mixing_coprime}, and
\ref{fig:vidur_mixing_partial_gcd}, together with the accompanying tables, separate the model-based
mechanism from the metric changes induced by coprime mixing and partial-GCD
reduction. In the model-based runs, routing changes only the set of decoding
lengths pooled on a node; in the Vidur runs, the same routing change is tested
under Sarathi's serving simulation, which tracks request-level KV cache memory.

\begin{figure}[ht]
    \centering
    \includegraphics[width=0.95\linewidth]{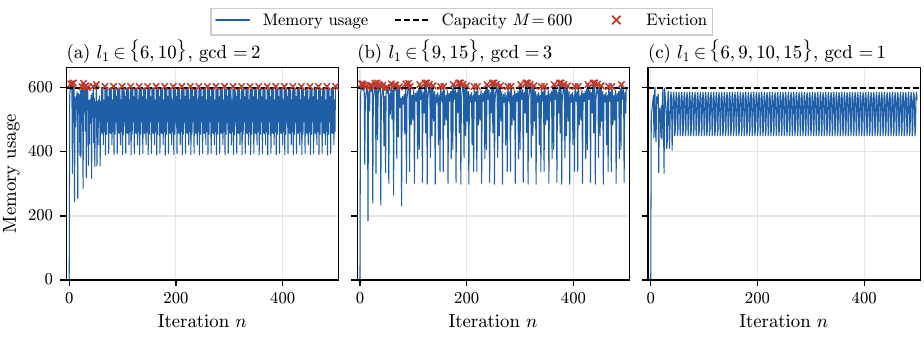}
    \caption{Request mixing dynamics in the model-based simulator with
    \(M=600\) and common \inputlen{} \(l_0=30\). Segregated routing assigns
    decoding lengths \(\{6,10\}\) to one node and \(\{9,15\}\) to another,
    producing periodic eviction because the two nodes have non-coprime GCDs.
    Uniformly mixing all four decoding lengths \(l_1\in\{6,9,10,15\}\) gives
    effective GCD \(=1\), desynchronizes completions, and converges to an
    eviction-free equilibrium. Memory usage is recorded after execution and before eviction, so it can exceed \(M\).}
    \label{fig:mixing_comparison}
\end{figure}

\begin{table}[htbp]
\centering
\caption{Request mixing effect (model-based simulator with $M=600$, common \inputlen{} $l_0=30$, and deterministic arrivals).
Segregating classes on separate nodes causes periodic eviction due to non-coprime GCDs (2 and 3).
Uniformly mixing classes creates effective GCD$=1$ and is consistent with near-zero eviction through desynchronization in this setup.
Mixing cuts evictions from 88 (segregated average) to 0.}
\label{tab:mixing_effect}
\small
\begin{tabular}{@{}lcc@{}}
\toprule
\textbf{Routing Strategy} & \textbf{Evictions} & \textbf{Configuration} \\
\midrule
Segregated $\{6,10\}$ & 63 & $l_1\in\{6,10\}$ \\
Segregated $\{9,15\}$ & 112 & $l_1\in\{9,15\}$ \\
\midrule
Segregated (average) & 88 & --- \\
\midrule
\textbf{Mixed (all classes)} & \textbf{0} & $l_1\in\{6,9,10,15\}$ \\
\bottomrule
\end{tabular}
\vspace{1mm}
\parbox{0.95\linewidth}{\footnotesize Note. Arrivals are deterministic: each iteration brings 100 requests of each class to a segregated node and 50 of each class to the mixed node. Evictions are counted over the last 1{,}000 of 3{,}000 iterations. Segregated average is the average of the two segregated nodes. Mixed uses the same two-node budget with all request classes pooled.}
\end{table}

\begin{figure}[ht]
    \centering
    \includegraphics[width=0.6\linewidth]{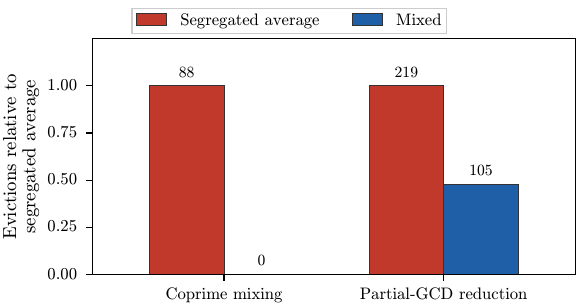}
    \caption{Model-based simulator summary under deterministic arrivals:
    evictions under mixed routing relative to the segregated average, with raw
    counts above the bars and counting windows as in
    Tables~\ref{tab:mixing_effect} and~\ref{tab:mixing_partial_gcd}. Coprime
    mixing eliminates evictions over this simulated horizon, and partial-GCD
    reduction cuts them by 52\%.}
    \label{fig:mixing_summary_cpu}
\end{figure}

\begin{figure}[ht]
    \centering
    \includegraphics[width=0.98\linewidth]{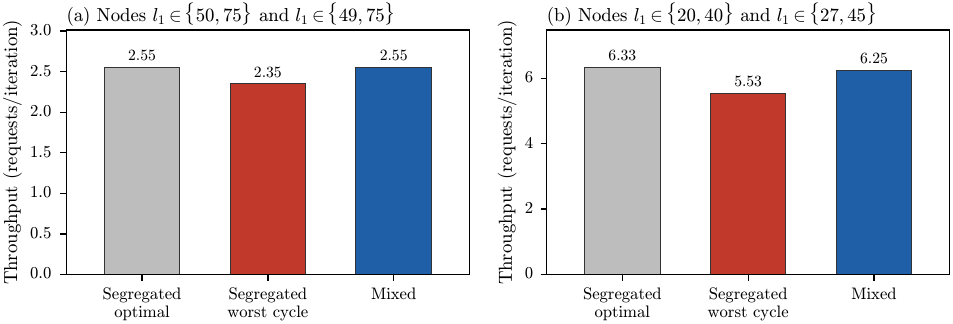}
    \caption{Theoretical throughput comparison for request mixing with
    \(M=10{,}000\) and \(l_0=30\). Mixing avoids the synchronized-cycle
    throughput penalty; the mixed and segregated eviction-free optima can
    differ because pooling changes the workload composition on each node.}
    \label{fig:throughput_mixing_curve}
\end{figure}

\begin{figure}[ht]
    \centering
    \includegraphics[width=0.95\linewidth]{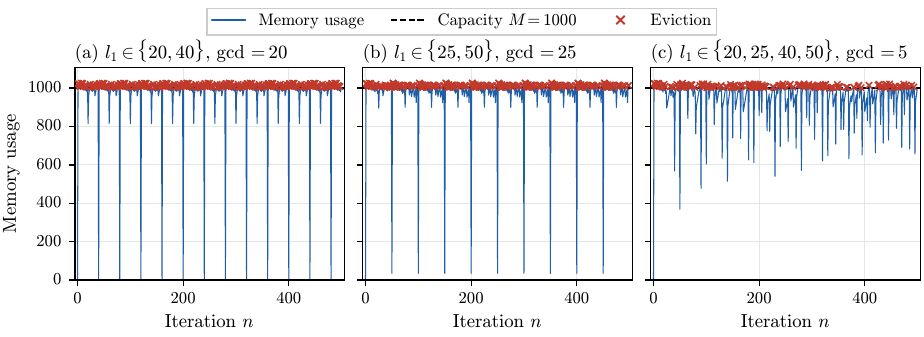}
    \caption{Partial GCD reduction in the model-based simulator with
    \(M=1000\) and common \inputlen{} \(l_0=30\). Segregated nodes with
    decoding lengths \(\{20,40\}\) and \(\{25,50\}\) have high GCDs and severe
    periodic eviction; mixing all four classes reduces the effective GCD to 5 and
    cuts eviction pressure. Memory usage is recorded after execution and before eviction, so it can exceed \(M\).}
    \label{fig:mixing_partial_gcd}
\end{figure}

\begin{table}[htbp]
\centering
\caption{Partial GCD reduction effect (model-based simulator with $M=1000$, common \inputlen{} $l_0=30$, and deterministic arrivals).
Segregating classes on separate nodes with high GCDs (20 and 25) causes severe periodic eviction.
Uniformly mixing classes reduces effective GCD to 5, achieving 52\% fewer evictions through partial desynchronization.
This demonstrates that mixing benefits extend beyond coprime configurations.}
\label{tab:mixing_partial_gcd}
\small
\begin{tabular}{@{}lcc@{}}
\toprule
\textbf{Routing Strategy} & \textbf{Evictions} & \textbf{Configuration} \\
\midrule
Segregated $\{20,40\}$ & 248 & $l_1\in\{20,40\}$, GCD=20 \\
Segregated $\{25,50\}$ & 190 & $l_1\in\{25,50\}$, GCD=25 \\
\midrule
Segregated (average) & 219 & --- \\
\midrule
\textbf{Mixed (all classes)} & \textbf{105} & $l_1\in\{20,25,40,50\}$, GCD=5 \\
\bottomrule
\end{tabular}
\vspace{1mm}
\parbox{0.95\linewidth}{\footnotesize Note. Arrivals are deterministic: each iteration brings 50 requests of each class to a segregated node and 25 of each class to the mixed node. Evictions are counted over the last 500 of 2{,}000 iterations. Segregated average is the average of the two segregated nodes. Mixed uses the same two-node budget with all request classes pooled.}
\end{table}

\begin{figure}[ht]
    \centering
    \includegraphics[width=0.95\linewidth]{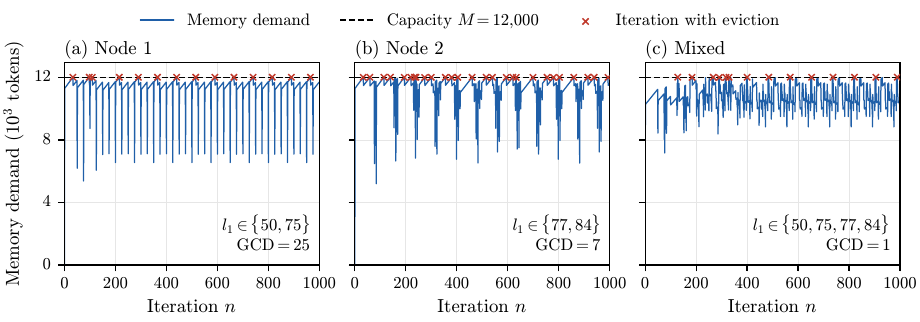}
    \caption{Coprime mixing in the model-based simulator with the decoding-length sets of \Cref{tab:vidur_mixing_coprime} (common \inputlen{} \(l_0=512\), capacity \(M=12{,}000\) tokens). The segregated nodes (\(l_1\in\{50,75\}\), GCD 25, and \(l_1\in\{77,84\}\), GCD 7) overflow at regular intervals. The mixed node (GCD 1) still overflows occasionally, but at irregular times. Over 1,000 iterations, evictions occur in 14 and 27 iterations on the segregated nodes and in 14 iterations on the mixed node; markers indicate these iterations.}
    \label{fig:vidur_mixing_coprime}
\end{figure}

\begin{figure}[ht]
    \centering
    \includegraphics[width=0.95\linewidth]{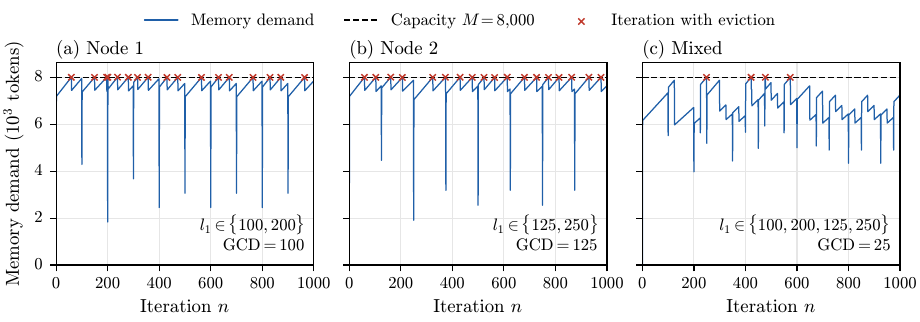}
    \caption{Partial-GCD mixing in the model-based simulator with the decoding-length sets of \Cref{tab:vidur_mixing_partial_gcd} (common \inputlen{} \(l_0=512\), capacity \(M=8{,}000\) tokens). Mixing \(l_1\in\{100,200\}\) and \(l_1\in\{125,250\}\) reduces the GCD from 100 and 125 to 25. Over 1,000 iterations, the number of iterations with an eviction falls from 17 and 18 on the segregated nodes to 4 on the mixed node; markers indicate these iterations.}
    \label{fig:vidur_mixing_partial_gcd}
\end{figure}

\begin{table}[htbp]
\centering
\caption{Request mixing effect with 20k requests, common \inputlen{} $l_0=512$, and overloaded arrivals.
Segregating classes on separate nodes causes periodic eviction due to non-coprime GCDs (25 and 7).
Uniformly mixing classes creates effective GCD$=1$ and has 98.7\% fewer evictions.
Holding the workload, scheduler, and two-device budget fixed, mixing cuts evictions from 1,349 to 18, improves latency by 17.4\%, and increases throughput by 19.0\% versus the segregated average.}
\label{tab:vidur_mixing_coprime}
\small
\begin{tabular}{@{}lcccc@{}}
\toprule
\textbf{Routing} & \textbf{GCD} & \textbf{Evictions} & \textbf{Latency (s)} & \textbf{Throughput (req/s)} \\
\midrule
Segregated, $l_1\in\{50,75\}$ & 25 & 1,059 & 136.3 & 73.4 \\
Segregated, $l_1\in\{77,84\}$ & 7 & 1,640 & 183.2 & 54.8 \\
\midrule
Segregated (average) & --- & 1,349 & 159.8 & 64.1 \\
\midrule
\textbf{Mixed}, $l_1\in\{50,75,77,84\}$ & \textbf{1} & \textbf{18} & \textbf{131.9} & \textbf{76.2} \\
\bottomrule
\end{tabular}
\vspace{1mm}
\parbox{0.95\linewidth}{\footnotesize Note. Segregated average is the average of the two segregated devices. Mixed uses the same two-device budget with all request classes pooled. Throughput is completed requests per second.}
\end{table}

\begin{table}[htbp]
\centering
\caption{Partial GCD reduction effect with 20k requests and common \inputlen{} $l_0=512$.
Segregating classes on separate nodes with high GCDs (100 and 125) causes severe periodic eviction.
Uniformly mixing classes reduces effective GCD to 25 and has 34.0\% fewer evictions.
Holding the workload, scheduler, and two-device budget fixed, mixing cuts evictions by 1,259 (3,708$\rightarrow$2,449), improves latency by 7.2\%, and increases throughput by 5.8\% versus the segregated average.
This demonstrates that mixing benefits extend beyond coprime configurations.}
\label{tab:vidur_mixing_partial_gcd}
\small
\begin{tabular}{@{}lcccc@{}}
\toprule
\textbf{Routing} & \textbf{GCD} & \textbf{Evictions} & \textbf{Latency (s)} & \textbf{Throughput (req/s)} \\
\midrule
Segregated, $l_1\in\{100,200\}$ & 100 & 3,296 & 367.9 & 27.3 \\
Segregated, $l_1\in\{125,250\}$ & 125 & 4,121 & 473.4 & 21.2 \\
\midrule
Segregated (average) & --- & 3,708 & 420.6 & 24.2 \\
\midrule
\textbf{Mixed}, $l_1\in\{100,125,200,250\}$ & \textbf{25} & \textbf{2,449} & \textbf{390.4} & \textbf{25.6} \\
\bottomrule
\end{tabular}
\vspace{1mm}
\parbox{0.95\linewidth}{\footnotesize Note. Segregated average is the average of the two segregated devices. Mixed uses the same two-device budget with all request classes pooled. Throughput is completed requests per second.}
\end{table}

\FloatBarrier

\subsection{Vidur and real-GPU experiments}
\label{app:real_system_diagnostics}

Figures~\ref{fig:gcd_effect_vidur}, \ref{fig:gpu_single_type},
\ref{fig:gpu_two_type_grid}, and~\ref{fig:sglang_mixing_ts} ask whether the observed dynamics are consistent with
the synchronization mechanism after reintroducing serving-system details. The
real-GPU experiments use an implemented serving stack. These experiments test
whether similar completion synchronization, memory growth, and eviction
patterns appear beyond the stylized simulator.

In the fixed-workload Vidur and real-GPU experiments, requests are generated
according to the stated class proportions and submitted at a sufficiently high
rate at the beginning of the run. They therefore form a persistent waiting queue
\(Q^n\), and the scheduler admits from this queue whenever memory is available.
This construction approximates the saturated-input regime analyzed in the
model, while the open-system experiments in \Cref{subsec:gcd_effect} separately
study Poisson arrivals.

\begin{figure}[ht]
    \centering
    \includegraphics[width=0.95\linewidth]{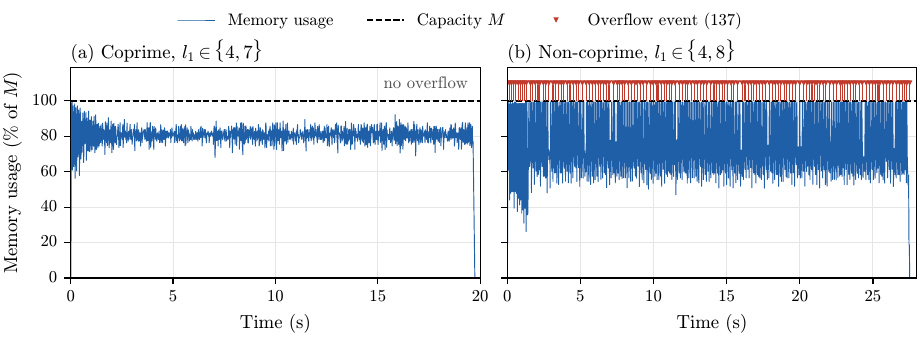}
    \caption{GCD effect on memory usage in Vidur with common \inputlen{} \(l_0=512\), rate \(=7000\), and 20k requests. Memory usage is recorded after eviction, so it never exceeds capacity. Coprime decoding lengths \(\{4,7\}\) settle near 80\% of capacity without eviction. Non-coprime decoding lengths \(\{4,8\}\) oscillate between about 60\% of capacity and capacity itself, record 137 overflow events, each resolved by eviction (red marks above the capacity line), and take longer to finish the workload. The simulator logs the number of overflow events but not their times, so the marks are spread evenly over the run.}
    \label{fig:gcd_effect_vidur}
\end{figure}

\begin{figure}[ht]
    \centering
    \begin{minipage}{0.48\linewidth}
        \centering
        \includegraphics[width=\linewidth]{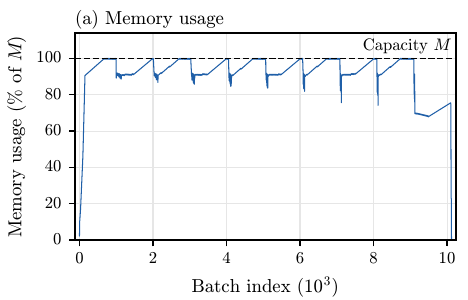}
    \end{minipage}
    \hfill
    \begin{minipage}{0.48\linewidth}
        \centering
        \includegraphics[width=\linewidth]{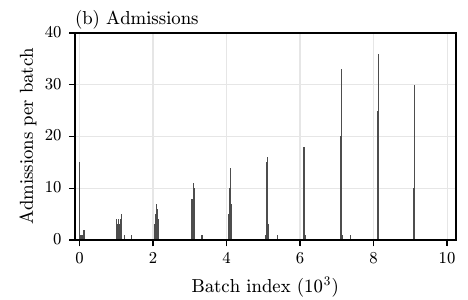}
    \end{minipage}
    \caption{Homogeneous real-GPU experiment (\(l_0=5000\), \(l_1=1000\)). The panels
    plot physical memory usage and admitted requests against the index of logged decoding batches. Memory
    repeatedly saturates capacity, triggers eviction, and drops; admissions
    mirror the same cycle.}
    \label{fig:gpu_single_type}
\end{figure}

\begin{figure}[ht]
    \centering
    \begin{minipage}{0.48\linewidth}
        \centering
        \includegraphics[width=\linewidth]{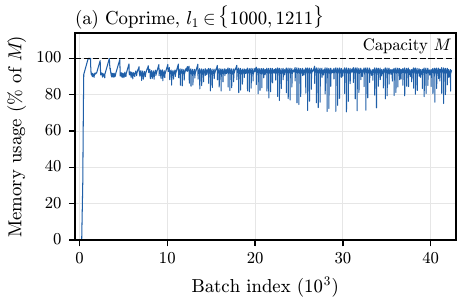}
    \end{minipage}
    \hfill
    \begin{minipage}{0.48\linewidth}
        \centering
        \includegraphics[width=\linewidth]{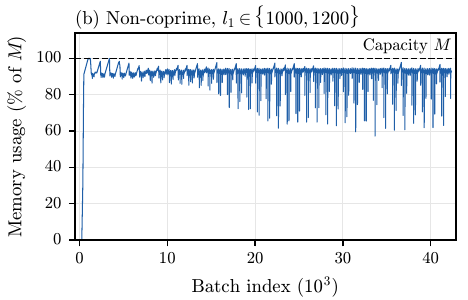}
    \end{minipage}

    \vspace{0.5em}

    \begin{minipage}{0.48\linewidth}
        \centering
        \includegraphics[width=\linewidth]{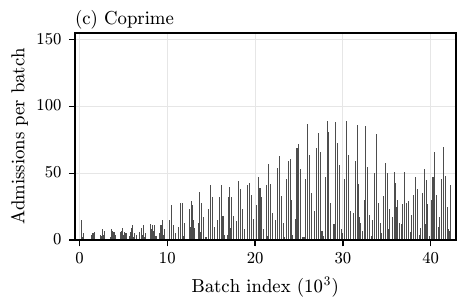}
    \end{minipage}
    \hfill
    \begin{minipage}{0.48\linewidth}
        \centering
        \includegraphics[width=\linewidth]{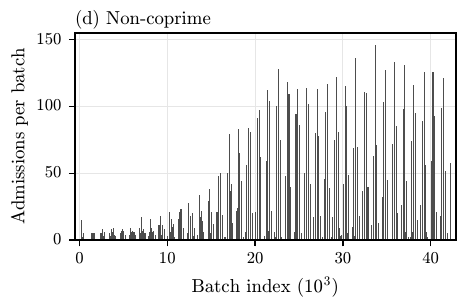}
    \end{minipage}
    \caption{Two-class real-GPU experiments with common \inputlen{}
    \(l_0=5000\). Columns compare the coprime decoding-length pair
    \((1000,1211)\) with the non-coprime pair \((1000,1200)\); rows show
    physical memory usage and admitted requests against the batch index. The non-coprime case exhibits
    stronger synchronized memory pressure.}
    \label{fig:gpu_two_type_grid}
\end{figure}

\begin{figure}[ht]
    \centering
    \includegraphics[width=0.98\linewidth]{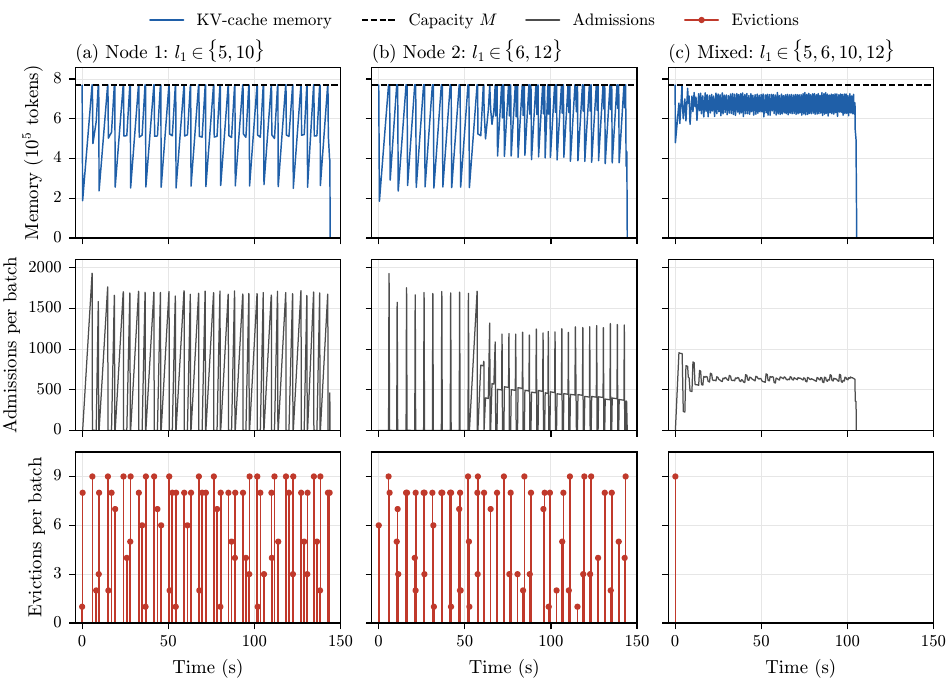}
    \caption{KV-cache memory (top), admissions (middle), and evictions (bottom) over wall-clock
    time for three decoding-length configurations on real GPU hardware.
    Segregated routing (\(l_1\in\{5,10\}\) and \(l_1\in\{6,12\}\)) produces
    recurring eviction bursts at fixed intervals; mixed routing
    (\(l_1\in\{5,6,10,12\}\), GCD\(=1\)) eliminates the periodic pattern after a single eviction batch at start-up.}
    \label{fig:sglang_mixing_ts}
\end{figure}

\FloatBarrier

\subsection{Request mixing with benchmark workloads on GPUs}
\label{app:benchmark_mixing}

Benchmark traffic differs from the designed workloads above in two ways that matter for the synchronization mechanism. Decoding lengths vary even within a single benchmark, and nothing forces them to share a common divisor. Natural length variation might therefore desynchronize completions on its own and leave little room for mixing. This subsection tests whether mixing still reduces evictions and raises throughput when requests come from standard LLM evaluation benchmarks: IFEval~\citep{zhou2023ifeval}, which checks instruction following, and GPQA Diamond~\citep{rein2024gpqa}, a set of graduate-level science questions. All runs serve Qwen3-1.7B~\citep{qwen3technicalreport2025} with the SGLang serving framework~\citep{sglang2024} on two homogeneous NVIDIA H20 GPUs, one serving instance per GPU. Prefix caching is disabled, so each request holds its own KV cache, as in the model. The full request set is submitted at the start of each run, and each GPU keeps a nonempty waiting queue throughout the analysis window. When memory runs out, SGLang discards the KV cache of a running request and later recomputes the request; we count these events as evictions.

The experiments cover two output regimes. Under natural outputs, each request decodes until its end-of-sequence token or the context limit, so the model determines the decoding lengths. Under controlled outputs, each request decodes to a prescribed length, which fixes the offered work and makes the decoding-length composition of each GPU explicit. The natural-output experiment changes only the within-GPU queue order and holds the request-to-GPU assignment fixed. Of the two controlled-output experiments, one changes the queue order and the other changes the routing of requests to GPUs. We report evictions, which measure discarded service; mean latency; and throughput in completed output tokens per second per GPU, since each decoding iteration produces one output token for every running request. KV-cache capacities and analysis windows differ across experiments, so comparisons are made only within an experiment.

With natural outputs, the comparison changes only the order in which each GPU serves its queue (Panel~A of \Cref{tab:benchmark_mixing_natural}). The workload has 6,546 requests, 588 from GPQA Diamond and 5,958 from IFEval, and each GPU has a KV-cache capacity of 14,336 tokens. Both configurations send every request to the same GPU, and the profiled service work of the two GPUs differs by less than 0.001\%. In the baseline, each GPU dispatches requests from the same benchmark in contiguous blocks, so requests with similar decoding lengths enter service together and tend to complete together. The mixed order interleaves the two benchmarks within the first 48 dispatches on each GPU and keeps the baseline order thereafter; a matched-output check finds no material change in benchmark accuracy. Over the 60--340-second window, mixing reduces evictions by 55.9\% and raises throughput by 27.4\%. Mean latency falls by 38.8\%, although this contrast partly reflects composition, because the two orders complete different benchmark mixtures within the window.

\begin{table}[ht]
\centering
\caption{Request mixing with natural benchmark outputs (Qwen3-1.7B, two NVIDIA H20 GPUs). Changes are relative to the first row.}
\label{tab:benchmark_mixing_natural}
\small
\begin{threeparttable}
\setlength{\tabcolsep}{4.5pt}
\begin{tabular}{@{}lcccccc@{}}
\toprule
\textbf{Configuration} &
\multicolumn{2}{c}{\textbf{Evictions}} &
\multicolumn{2}{c}{\textbf{Mean latency (s)}} &
\multicolumn{2}{c}{\textbf{Throughput}} \\
\cmidrule(lr){2-3} \cmidrule(lr){4-5} \cmidrule(lr){6-7}
& Value & Change & Value & Change & Value & Change \\
\midrule
\multicolumn{7}{@{}l}{\textit{(A) GPQA Diamond and IFEval: queue order changed (one run)}} \\
\quad Same-benchmark blocks & 34 & --- & 172.19 & --- & 141.85 & --- \\
\quad \textbf{Mixed order} & \textbf{15} & \textbf{$-$55.9\%} & \textbf{105.33} & \textbf{$-$38.8\%} & \textbf{180.66} & \textbf{+27.4\%} \\
\bottomrule
\end{tabular}
\begin{tablenotes}[flushleft]\footnotesize
\item Note. Throughput is completed output tokens per second per GPU. Panel~A is measured over the 60--340-second window: latency is averaged over requests completed in the window, and throughput counts completed output tokens in the window. Evictions are summed over both GPUs.
\end{tablenotes}
\end{threeparttable}
\end{table}

Controlled outputs fix every decoding length in advance and therefore hold the offered work fixed. Panel~B of \Cref{tab:benchmark_mixing_controlled} uses the benchmark composition of Panel~A and fixes each response at the length of a previously generated natural answer. The exact-length responses supply a controlled systems load, while benchmark scores refer to the natural answers. As in Panel~A, only the within-GPU order changes, here from alternating blocks of 24 same-benchmark requests to an interleaved order. Mixing reduces evictions by 59.5\% and raises throughput by 68.9\%; within the window, the mixed order completes 492 requests and the baseline completes 100. Mean latency falls by 67.8\% and 95th-percentile latency by 51.4\%. The macro-average latency, which weights the two benchmarks equally and thus removes the composition effect, falls by 17.6\%.

Panel~C tests the decoding-length GCD condition behind Theorem~\ref{thm:gcd} directly. The workload consists of 7,520 IFEval-derived requests with exact decoding lengths of 256, 1,275, 1,285, and 2,304 tokens. Each GPU receives 3,760 requests, and the token-equivalent work in the initial queues of the two GPUs differs by about 0.03\% in both configurations. Segregated routing serves lengths 256 and 2,304 (GCD 256) on one GPU and lengths 1,275 and 1,285 (GCD 5) on the other. Mixed routing gives each GPU equal shares of all four lengths, whose GCD is 1. The mixed configuration also uses a mixing-aware queue order, so the comparison evaluates routing and queue order jointly. Mixing nearly eliminates evictions, from 59.5 to 0.5 per run, increases throughput by 20.2\%, and lowers mean latency by 6.5\%. This pattern matches the GCD condition: the segregated GPUs carry common periods of 256 and 5 tokens, and mixing removes them.

\begin{table}[ht]
\centering
\caption{Request mixing with controlled decoding lengths (Qwen3-1.7B, two NVIDIA H20 GPUs). Changes are relative to the baseline row of each panel.}
\label{tab:benchmark_mixing_controlled}
\small
\begin{threeparttable}
\setlength{\tabcolsep}{4pt}
\begin{tabular}{@{}lccccccc@{}}
\toprule
\textbf{Configuration} & \textbf{GCD} &
\multicolumn{2}{c}{\textbf{Evictions}} &
\multicolumn{2}{c}{\textbf{Mean latency (s)}} &
\multicolumn{2}{c}{\textbf{Throughput}} \\
\cmidrule(lr){3-4} \cmidrule(lr){5-6} \cmidrule(lr){7-8}
& & Value & Change & Value & Change & Value & Change \\
\midrule
\multicolumn{8}{@{}l}{\textit{(B) GPQA Diamond and IFEval, exact-length replay: queue order changed (one run)}} \\
\quad Same-benchmark blocks & --- & 42 & --- & 122.74 & --- & 653.5 & --- \\
\quad \textbf{Mixed order} & --- & \textbf{17} & \textbf{$-$59.5\%} & \textbf{39.50} & \textbf{$-$67.8\%} & \textbf{1{,}104.1} & \textbf{+68.9\%} \\
\midrule
\multicolumn{8}{@{}l}{\textit{(C) IFEval-derived requests, $l_1\in\{256,1{,}275,1{,}285,2{,}304\}$: routing changed (two paired runs)}} \\
\quad Node 1, $l_1\in\{256,2{,}304\}$ & 256 & 27.0 & --- & 286.91 & --- & 113.23 & --- \\
\quad Node 2, $l_1\in\{1{,}275,1{,}285\}$ & 5 & 32.5 & --- & 307.22 & --- & 111.54 & --- \\
\quad Segregated & --- & 59.5 & --- & 297.06 & --- & 112.38 & --- \\
\quad \textbf{Mixed} & \textbf{1} & \textbf{0.5} & \textbf{$-$99.2\%} & \textbf{277.78} & \textbf{$-$6.5\%} & \textbf{135.06} & \textbf{+20.2\%} \\
\bottomrule
\end{tabular}
\begin{tablenotes}[flushleft]\footnotesize
\item Note. Throughput is completed output tokens per second per GPU. Panel~B uses the 60--340-second window and Panel~C the 90--480-second window. Latency is averaged over requests completed in the window, and throughput counts completed output tokens in the window. Evictions are summed over both GPUs and averaged over paired runs; Node rows report a single GPU. The Segregated row averages the two nodes for latency and throughput and sums them for evictions. In Panel~B the exact lengths replay the natural outputs, so the decoding lengths have no common divisor structure to report.
\end{tablenotes}
\end{threeparttable}
\end{table}

In each experiment, the baseline queue order or decoding-length composition synchronizes completions, and mixing breaks this synchronization: evictions fall and throughput rises (Panels A, B, and C). The effect is largest when mixing turns a non-coprime decoding-length composition into a coprime one (Panel C).

\FloatBarrier

\subsection{Rate-limit mechanism and sensitivity}
\label{app:rate_limit_details}

Section~\ref{subsec:rate_limited_admission} gives the Vidur rate-limited
admission result. Figures~\ref{fig:rate_limit_dynamics_cpu} and~\ref{fig:rate_limit_tradeoff_cpu}
illustrate the corresponding
mechanism in the model-based simulator and show the throughput-eviction tradeoff across
admission limits.

\begin{figure}[!htbp]
    \centering
    \includegraphics[width=0.95\linewidth]{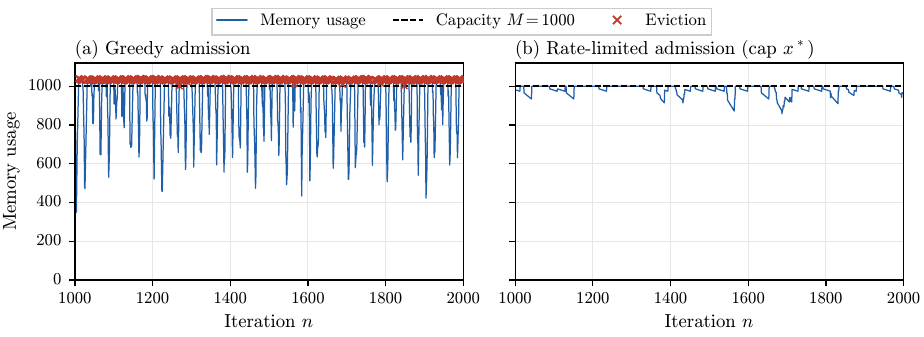}
    \caption{Rate-limited admission dynamics in the model-based simulator with
    \(M=1000\), \(l_0=20\), and \(l_1=20\). Greedy admission produces periodic
    memory overflow and eviction spikes; rate-limited admission keeps memory
    at or below capacity and eliminates eviction over the simulated horizon. Memory usage is recorded after execution and before eviction.}
    \label{fig:rate_limit_dynamics_cpu}
\end{figure}

\begin{figure}[ht]
    \centering
    \includegraphics[width=0.8\linewidth]{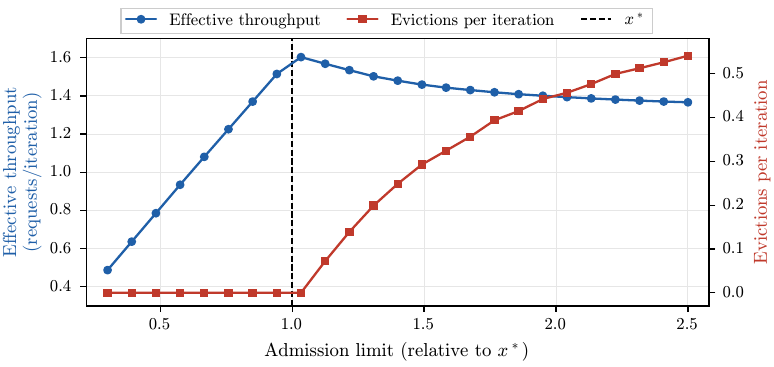}
    \caption{Throughput-eviction tradeoff in the model-based simulator.
    Throughput peaks near the eviction-free equilibrium \(x^*\) and declines
    beyond it as eviction overhead grows.}
    \label{fig:rate_limit_tradeoff_cpu}
\end{figure}

\begin{table}[htbp]
\centering
\caption{Rate-limited admission eliminates eviction while improving throughput (model-based simulator with $M=1000, l_0=20, l_1=20$).}
\label{tab:rate_limit_comparison}
\small
\begin{tabular}{@{}lccc@{}}
\toprule
\textbf{Metric} & \textbf{Baseline (Greedy)} & \textbf{Rate-Limited ($x^*$)} & \textbf{Improvement} \\
\midrule
Total Evictions (4000 steps) & 2662 & \textbf{0} & 100\% \\
Effective Throughput (req/step) & 1.33 & \textbf{1.61} & 20.7\% \\
Mean Eviction Rate & 0.666 & 0.000 & --- \\
Stability & Unstable (periodic eviction) & \textbf{Stable} (zero eviction) & --- \\
\bottomrule
\end{tabular}
\end{table}

\FloatBarrier

\subsection{Burst-window and burst-factor robustness}
\label{app:burst_robustness_details}

Figures~\ref{fig:window_length_ablation_cpu},
\ref{fig:vidur_robustness_coprime}, \ref{fig:vidur_robustness_partial},
\ref{fig:robustness_coprime}, and~\ref{fig:robustness_partial}
combine the finite-window ablation with the burst-factor
robustness sweeps. Mixing requires overlap among distinct request classes. Under
bursty arrivals, a short high-rate window may end before the scheduler sees
enough heterogeneous requests to form the intended mixed pool. Longer burst
windows move the realized class mix closer to the target proportions, increase
the probability of effective coprime mixing, and improve the performance of
mixed routing. The burst-factor results are consistent with the same mechanism: mixing
helps when the burst contains enough heterogeneity, while rate-limited
admission remains the direct hard cap against temporary overshoot.

\begin{figure}[ht]
    \centering
    \includegraphics[width=0.98\linewidth]{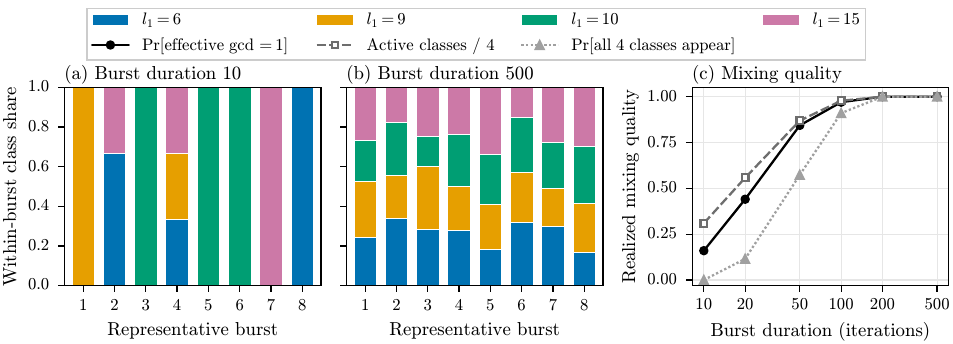}
    \caption{Window-length ablation for finite burst mixing in the model-based
    simulator. Short bursts may not contain enough active classes to realize the
    intended coprime pool; longer bursts do.}
    \label{fig:window_length_ablation_cpu}
\end{figure}

\begin{figure}[ht]
\centering
\includegraphics[width=0.98\linewidth]{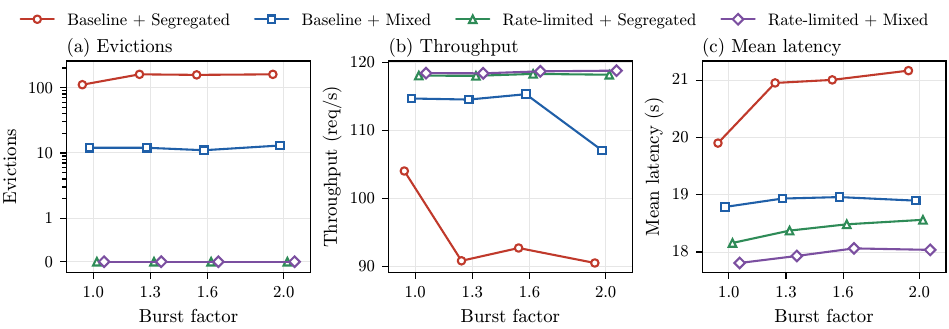}
\caption{Robustness to periodic burst arrivals: coprime mixing in Vidur. The
sweep compares baseline segregated routing, baseline mixed routing,
rate-limited segregated routing, and rate-limited mixed routing. Markers are offset horizontally for visibility.}
\label{fig:vidur_robustness_coprime}
\end{figure}

\begin{figure}[ht]
\centering
\includegraphics[width=0.98\linewidth]{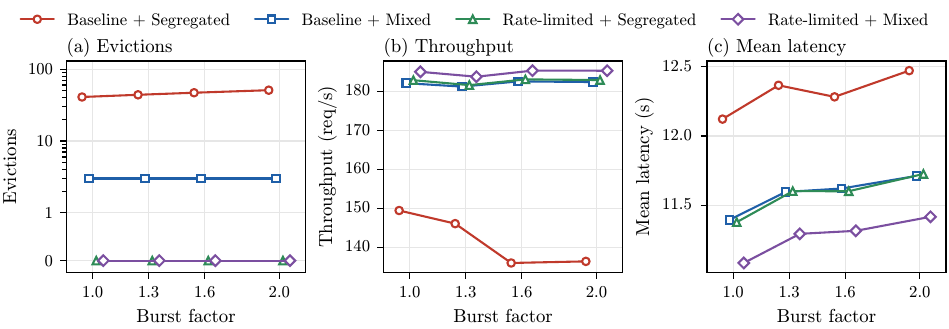}
\caption{Robustness to periodic burst arrivals: partial-GCD reduction in Vidur.
Mixed routing reduces the effective GCD to a value above one; rate limiting
remains the policy that eliminates evictions over this horizon. Markers are offset horizontally for visibility.}
\label{fig:vidur_robustness_partial}
\end{figure}

\begin{figure}[ht]
    \centering
    \includegraphics[width=0.98\linewidth]{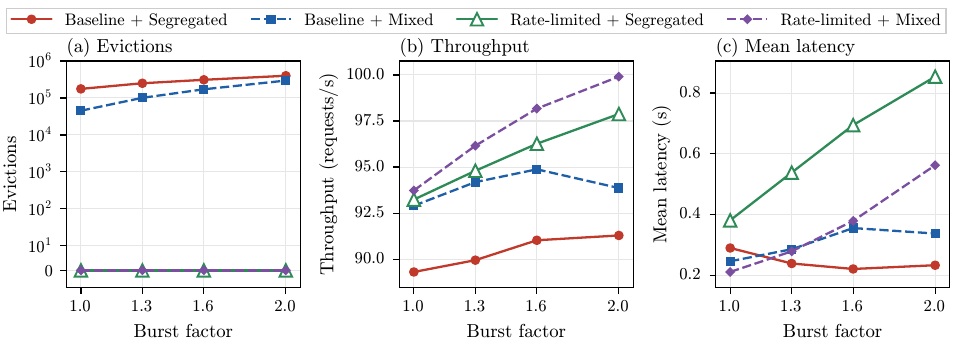}
\caption{Model-based robustness under periodic burst arrivals: coprime
    mixing. Mixing reduces eviction pressure under greedy admission, while
    rate-limited policies remain at zero eviction. Segregated values average the two nodes.}
    \label{fig:robustness_coprime}
\end{figure}

\begin{figure}[ht]
    \centering
    \includegraphics[width=0.98\linewidth]{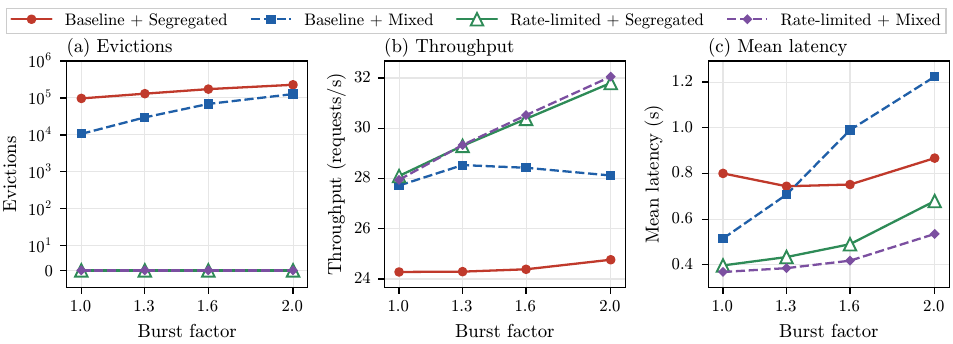}
    \caption{Model-based robustness under periodic burst arrivals: partial-GCD
    reduction. Partial mixing improves the greedy baseline, while rate limiting
    remains the hard stability mechanism. Segregated values average the two nodes.}
    \label{fig:robustness_partial}
\end{figure}

\FloatBarrier

\end{document}